\documentclass{amsart}

\usepackage[utf8]{inputenc}
\usepackage[T1]{fontenc}
\usepackage[hidelinks, unicode]{hyperref}
\usepackage{cleveref}
\usepackage{amsthm}
\usepackage{amssymb}
\usepackage{bm}
\usepackage{quiver}

\usepackage{import}

\usepackage{graphicx}

\usepackage{tikzit}

\tikzstyle{green}=[text={black!30!green}]
\tikzstyle{blue}=[text=blue]
\tikzstyle{red}=[text=red]
\tikzstyle{puncture}=[fill=white, draw=red, shape=circle, minimum size=1pt]
\tikzstyle{blackpuncture}=[fill=white, draw=black, shape=circle, minimum size=1pt]
\tikzstyle{cyan}=[text=cyan]

\tikzstyle{markings}=[-, draw=red, line width=1pt, line cap=round]
\tikzstyle{overbraid}=[-, draw=white, fill=none, line width=6pt]
\tikzstyle{thick}=[-, line width=2pt, draw=blue]
\tikzstyle{dashedline}=[-, dashed]
\tikzstyle{dottedline}=[-, dash pattern=on 0.75pt off 0.75pt, line width=0.75pt]
\tikzstyle{thin red}=[-, line width=0.25pt, draw=black]
\tikzstyle{tangle}=[-, draw=blue, line width=1pt, fill={blue!20}]
\tikzstyle{scc}=[-, draw={black!30!green}, fill={blue!20}, line width=1pt]
\tikzstyle{inner boundary}=[-, fill=white]
\tikzstyle{outer boundary}=[-, fill={red!20}]
\tikzstyle{lowerboundery}=[-, line width=1.5pt, line cap=round, draw=red]
\tikzstyle{upperboundery}=[-, line width=1.5pt, line cap=round, draw=blue]
\tikzstyle{dottedcycle}=[-, draw=blue, dash pattern=on 0.5pt off 1pt on 4pt off 1pt, decoration={markings, mark=at position 0.5 with {\arrow{>}}}, postaction=decorate]
\tikzstyle{cycle}=[-, draw=blue, decoration={markings, mark=at position 0.5 with {\arrow{>}}}, postaction=decorate]
\tikzstyle{path}=[-, draw=cyan, line width=0.25pt]
\tikzstyle{arrowpath}=[-, draw=cyan, line width=0.25pt, decoration={markings, mark=at position 0.5 with {\arrow{>}}}, postaction=decorate]
\tikzstyle{orientedpath}=[-, line width=0.25pt, decoration={markings, mark=at position 0.5 with {\arrow{<}}}, postaction=decorate]
\tikzstyle{inner square}=[-, fill={blue!20}]
\tikzstyle{outer square}=[-, fill={red!20}]
\tikzstyle{blueline}=[-, draw=blue]
\tikzstyle{greenline}=[-, draw=green]
\tikzstyle{bluesquare}=[-, draw=blue, fill={blue!20}]

\usetikzlibrary{decorations.markings}

\usepackage{caption}
\usepackage{xparse}

\numberwithin{figure}{section}

\newtheorem{theorem}{Theorem}[section]
\newtheorem*{theorem*}{Theorem}

\newtheorem{lemma}[theorem]{Lemma}
\newtheorem{proposition}[theorem]{Proposition}
\newtheorem*{proposition*}{Proposition}

\newtheorem{conjecture}[theorem]{Conjecture}

\newtheorem{corollary}[theorem]{Corollary}
\newtheorem*{corollary*}{Corollary}

\theoremstyle{remark}
\newtheorem{remark}[theorem]{Remark}
\newtheorem{notation}[theorem]{Notation}
\newtheorem*{question*}{Question}
\theoremstyle{definition}
\newtheorem{definition}[theorem]{Definition}
\newtheorem*{definition*}{Definition}

\newcommand{\Autb}{\mathrm{Aut}_\beta}

\newcommand{\Ass}{\mathrm{Ass}}

\newcommand{\Shift}{\mathrm{Shifts}(\Lambda)}
\newcommand{\Shiftz}{\mathrm{Shifts}^0(\Lambda)}

\newcommand{\CNt}{\mathrm{C}_{\Nu_t}}
\newcommand{\CNeps}{\mathrm{C}_{\Nu,\mathbb{C}[\epsilon]}}

\newcommand{\Cat}{\mathcal{C}}

\newcommand{\dol}{\overline{\partial}}

\newcommand{\Ceps}{{\mathbb{C}[\epsilon]}}

\newcommand{\DMC}{\mathrm{DM}^\mathrm{sm}}
\newcommand{\DMCp}{\mathrm{DM}^\mathrm{sm,proper}}

\newcommand{\Cinf}{C^\infty}

\newcommand{\cE}{\mathcal{E}}
\newcommand{\cF}{\mathcal{F}}
\newcommand{\cG}{\mathcal{G}}

\newcommand{\uE}{\underline{E}}
\newcommand{\uF}{\underline{F}}
\newcommand{\uG}{\underline{G}}

\newcommand{\rmd}{\mathrm{d}}

\newcommand{\Op}[2]{\mathrm{Op}_{#1}({#2})}

\newcommand{\VLR}{\mathrm{Vect}_\Lambda^R}
\newcommand{\VL}{\mathrm{Vect}_\Lambda}
\newcommand{\VLone}{\mathrm{Vect}_{\Lambda_1}}
\newcommand{\VLtwo}{\mathrm{Vect}_{\Lambda_2}}
\newcommand{\CN}{\mathrm{C}_\Nu}
\newcommand{\CNone}{\mathrm{C}_{\Nu_1}}
\newcommand{\CNtwo}{\mathrm{C}_{\Nu_2}}

\newcommand{\ch}{\mathrm{ch}}

\newcommand{\Zuv}{\Z[u,v]_{\mathrm{rat}}}
\newcommand{\Quv}{\Q[u,v]_{\mathrm{rat}}}

\newcommand{\Mod}{\mathfrak{mod}}
\newcommand{\Ribfun}{\mathfrak{ribfun}}
\newcommand{\Braidfun}{\mathfrak{braidfun}}
\newcommand{\Rib}{\mathfrak{Rib}}
\newcommand{\Braid}{\mathfrak{Braid}}
\newcommand{\Fus}{\mathfrak{Fus}}
\newcommand{\Modcat}{\mathfrak{Mod}}

\newcommand{\tMod}{\mathfrak{m\tilde{o}d}}

\newcommand{\SO}{\mathrm{SO}(3)}
\newcommand{\SU}{\mathrm{SU}(2)}

\newcommand{\scc}{simple closed curve}

\newcommand{\lra}{\longrightarrow}

\newcommand{\ra}{\rightarrow}

\newcommand{\Nu}{\mathcal{V}}
\newcommand{\Line}{\mathcal{L}}

\newcommand{\PMod}[1]{\mathrm{PMod}(#1)}

\newcommand{\tPMod}[1]{\mathrm{PM\widetilde{o}d}(#1)}

\newcommand{\Modl}[2]{\mathrm{Mod}^{#1}(#2)}

\protected\def\myphantom#1{\vphantom{#1}}
\newcommand{\myleftidx}[3]{{\myphantom{#2}}#1\!#2#3}
\newcommand{\mysecondleftidx}[3]{{\myphantom{#2}}#1#2#3}

\newcommand{\Mrpo}[2]{\mysecondleftidx{^1}{\overline{\mathcal{M}}}{_{0,{#1}}}(#2)}
\newcommand{\sMrpo}[2]{\myleftidx{^1}{\mathcal{M}}{_{0,{#1}}}(#2)}

\newcommand{\Mg}[2]{\mathcal{M}_{{#1},{#2}}}
\newcommand{\Mgb}[2]{\mathcal{M}_{#1}^{#2}}

\newcommand{\oMg}[2]{\overline{\mathcal{M}}_{{#1},{#2}}}

\newcommand{\Mgrp}[3]{\overline{\mathcal{M}}_{{#1},{#2}}(#3)}

\newcommand{\Cgrp}[3]{\overline{\mathcal{C}}_{{#1},{#2}}(#3)}

\newcommand{\Mgrb}[3]{\overline{\mathcal{M}}_{#1}^{#2}(#3)}

\newcommand{\Mgrbt}[4]{\overline{\mathcal{M}}_{#1}^{#2}(#3,#4)}

\newcommand{\Surf}{\mathrm{Surf}_\Lambda}
\newcommand{\tSurf}{\mathrm{S\tilde{u}rf}_\Lambda}

\newcommand{\Nuab}{\Nu^{\mathrm{ab}}_{\zeta_r}}
\newcommand{\Nusl}{\Nu_{\zeta_r}}

\newcommand{\justM}{\mathcal{M}}
\newcommand{\tjustM}{\widetilde{\mathcal{M}}}
\newcommand{\justC}{\mathcal{C}}

\newcommand{\Br}[1]{\mathrm{B}\mu_{#1}}

\newcommand{\End}[1]{\mathrm{End}(#1)}
\newcommand{\GL}[1]{\mathrm{GL}(#1)}

\newcommand{\GLn}[2]{\mathrm{GL}_{#1}(#2)}

\newcommand{\tr}[1]{\mathrm{tr}(#1)}

\newcommand{\Aut}[1]{\mathrm{Aut}(#1)}

\newcommand{\Z}{\mathbb{Z}}
\newcommand{\N}{\mathbb{N}}

\newcommand{\C}{\mathbb{C}}
\newcommand{\Q}{\mathbb{Q}}
\newcommand{\R}{\mathbb{R}}

\newcommand{\id}{\mathrm{id}}

\newcommand{\Dc}{\mathcal{D}}

\newcommand{\ul}{\underline{\lambda}}
\newcommand{\umu}{\underline{\mu}}

\newcommand{\Disk}{\mathrm{Disk}_\Lambda}

\title{Hodge structures on conformal blocks}
\author{Pierre Godfard}

\begin{document}

\begin{abstract}
  We prove existence and uniqueness of complex Hodge structures on modular functors.
  The proof is based on the non-Abelian Hodge correspondence and Ocneanu rigidity.
  Given a modular functor, we explain how its Hodge numbers fit into a Frobenius algebra and the Chern characters of its Hodge decompositions into a new
  cohomological field theory (CohFT).
  In the case of $\SU$ modular functors of level $2$ times an odd number,
  we give explicit formulas for all Hodge numbers, in any genus $g$.
\end{abstract}

\maketitle

\section{Introduction}

This article is the second in a series of two which aims at constructing and studying Hodge structures on the flat vector bundles (spaces of conformal blocks) associated
to modular functors. In this article, which can be read independently from the first one, we prove
existence and uniqueness of these Hodge structures, and compute Hodge numbers of $\SU$ modular functors.

This computation uses an explicit construction of the Hodge structures in genus $0$, which was the subject of the first paper
\cite{godfardConstructionHodgeStructures2024},
and comes from a geometric interpretation of the fibers of
bundles of conformal blocks as cohomology of a configuration space in $\mathbb{P}^1$ minus a finite number of points with coefficients in a local system.
The existence and uniqueness result of this second paper is used
to extend our computation of Hodge numbers from genus $0$ to any genus $g$ without the need to supply
geometric constructions for $g>0$.


\subsection{Modular functors}

Modular functors are families of representations of mapping class groups with compatibility conditions.
More precisely, a modular functor $\Nu$ with finite set of colors $\Lambda$ gives for every choice of $g,n\geq 0$ and $\lambda_1,\dotsc,\lambda_n\in\Lambda$,
a representation:
\begin{equation*}
  \rho_g(\ul):\tPMod{S_g^n}\lra \GLn{d}{\C}
\end{equation*}
where $\tPMod{S_g^n}$ is a central extension by $\Z$ of the mapping class group $\PMod{S_g^n}$ of the compact surface of genus $g$ with $n$ boundary components.
The dimension $d$ depends on $g$, $n$, and $\ul$.

For any modular functor there exists a pair of integers $r,s\geq 1$ such that
the $s$-th power of the central element and all $r$-th powers of Dehn twists act trivially in every $\rho_g(\ul)$.
The quotient of $\tPMod{S_g^n}$ by these elements is the fundamental group of a $\mu_{s}$-gerbe $\Mgrbt{g}{n}{r}{s}$ over
the compact moduli space $\Mgrb{g}{n}{r}$ of genus $g$ $r$-twisted nodal curves with $n$ marked orbifold points of order $r$
(see \Cref{remarkmgnr,remarkmgnrs}). The coarse space of $\Mgrbt{g}{n}{r}{s}$ is just the Knudsen-Deligne-Mumford moduli space of nodal curves $\overline{M}_{g,n}$.
By the correspondence between representations of the fundamental group and bundles with flat connections,
we can alternatively see each $\rho_g(\ul)$ as a bundle with flat connection called \textit{conformal block}:
\begin{equation*}
  (\Nu_g(\ul),\nabla)\lra \Mgrbt{g}{n}{r}{s}.
\end{equation*}
These bundles satisfy strong compatibility conditions, the most important of which are the gluing axioms, that we now describe.
As for the moduli spaces of nodal curves $\oMg{g}{n}$,
the orbifolds $\Mgrbt{g}{n}{r}{s}$ have gluing maps:
\begin{align*}
  q:\Mgrbt{g_1}{n_1+1}{r}{s}\times\Mgrbt{g_2}{n_2+1}{r}{s}&\lra \Mgrbt{g_1+g_2}{n_1+n_2}{r}{s}, \\
  p:\Mgrbt{g-1}{n+2}{r}{s}&\lra \Mgrbt{g}{n}{r}{s}
\end{align*}
whose images are boundary divisors.
The map $q$ is induced by the gluing of the $(n_1+1)$-th section of a curve in $\Mgrbt{g_1}{n_1+1}{r}{s}$ to the $(n_2+1)$-th
section of a curve in $\Mgrbt{g_2}{n_2+1}{r}{s}$ (see \Cref{separatinggluing}).
Similarly, $p$ is induced by the gluing of the $(n+1)$-th and $(n+2)$-th sections of a curve in $\Mgrbt{g-1}{n+2}{r}{s}$ (see \Cref{nonseparatinggluing}).

Then the gluing axioms say that for every such maps $q$ and $p$, and choice of colors $\ul\in\Lambda^n$, isomorphisms as follows are given:
\begin{align*}
  q^*\Nu_{g_1+g_2}(\lambda_1,\dotsc,\lambda_n)&\simeq \bigoplus_\mu\Nu_{g_1}(\lambda_1,\dotsc,\lambda_{n_1},\mu)\otimes
  \Nu_0(\mu,\mu^\dagger)\otimes
  \Nu_{g_2}(\lambda_{n_1+1},\dotsc,\lambda_n,\mu^\dagger), \\
  p^*\Nu_g(\lambda_1,\dotsc,\lambda_n)&\simeq \bigoplus_\mu\Nu_{g-1}(\lambda_1,\dotsc,\lambda_n,\mu,\mu^\dagger)
  \otimes \Nu_0(\mu,\mu^\dagger).
\end{align*}
The map $\Lambda\ra\Lambda$, $\lambda\mapsto\lambda^\dagger$ used above is an involution which is part of the data of the modular functor.

\subsection{Hodge structures}

A complex variation of Hodge structures (CVHS) on a bundle with connection $(E,\nabla)$ over a compact complex manifold
is the data of a rational number $m$ called weight together with a $\Cinf$ decomposition of $E$:
\begin{equation*}
  E = \bigoplus_{\substack{p+q=m\\p,q\in\Q}}E^{p,q}
\end{equation*}
satisfying Griffiths transversality and polarizability (see \Cref{definitionCVHS}). For why we allow $m$, $p$ and $q$ to be rational, see \Cref{remarkrationalweight}.

Note that we can naturally take duals and tensor products of such structures, using the formulas
$(E^\vee)^{p,q}=(E^{-p,-q})^\vee$ and $(E\otimes E')^{p,q}=\bigoplus_{\substack{p_1+p_2=p\\q_1+q_2=q}}E^{p_1,q_1}\otimes E'^{p_2,q_2}$.

\begin{definition*}[{Hodge structure, (\ref{definitionCVHSmodular})}]
  Let $\Nu$ be a complex modular functor. A complex Hodge structure on $\Nu$ is the data, for each $g$, $n$ and $\ul\in\Lambda^n$,
  of a CVHS on $\Nu_g(\lambda_1,\dotsc,\lambda_n)$,
  such that all gluing, vacuum and permutation isomorphisms in \Cref{definitiongeometricnonanomalousmodularfunctor} are isomorphisms of CVHS.
\end{definition*}

The existence of Hodge structures on each $\Nu_g(\lambda_1,\dotsc,\lambda_n)$ is expected for several reasons.
First, in the special case of genus $0$ and when $\Nu$ comes from a simple Lie algebra, 
the local systems $\Nu_0(\lambda_1,\dotsc,\lambda_n)$ are known to be motivic and thus support CVHS (see \Cref{sectionsu} below).

For $g\geq 3$, it is expected that the mapping class group $\PMod{S_g}$ satisfies Kazhdan’s property $(T)$.
In the case where the local system $\Nu_g$ has a unitary Galois conjugate, this would imply its rigidity and thus, by Simpson's result
\cite[4.5]{simpsonHiggsBundlesLocal1992} reproduced below, existence of a CVHS on $\Nu_g$. Note that if $\Nu$ comes from a simple Lie algebra,
$\Nu_g$ has a unitary Galois conjugate.

When $\Nu$ is the $\SO$ modular functor of level $5$, it was proved by the author in \cite{godfardRigidityFibonacciRepresentations2023} that
each $\Nu_g(\lambda_1,\dotsc,\lambda_n)$ is rigid for any $g$ and $\lambda_1,\dotsc,\lambda_n$. Hence, again by Simpson's result,
this implies existence of a CVHS on each $\Nu_g(\lambda_1,\dotsc,\lambda_n)$. Note that this does not a priori imply
existence of a CVHS on $\Nu$ in the sense of the above definition, as compatibility with gluing, vacuum and permutations isomorphisms is not immediate.

In \cite[6.]{deroinToledoInvariantsTopological2022}, Deroin and Marché use the $\SO$ modular functor of level $5$
to endow the orbifolds $\overline{\mathcal{M}}^\mathcal{E}_{g,n}$ for $(g,n)\in \{(0,4), (0,5), (1,2), (1,3), (2,1)\}$
with a complex hyperbolic structure whose holonomy is the monodromy of $\Nu_{g,n}(2,\dotsc,2)$. Here $\mathcal{E}$ means that elliptic tales have been contracted.
To do so, they use the invariant Hermitian form on $\Nu_{g,n}(2,\dotsc,2)$, of signature $(1,q)$ in these cases, and a result of Siu to prove existence
of a local biholomorphism from the universal cover of $\overline{\mathcal{M}}^\mathcal{E}_{g,n}$ to $\mathbb{H}^{1,q}$.
In light of the result of \cite{godfardRigidityFibonacciRepresentations2023},
this map can actually be chosen to be the period map of the Hodge structure on $\Nu_{g,n}(2,\dotsc,2)$.
It would be interesting to further study period maps of Hodge structures on modular functors. This article provides a proof of their existence.

A potential application of Hodge structures on $\Nu_g(\lambda_1,\dotsc,\lambda_n)$ is detecting irreducibility.
Indeed, any gap in the Hodge decomposition yields a direct sum decomposition. Such gaps could be found by computing Hodge numbers.
However, in \Cref{subsectiongaps}, we prove that for all $\SO$ and $\SU$ modular functors, the Hodge structures on
the $\Nu_g(\lambda_1,\dotsc,\lambda_n)$ have no gaps. It would be interesting to study gaps in Hodge numbers
more generally for modular functors coming from simple Lie algebras (see \Cref{subsectionlie} below).


\subsection{Existence and uniqueness}

A large class of flat bundles that support complex variations of Hodge structures
is supplied by bundles of geometric origin, where a flat bundle $E$ on a variety $B$ is called of geometric origin if on
some dense Zariski open subset $U\subseteq B$,
there exists a proper fibration $p : X \ra U$ and $m \geq 0$ an integer such that $E$ is a flat subquotient of $R^mp_*\underline{\C}$.

According to the standard conjectures, rigid flat bundles on projective smooth varieties are expected to be of geometric origin,
and thus in particular to support CVHS.
In fact, Simpson proved the latter using non-Abelian Hodge theory:

\begin{theorem*}[{\cite[4.5]{simpsonHiggsBundlesLocal1992}}]
  Any rigid flat bundle $(E,\nabla)$ on a compact Kähler manifold supports a CVHS. If moreover $(E,\nabla)$ is simple,
  then the CVHS is unique up to shifts.
\end{theorem*}
Here, by a shift of $E = \bigoplus_{p+q=m}E^{p,q}$, we mean another decomposition
$E = \bigoplus_{p+q=m}E'^{p,q}$, where $E'^{p,q}=E^{p+s,q+t}$ for a pair $(s,t)\in\Q^2$ independent of $(p,q)$.

Our main result is an analog of Simpson's in the case of modular functors. Although the flat bundles $\Nu_g(\ul)$ of a modular functor $\Nu$
may not in general individually be rigid, they do not admit compatible families of deformations (Ocneanu rigidity).
Using this fact and non-Abelian Hodge theory we are able to prove that
Hodge structures on modular functors exist and are unique up to shifts.

\begin{theorem*}[{Existence and uniqueness, (\ref{maintheoremcomplex})}]
  Let $\Nu$ be a complex modular, genus $0$ modular, ribbon or braided functor.
  Then $\Nu$ admits a complex Hodge structure.
  Moreover, $\Shift=\{\delta:\Lambda\ra\Q^2\mid \delta(0)=0\}$ acts transitively on the set of complex Hodge structures on $\Nu$,
  where $\Lambda$ denotes the set of colors of $\Nu$.
\end{theorem*}

\begin{corollary*}[{RVHS on conformal blocks, (\ref{rationalqreps})}]
  Let $\Nu$ be as in the Theorem.
  Then for any $g,n\geq 0$ tangent-stable and $\ul\in\Lambda^n$,
  the local system $\Nu_g(\ul)$ over $\Mgrbt{g}{n}{r}{s}$ can be defined over a CM number field $L$
  and supports a rational variation of Hodge structures over $L$,
  which can be assumed to have weight $m$ in $\Z$, Hodge bidegrees in $\Z^2$, and be strongly polarized over $L$ (see \Cref{definitionRVHS}).
\end{corollary*}

The shift $\delta$ acts on the CVHS on $\Nu_g(\ul)$ by shifting its bidegrees by $\sum_i\delta(\lambda_i)$.
Note that the notions of modular functors and modular fusion categories are equivalent, thanks to work of Bakalov-Kirillov \cite{bakalovLecturesTensorCategories2000}
and a recent paper of Etingof-Penneys \cite{etingofRigidityNonnegligibleObjects2024}.
Similarly, the notions of genus $0$ modular functor, ribbon functor and ribbon fusion category are equivalent,
and those of braided functor and braided fusion category are equivalent.
Hence the results above apply to flat bundles associated to modular, ribbon or braided fusion categories.

\begin{remark}\label{remarksemisimplicity}
  A flat bundle underlying a CVHS is always semisimple. Hence the statement of \Cref{maintheoremcomplex}
  \emph{contains the semisimplicity of conformal blocks}.
  However, as the proof of semisimplicity relies on results in non-Abelian Hodge theory (formality) which are different from
  those used in this paper, we have chosen to write it separately in the companion paper \cite{godfardSemisimplicityConformalBlocks2025}.
  Here are the explicit dependencies between the papers:
  \cite{godfardSemisimplicityConformalBlocks2025} uses results from \Cref{sectioncn,sectionocneanu};
  the proof of \Cref{maintheoremcomplex} given in
  \Cref{sectionproof} relies on \cite{godfardSemisimplicityConformalBlocks2025} and \Cref{sectioncn,sectionnonabelian,sectionocneanu}.
\end{remark}

The main motivation for proving the above Theorem is the following.

\begin{corollary*}[{\ref{uniqueextension}}]
  Let $\Nu$ be a complex modular functor.
  Then any complex Hodge structure on the braided functor
  \footnote{which is a genus $0$ structure!}
  associated to $\Nu$
  extends uniquely to the whole modular functor. In particular, knowing genus $0$ Hodge numbers is knowing all Hodge numbers.
\end{corollary*}

In \Cref{subfrobenius}, we explain how all the Hodge numbers of a Hodge structure on a modular functor $\Nu$ with set of colors $\Lambda$
can be packaged into one commutative associative product $\star$ on the free $\Zuv$-module $V$ with basis $\Lambda$,
where $\Zuv$ is the ring of polynomials in rational powers of the $2$ variables $u$ and $v$, and with coefficients in $\Z$.
Below we give a description of $\star$ and how it relates to Hodge numbers in the case of $\SU$ modular functors.

The idea that Hodge numbers of braided functors could be organized in an algebra appeared independently in the recent paper \cite{belkaleMotivicFactorisationKZ2023}
of Belkale, Fakhruddin and Mukhopadhyay. See \Cref{subintroBFM,subsectionlie} for more comments on their paper.

Moreover, given a Hodge structure on a modular functor $\Nu$, the characteristic classes of the CVHS on the $\Nu_g(\ul)$
can be packaged into a cohomological field theory (CohFT) $\omega$, such that the Frobenius algebra associated to the topological part of $\omega$
is $V$. Hence to any modular functor $\Nu$ we associate a CohFT over the ring $\Quv$, unique up to the action of $\Shift$.
If $\Nu$ is Hermitian, under some assumptions on the Hermitian form and the Hodge structure,
this CohFT is related to the one constructed by Deroin and Marché
in \cite[4.3]{deroinToledoInvariantsTopological2022}, see \Cref{DeroinMarche}.
We do not expect this CohFT to be related to the one constructed differently from $\Nu$
by Andersen, Borot and Orantin in \cite{andersenModularFunctorsCohomological2016}
(see also the related papers by Marian, Oprea, Pandharipande and Zvonkine 
\cite{marianChernCharacterVerlinde2017} and by Damiolini, Gibney and Tarasca \cite{damioliniVertexAlgebrasCohFTtype2022}).


\subsection{Outline of the proof}

We sketch the proof of \Cref{maintheoremcomplex} (existence and uniqueness of CVHS) in the case of modular functors.
The proofs for genus $0$ modular functors, ribbon functors and braided functors are similar.

Let $X$ be a smooth projective manifold, or more generally a proper smooth DM stack over $\C$. Choose a $\Cinf$ bundle $E$ on $X$.
Then the Simpson correspondence provides a bijection between semisimple flat connections $\nabla$ on $X$
and Higgs bundles of harmonic type on $E$, where a Higgs bundle on $E$ is a pair $(\dol_E,\theta)$,
with $\dol_E:E\ra E\otimes \Omega^{0,1}_X$ a holomorphic structure on $E$
and $\theta:E\ra E\otimes \Omega^{1,0}_X$ a holomorphic, $\mathcal{O}_X$-linear map satisfying $\theta\wedge\theta=0$.
This bijection is compatible with morphisms, tensor products, duals, and pullbacks along algebraic maps $Y\ra X$.

Given a semisimple flat bundle $(E,\nabla)$ on $X$ to which the correspondence associates the Higgs bundle $(E,\dol_E,\theta)$,
for each real number $t\in \R_+^*\setminus\{1\}$ and $m\in\Z$, we have maps:
\[\begin{tikzcd}
  {\{\text{isomorphisms }f:(E,\dol_E,\theta)\ra(E,\dol_E,t\theta)\}} \\
  {\{\text{structures of CVHS on }(E,\nabla)\text{ with weight }m\}}
  \arrow["r"', curve={height=6pt}, two heads, from=1-1, to=2-1]
  \arrow["i"', curve={height=6pt}, hook, from=2-1, to=1-1]
\end{tikzcd}\]
such that $r\circ i=\id$ (see \Cref{actiontoCVHS}).
Hence if $(E,\nabla)$ is rigid, so is $(E,\dol_E,\theta)$, and thus the continuous deformation $(E,\dol_E,t\theta)$ is isomorphic to $(E,\dol_E,\theta)$.
Applying $r$ to an isomorphism we get a CVHS structure on $(E,\nabla)$. This is how Simpson proves that rigid local systems admit a CVHS structure.

We want to apply the Simpson correspondence to a modular functor $\Nu$.
To do so we need to make sure that each flat bundle $\Nu_g(\ul)$ is defined on a proper smooth DM stack and that gluing maps are algebraic.
This is not the case with most definitions of modular functors found in the literature: usually either $\Nu_g(\ul)$ is defined on a toric fibration over $\Mg{g}{n}$,
which is not proper and provides non-algebraic gluing maps,
or on $\oMg{g}{n}$ with algebraic gluing maps, but then $\Nu_g(\ul)$ is not flat.
That is why we choose to define $\Nu_g(\ul)$ on the DM stack $\Mgrbt{g}{n}{r}{s}$, which is proper, smooth and has algebraic gluing maps, as explained above.
See \Cref{subgeometricdefinition} for more details.

Applying the Simpson correspondence to modular functors, we get an equivalence of categories between the category of modular functors
with semisimple conformal blocks $(\Nu,\nabla)$ and that of modular functors of Higgs bundles with conformal blocks of harmonic type $(\Nu,\theta)$
(see \Cref{correspondencemodular}).
Here by $(\Nu,\nabla)$ and $(\Nu,\theta)$ we mean families $(\Nu_g(\ul),\nabla^{g,\ul})$ and $(\Nu_g(\ul),\theta^{g,\ul})$ with gluing, vacuum and permutation isomorphisms.

Now, as above, given a modular functor $(\Nu,\nabla)$ to which the correspondence associates the
modular functor of Higgs bundle $(\Nu,\theta)$,
for each real number $t\in \R_+^*\setminus\{1\}$ and weight function $w$ (see \Cref{definitionweightfunction}), we have maps:
\[\begin{tikzcd}
  {\{\text{isomorphisms }f:(\Nu,\theta)\ra(\Nu,t\theta)\}} \\
  {\{\text{structures of CVHS on }(\Nu,\nabla)\text{ with weight function }w\}}
  \arrow["r"', curve={height=6pt}, two heads, from=1-1, to=2-1]
  \arrow["i"', curve={height=6pt}, hook, from=2-1, to=1-1]
\end{tikzcd}\]
such that $r\circ i=\id$ (see \Cref{actiontoCVHSmodular}).

To prove existence of a CVHS structure on $(\Nu,\nabla)$, we apply a rigidity result on modular functors that will imply that $(\Nu,\theta)$ and $(\Nu,t\theta)$ must be isomorphic.
This result is Ocneanu rigidity for ribbon fusion categories, that we explain now.

To a modular functor $\Nu$, one can associate a ribbon fusion category $\CN$ in a functorial and continuous way. In \Cref{sectioncn},
we prove that this association is a fully faithful functor $\Mod\ra\Rib$, using results of Bakalov-Kirillov and Andersen-Ueno.

Ocneanu rigidity implies that a ribbon fusion category has no nontrivial continuous deformation (\Cref{theoremocneanuribbon}).
Let $(\Nu,\nabla)$ be a modular functor with semisimple conformal blocks such that $\CN$ is ribbon.
Using the continuity of both the map $\Nu\mapsto \CN$ and the Simpson correspondence, we have that for any $t\in \R_+^*\setminus\{1\}$,
the category $\mathrm{C}_{\Nu_t}$ associated to $(\Nu,t\theta)$ is a continuous deformation of $\CN$. Hence $\mathrm{C}_{\Nu_t}$ is isomorphic to $\CN$.
By full faithfulness of $\Mod\ra\Rib$ and the Simpson correspondence, $(\Nu,\theta)$ and $(\Nu,t\theta)$ are isomorphic.
Then any choice of isomorphism $f:(\Nu,\theta)\ra(\Nu,t\theta)$ provides a CVHS structure $r(f)$ on $(\Nu,\nabla)$.
This proves existence.

Uniqueness is a consequence of Ocneanu rigidity for endofunctors,
which says that any continuous deformation of the identity functor of a ribbon fusion category is given by a very simple formula
from some map $\delta:\Lambda\ra \C^*$ (see \Cref{theoremocneanufunctor}). Here $\Lambda$ is the set of simple objects in the category, and equals the set of colors of $\Nu$ when
the category is $\CN$.

Let $\Nu^1$ and $\Nu^2$ be $2$ Hodge structures on $\Nu$. 
Using the map $i$ above, for each $t\in \R_+^*$ the Hodge structures induce isomorphisms $f^1_t$ and $f^2_t$
from $(\Nu,\theta)$ to $(\Nu,t\theta)$.
Then the maps $(f^2_t)^{-1}\circ (f^1_t)$ form a continuous family of automorphisms of $(\Nu,\theta)$, which induce a
continuous family of automorphisms $k_t$ of $(\Nu,\nabla)$ and also of the ribbon fusion category $\CN$. From Ocneanu rigidity for endofunctors we then deduce
that there exists $\delta_t:\Lambda\ra \C^*$ such that $k_t$ is on $\Nu_g(\ul)$ multiplication by $\prod_i\delta_t(\lambda_i)$.
Using $r\circ i=\id$ and the definition of $r$ (see \Cref{actiontoCVHS}), we have  $\Nu^1=(\pi\circ\log_t|\delta_t|)\cdot \Nu^2$ where 
$(\pi\circ\log_t|\delta_t|)$ is in $\Shift$ and "$\cdot$" is the action of $\Shift$ on the set of complex Hodge structures.
This proves uniqueness, i.e. that $\Shift$ acts transitively on the set of complex Hodge structures on $\Nu$.


\subsection{\texorpdfstring{Hodge numbers of $\SU$ modular functors}{Hodge numbers of SU(2) modular functors}}\label{subintronumbers}

In the first article \cite{godfardConstructionHodgeStructures2024},
we described a geometric construction of Hodge structures on the genus $0$ part of $\SU$ modular functors of level $2$ times an odd number
\footnote{In genus $0$, these coincide with $\SO$ modular functors.}.
Here we explain how to deduce all Hodge numbers of these $\SU$ modular functors in all genera from the construction and \Cref{uniqueextension}
(reproduced above).

For each $r\geq 3$ odd and $0<s<r$ prime to $r$, the $\SU$ modular functor of level $2r$ with root $\zeta_r=-e^{i\pi s/r}$
has set of colors $\Lambda_r=\{0,1,\dotsc,r-2\}$
and trivial involution $(\cdot)^\dagger$.
It has a geometric construction, based on models of Felder-Wieczerkowski and Martel,
coming from an interpretation of the fibers of the bundles of conformal blocks
as cohomology of a configuration space in $\mathbb{P}^1$ minus a finite number of points with coefficients in a local system.
The results of \cite{godfardConstructionHodgeStructures2024} and \Cref{uniqueextension} in particular imply the following.

\begin{corollary*}[{\ref{corollaryhodgesu}}]
  Let $r\geq 3$ be odd and $0<s<r$ prime to $r$, then the $\SU$ modular functor $\Nu$ of level $2r$ with root $\zeta_r=-e^{i\pi s/r}$
  supports a complex Hodge structure with integer 
  Hodge bidegrees. Moreover the CVHS on $\Nu_g(\ul)$ has weight\footnote{When $\sum_i\lambda_i$ is odd, $\Nu_g(\ul)$ is $0$.} $m=\sum_i\lambda_i/2$.
\end{corollary*}

The $\SU$ modular functors $\Nu$ of level $2r$ also have Hermitian structures, i.e. families of Hermitian forms $h_g(\ul)$ on $\Nu_g(\ul)$
that are compatible with gluing isomorphisms and other axioms.
In \cite{godfardConstructionHodgeStructures2024}, we proved that each form $h_0(\ul)$ polarizes the CVHS $\Nu_0(\ul)=\bigoplus_{p+q=m}E^{p,q}$,
and computed explicitly its signs on each $E^{p,q}$. In \Cref{sectionsu}, we explain how from this sign computation,
formulas for the signatures of the forms $h_0(\ul)$ and associativity of the product $\star$ on $V$, one can obtain the following
formulas for Hodge numbers.

\begin{theorem*}[{\ref{theoremhodgenumberssu}}]
  Let $r\geq 3$ be odd and $0<s<r$ be odd and prime to $r$.
  Consider the $\SU$ modular functor $\Nu$ of level $2r$ with root $\zeta_r=-e^{i\pi s/r}$.
  Define, for $1\leq k\leq r-2$:
  \begin{equation*}
    w_k = \begin{cases} u &\text{if }\lfloor\frac{2s}{r}\rfloor+\lfloor\frac{ks}{r}\rfloor+\lfloor\frac{(k+1)s}{r}\rfloor\text{ is even,}\\
       v&\text{otherwise.}\end{cases}
  \end{equation*}
  Let $M$ be the following square matrix of size $r-1$ with coefficients in $\Z[u,v]$:
  \begin{equation*}
    M = \begin{pmatrix}
      0 & w_1     &           &           &     \\
      1 & 0       & w_2       &           &     \\
        & 1       & 0         & \ddots    &      \\
        &         & \ddots    & \ddots    & w_{r-2}   \\
        &         &           & 1         & 0
    \end{pmatrix}
  \end{equation*}
  and for $0\leq k\leq r-2$, let $M_k$ be its square top left submatrix of size $k$.
  Define a ring $V$ and elements $[k]\in V$ for $0\leq k\leq r-2$ by:
  \begin{equation*}
    V = \Z[u^{\pm 1},v^{\pm 1}][X]/\chi_M(X),\text{ and }[k]=\chi_{M_k}(X)
  \end{equation*}
  where $\chi_M(X)=\mathrm{det}(XI_{r-1}-M)$ denotes the characteristic polynomial of $M$.
  Let $\epsilon:V\ra \Z[u^{\pm 1},v^{\pm 1}]$ be the linear map sending $[0]$ to $1$ and $[k]$ to $0$ for $k>0$.
  Define $\Omega=\sum_k\epsilon([k]^2)^{-1}[k]^2\in V$.

  Then for $g\geq 0$ and $\lambda_1,\dotsc,\lambda_n\in \{0,1,\dotsc,r-2\}$,
  the Hodge numbers of the decomposition $\Nu_g(\ul)=\bigoplus_{p+q=m}E^{p,q}$
  are computed by:
  \begin{equation*}
    \sum_{p+q=m}\dim (E^{p,q})u^pv^q = \epsilon([\lambda_1]\dotsb[\lambda_n]\cdot\Omega^g).
  \end{equation*}
  Moreover, specializing at $(u,v)=(-1,1)$ if $\lfloor\frac{2s}{r}\rfloor$ is even or $(u,v)=(1,-1)$ if $\lfloor\frac{2s}{r}\rfloor$ is odd
  gives the Frobenius algebra computing signatures (see \cite[3]{marcheSignaturesTQFTsTrace2023}).
\end{theorem*}

This description is an extension of that given by Marché in \cite[3]{marcheSignaturesTQFTsTrace2023} for signatures of $\SU$ modular functors.
In their recent paper \cite{belkaleMotivicFactorisationKZ2023}, Belkale, Fakhruddin and Mukhopadhyay independently computed the same rings $V$
by computing coefficients $a_k$ and $b_k$ in the proof of \Cref{theoremhodgenumberssu} with a different method (see \cite[Th. 1.11]{belkaleMotivicFactorisationKZ2023}).
See \Cref{subintroBFM} below for more comments on their work.

Let us finish by giving an example of Hodge numbers for $r=5$. Let $\Nu$ be the $\SU$ modular functor of level $5$ for the root $\zeta_5=-e^{i\pi 3/5}$.
Using the theorem above, one can check that in this case the subalgebra generated by $Y=[2]$ is
$\Z[u^{\pm 1},v^{\pm 1}][Y]/(Y^2-vY-uv)$.
Now $\Nu_0(2,\dotsc,2)$ ($n$ times the color $2$) has dimension $F_{n-1}$ the $(n-1)$-th Fibonacci number ($F_0=0$, $F_1=1$).
The weight of the Hodge decomposition is $m=\frac{2n}{2}=n$ and if we denote it by $\Nu_0(2,\dotsc,2)=\bigoplus_{p+q=n}E^{p,q}$,
we have:
\begin{equation*}
  \dim E^{p,q}=\binom{q-1}{p-1}.
\end{equation*}
These numbers can be arranged into a Pascal triangle as in \Cref{pascal}.
The antidiagonals contain the Hodge decompositions of the $\Nu_0(2,\dotsc,2)$ for different values of $n$.
For fixed $n$, the sum of the antidiagonal $\sum_{p+q=n}\binom{q-1}{p-1}=F_{n-1}$ represents the dimension of $\Nu_0(2,\dotsc,2)$,
while the alternate sum $\sum_{p+q=n}(-1)^q\binom{q-1}{p-1}$ is the signature of the Hermitian form $h_0(2,\dotsc,2)$.

\begin{figure}
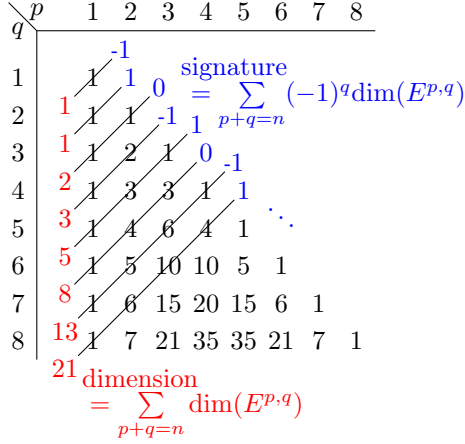

  \ctikzfig{pascal}
  \caption{Hodge numbers of $\Nu_0(2,\dotsc,2)$ for $r=5$ and $\zeta_5=-e^{i\pi 3/5}$, arranged inside a Pascal triangle.}
  \label[figure]{pascal}
\end{figure}


\subsection{\texorpdfstring{The case of modular functors associated to $\mathfrak{sl}_n$}{The case of modular functors associated to sln}}\label{subintroBFM}

In their work \cite{belkaleMotivicFactorisationKZ2023}, Belkale, Fakhruddin and Mukhopadhyay study Hodge structures
on some geometrically constructed braided functors $\widetilde{\Nu}_{\mathfrak{g},\zeta}$ associated to a complex finite dimensional simple Lie algebra $\mathfrak{g}$.
The conformal blocks of the functor $\widetilde{\Nu}_{\mathfrak{g},\zeta}$ coincide, where they are defined (i.e. in genus $0$), with those
of the usual modular functor $\Nu_{\mathfrak{g},\zeta}$ associated to $\mathfrak{g}$ (see \cite[chp. 7]{bakalovLecturesTensorCategories2000} for a definition).
Moreover, in \cite[11.1]{belkaleMotivicFactorisationKZ2023}, they explain how to compute the algebra organizing Hodge numbers of $\widetilde{\Nu}_{\mathfrak{sl}_n,\zeta}$
for every $n\geq 2$.

If one is able to prove that actually the braided functors of $\Nu_{\mathfrak{g},\zeta}$ coincides with $\widetilde{\Nu}_{\mathfrak{g},\zeta}$ (\Cref{conjectureBFM}),
then using the main result of this paper,
the computations of Belkale, Fakhruddin and Mukhopadhyay will extend painlessly to higher genus, without the need for a geometric model (\Cref{propositionBFM}).
See \Cref{subsectionlie} for more details.


\subsection{Organization of the paper}

In \Cref{sectionmodular}, we introduce the equivalent notions of topological and geometric modular functors. We follow Bakalov and Kirillov
\cite{bakalovLecturesTensorCategories2000}, with the main difference that we use twisted compactifications of moduli spaces of curves.

\Cref{sectionhodge} is devoted to the definition of the notion of Hodge structure on a complex modular functor, to the statement
of the main result (existence and uniqueness, \ref{maintheoremcomplex}) and to a discussion of the Frobenius algebra governing Hodge numbers and
of the CohFT associated to characteristic classes of Hodge structures.

In \Cref{sectionsu}, we use the results of the first paper \cite{godfardConstructionHodgeStructures2024} in this series
and \Cref{maintheoremcomplex} to compute Hodge numbers of all $\SU$ modular functors of level $2$ times an odd number.
We also explain how one could use \cite[1.11]{belkaleMotivicFactorisationKZ2023} to compute Hodge numbers of modular functors associated to $\mathfrak{sl}_n$,
and study gaps in Hodge decompositions (\Cref{subsectiongaps}).

The relationship between a modular functor $\Nu$ and its associated ribbon fusion category $\CN$ is the subject of \Cref{sectioncn}.
There, we prove that the functor $\Nu\mapsto \CN$ is fully-faithful.

\Cref{sectionnonabelian} explains what the non-Abelian Hodge correspondence means for modular functors, and \Cref{sectionocneanu}
discusses Ocneanu rigidity and applies it to prove the main result.


\subsection{Acknowledgements}

This paper forms part of the PhD thesis of the author.
The author thanks Julien Marché for his help in writing this paper.
The author also thanks Prakash Belkale, Adrien Brochier, Bertrand Deroin, Pavel Etingof, Philippe Eyssidieux, Matthieu Faitg, Javier Fresán,
Gregor Masbaum, Ramanujan Santharoubane and Aleksander Zakharov
for helpful discussions and correspondences.


\section{Modular functors}\label{sectionmodular}


\subsection{Topological definition of a modular functor}

Our definition of modular functor is engineered to be equivalent to that of Bakalov and Kirillov.
The main reference for this section is their book \cite[chap. 5]{bakalovLecturesTensorCategories2000}.

\begin{definition}\label{definitioncolourset}
    A set of colors is a finite set $\Lambda$ with a preferred element $0\in \Lambda$ and an involution $\lambda\mapsto \lambda^\dagger$
    such that $0^\dagger=0$.
\end{definition}

\begin{definition}\label{definitionsurfaces}
    Let $g,n\geq 0$. We define $S_g^n$ to be the compact surface of genus $g$ with $n$ boundary components,
    and $S_{g,n}$ to be the surface of genus $g$ with $n$ punctures and no boundary.
\end{definition}

We now define the source category of (non-anomalous) modular functors.

\begin{definition}\label{definitioncolouredcategory}
    Let $\Lambda$ be a set of colors. The category $\Surf$ of surfaces colored with $\Lambda$ is such that:
    \begin{description}
        \item[(1)] its objects are compact oriented surfaces $S$ together with
        an orientation preserving identification $\varphi_B: B\simeq S^1$ and a color $\lambda_B\in \Lambda$ for every component $B$ of $\partial S$ ;
        \item[(2)] its morphisms from $\Sigma_1=(S_1,\varphi^1,\underline{\lambda}^1)$
        to $\Sigma_2=(S_2,\varphi^2,\underline{\lambda}^2)$
        are isotopy classes of homeomorphisms $f:S_1\ra S_2$ preserving orientation such that for each component $B_1\subset\partial S_1$
        and its image $f(B_1)=B_2\subset \partial S_2$, we have $\lambda_{B_1}=\lambda_{B_2}$ and $\varphi^2_{B_2}\circ f=\varphi^1_{B_1}$.
        \item[(3)] the composition is the composition of homeomorphisms.
    \end{description}
    This category has a natural monoidal structure induced by the disjoint union $\sqcup$.
\end{definition}

We now describe the additional gluing structure on $\Surf$.
Let $S$ be a compact surface and $\partial_+S\sqcup\partial_-S\subset\partial S$ be two components of its boundary.
Let $\varphi_{\partial_{\pm}S}:\partial_{\pm}S\simeq S^1$ be identifications of these components with $S^1$.

Let $S_{\pm}$ be the surface obtained from $S$ by gluing $\partial_+S$ to $\partial_-S$
along $\varphi_{\partial_-S}^{-1}\circ(z\mapsto z^{-1})\circ\varphi_{\partial_+S}$.
Then $S_{\pm}$ is called the gluing of $S$ along $\partial_{\pm}S$. Notice that if we exchange the roles of $\partial_+S$ and $\partial_-S$
in the gluing process we get a surface $S_{\mp}$ that is canonically isomorphic to $S_{\pm}$. See \Cref{figuregluing}.

\begin{figure}
  \def\svgwidth{0.6\linewidth}
\begingroup%
  \makeatletter%
  \providecommand\color[2][]{%
    \errmessage{(Inkscape) Color is used for the text in Inkscape, but the package 'color.sty' is not loaded}%
    \renewcommand\color[2][]{}%
  }%
  \providecommand\transparent[1]{%
    \errmessage{(Inkscape) Transparency is used (non-zero) for the text in Inkscape, but the package 'transparent.sty' is not loaded}%
    \renewcommand\transparent[1]{}%
  }%
  \providecommand\rotatebox[2]{#2}%
  \newcommand*\fsize{\dimexpr\f@size pt\relax}%
  \newcommand*\lineheight[1]{\fontsize{\fsize}{#1\fsize}\selectfont}%
  \ifx\svgwidth\undefined%
    \setlength{\unitlength}{253.71787244bp}%
    \ifx\svgscale\undefined%
      \relax%
    \else%
      \setlength{\unitlength}{\unitlength * \real{\svgscale}}%
    \fi%
  \else%
    \setlength{\unitlength}{\svgwidth}%
  \fi%
  \global\let\svgwidth\undefined%
  \global\let\svgscale\undefined%
  \makeatother%
  \begin{picture}(1,0.78063141)%
    \lineheight{1}%
    \setlength\tabcolsep{0pt}%
    \put(0,0){\includegraphics[width=\unitlength,page=1]{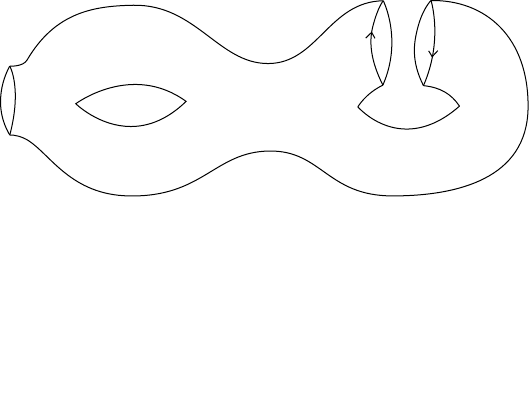}}%
    \put(0.83310804,0.68330276){\color[rgb]{0,0,0}\makebox(0,0)[lt]{\lineheight{1.25}\smash{\begin{tabular}[t]{l}$\partial_-S$\end{tabular}}}}%
    \put(0.59633921,0.67482277){\color[rgb]{0,0,0}\makebox(0,0)[lt]{\lineheight{1.25}\smash{\begin{tabular}[t]{l}$\partial_+S$\end{tabular}}}}%
    \put(0.23714832,0.66825613){\color[rgb]{0,0,0}\makebox(0,0)[lt]{\lineheight{1.25}\smash{\begin{tabular}[t]{l}$S$\end{tabular}}}}%
    \put(0.23714832,0.2549931){\color[rgb]{0,0,0}\makebox(0,0)[lt]{\lineheight{1.25}\smash{\begin{tabular}[t]{l}$S_\pm$\end{tabular}}}}%
    \put(0,0){\includegraphics[width=\unitlength,page=2]{gluing.pdf}}%
  \end{picture}%
\endgroup%

  \caption{Gluing $S_{\pm}$ of $S$ along $\partial_{\pm}S$. The $\bullet$ represents the image of $1$ in the boundary parametrization.}
  \label[figure]{figuregluing}
\end{figure}

\begin{notation}
  The sphere $S_0^2$ with $2$ holes and parametrized boundary components has a unique order $2$ isotopy class of homeomorphism exchanging the boundary components
  in a way that is compatible with their parametrization. In the model $S^1\times [-1,+1]$ it is given by $(z,t)\mapsto (\overline{z},-t)$.
  We will denote this class by $\iota$.
  For any $\lambda$ in a set of colors $\Lambda$, we will also denote by $\iota$ the corresponding morphism in $\Surf$
  from $(S_0^2,\lambda,\lambda^\dagger)$ to $(S_0^2,\lambda^\dagger,\lambda)$.
\end{notation}

We now introduce the notion of a non-anomalous modular functor.
We will work with the more general notion of anomalous modular functors in the sequel so as to include more interesting examples,
like the $\SU$ modular functors and modular functors associated to simple Lie algebras.
However, the anomaly will be of little importance in the Hodge theory of modular functors and can be safely ignored by the reader.
See \Cref{subsectionanomaly} below for a discussion of the anomaly.

\begin{definition}[Non-anomalous Modular Functor]\label{definitionnonanomalousmodularfunctor}
    Let $\Lambda$ be a set of colors and $R$ a ring.
    Then a modular functor is the data of a monoidal functor:
    $$\Nu:\Surf\lra \text{finite projective }R-\text{modules}$$
    where the monoidal structure on the target is understood to be the tensor product.
    This data is augmented by the following isomorphisms.
    \begin{description}
        \item[(G)] For any surface $S$ and pair of boundary components $\partial_{\pm}S$ lying on distinct connected components of $S$,
        let $S_{\pm}$ be the gluing of $S$ along $\partial_{\pm}S$. For any coloring $\underline{\lambda}$
        of the components of $\partial S_{\pm}$, an isomorphism as below is given:
        \begin{equation}\label{gluingunrooted}
            \Nu(S_{\pm},\underline{\lambda})\simeq \bigoplus_{\mu\in\Lambda}\Nu(S,\mu,\mu^\dagger,\underline{\lambda})
            \otimes \Nu(S_0^2,\mu,\mu^\dagger)^\vee.
        \end{equation}
    \end{description}
    The isomorphisms of \textbf{(G)} are assumed to be functorial and compatible with disjoint unions.
    Gluing isomorphisms along distinct pairs of boundary components must commute with each other. Moreover, we ask for the gluing to be symmetric in the sense
    that the following diagram commutes:
    \[\begin{tikzcd}
      {\Nu(S_{\pm},\underline{\lambda})} & {\bigoplus_{\mu\in\Lambda}\Nu(S,\mu,\mu^\dagger,\underline{\lambda})             \otimes \Nu(S_0^2,\mu,\mu^\dagger)^\vee} \\
      \\
      {\Nu(S_{\mp},\underline{\lambda})} & {\bigoplus_{\mu\in\Lambda}\Nu(S,\mu^\dagger,\mu,\underline{\lambda})\otimes \Nu(S_0^2,\mu,\mu^\dagger)^\vee}
      \arrow["{\mathrm{S}}", from=1-2, to=3-2]
      \arrow["{\textbf{(G)}}", from=1-1, to=1-2]
      \arrow["{\textbf{(G)}}", from=3-1, to=3-2]
      \arrow["{\mathrm{can}}"', from=1-1, to=3-1]
    \end{tikzcd}\]
    where $\mathrm{S}$ permutes the factors $\Nu(S,\mu,\mu^\dagger,\underline{\lambda})$
    and acts by $\Nu(\iota)^\vee$ on $\Nu(S_0^2,\mu,\mu^\dagger)^\vee$.

    This rule, also sometimes called fusion or factorization rule, is the most important property of modular functors.
    The functor is also assumed to verify $2$ more axioms, called normalization and non-degeneracy:
    \begin{description}
        \item[(N)] $\Nu(S_0^1,\lambda)$ is canonically isomorphic to $R$ if $\lambda=0$ and is $0$ otherwise.
        $\Nu(S_0^2,0,0)$ is also canonically isomorphic to $R$;
        \item[(nonD)] For each $\lambda$, $\Nu(S_0^2,\lambda,\lambda^\dagger)\neq 0$.
    \end{description}
\end{definition}

\begin{remark}\label{remarkequivalencewithBK}
  Over a field, our definition of modular functor is equivalent with that of Bakalov and Kirillov (\cite[5.1.13]{bakalovLecturesTensorCategories2000}).
  To see this, set $\mathcal{C}:=\VL$ (see \Cref{definitionsemisimplecategory}) and $R:=\bigoplus_\lambda\Nu(S_0^2,\lambda,\lambda^\dagger)[\lambda]\otimes[\lambda^\dagger]$.
\end{remark}

\begin{remark}
  One can think of \Cref{gluingunrooted} as the gluing of $\partial S_{\pm}$ to the boundaries of a cylinder $S_0^2$.
  See \Cref{subcobordism} below for more on this.
  By applying the axiom to the gluing of one boundary component of $S_0^2$ to a boundary component of another copy of $S_0^2$, one gets the following:
  \begin{equation*}
    \mathrm{dim}\: \Nu(S_0^2,\lambda,\mu) = \begin{cases}
      $1$ & \text{if }\lambda=\mu^\dagger\text{,} \\
      $0$ & \text{otherwise.}
    \end{cases}
  \end{equation*}
  Hence if $\lambda=\lambda^\dagger$ for some $\lambda\in\Lambda$, as $\Nu(\iota)$ has order $2$, it acts on $\Nu(S_0^2,\lambda,\lambda)$
  by multiplication by $1$ or $-1$. In the former case we will call $\lambda$ a symmetric object and in the latter an antisymmetric object.
  Note that by normalization, $0$ is symmetric.
\end{remark}

\begin{remark}\label{remarkvacuum}
  By the normalization axiom, for $S$ obtained from $S'$ by removing the interior of a disk we have a canonical isomorphism
  $\Nu(S,0)\simeq \Nu(S')$ called a vacuum isomorphism.
\end{remark}

\begin{remark}\label{remarklevel}
  Let $S_{\pm}$ be a colored surface constructed as a gluing of a surface $S$ along $\partial_{\pm}S$.
  Let $\gamma$ denote the simple closed curve that is the image of $\partial_{\pm}S$ in $S_{\pm}$.
  Then the Dehn twist $T_\gamma$ acts block-diagonally on the decomposition \textbf{(G)}.

  Moreover, one can easily see that its action on the block $\Nu(S,\mu,\mu^\dagger,\underline{\lambda})\otimes \Nu(S_0^2,\mu,\mu^\dagger)$
  is by a scalar $t_\mu\in R^\times$ that depends only on $\mu$ and not on the surface $S$. Also, $t_\mu=t_{\mu^\dagger}$.
\end{remark}

We now define isomorphisms of modular functors.

\begin{definition}\label{definitionmorphismstopological}
  Let $\Nu^1$, $\Nu^1$ be modular functors with respective sets of colors $\Lambda_1$ and $\Lambda_2$.
  An isomorphism of modular functors is a pair $(f,\phi)$, where $f:\Lambda_1\ra\Lambda_2$
  is a bijection preserving $0$ and commuting to $(\cdot)^\dagger$, and $\phi$ is a monoidal natural isomorphism:
  \begin{equation*}
    \phi:\Nu^1\simeq f^*\Nu^2
  \end{equation*}
  commuting to isomorphisms of \textbf{(G)} and \textbf{(N)} in \Cref{definitionnonanomalousmodularfunctor}.
  Here $f^*$ is the pullback by the isomorphism $\mathrm{Surf}_{\Lambda_1}\ra\mathrm{Surf}_{\Lambda_2}$ induced by $f$.

  We will denote by $\Mod$ the groupoid of modular functors.
\end{definition}

We end this subsection by noticing that we can restrict the definition
to genus $0$.

\begin{definition}[Genus $0$ Modular Functor]
  Let $\Lambda$ be a set of colors and $\Surf^0\subset\Surf$ be the full subgroupoid of genus $0$ surfaces.
  The definition of a genus $0$ modular functor is then obtained from that of a modular functor by replacing $\Surf$
  by $\Surf^0$ and restricting the gluing axiom to gluing of boundary components that lie on distinct connected components.
\end{definition}

\begin{remark}
  We have a forgetful functor:
  \begin{equation*}
    \Mod\lra \Mod^0
  \end{equation*}
  from the category of (non-anomalous) modular functors to the category of genus $0$ modular functors.
\end{remark}


\subsection{\texorpdfstring{Cobordism notation, genus $0$, ribbon and braided functors}{Cobordism notation, genus 0, ribbon and braided functors}}\label{subcobordism}

\begin{definition}[Cobordism Notation]
  Let $S$ be a compact surface with parametrized boundary and $\partial\Sigma=I\sqcup O$ a partition of its boundary components into inputs and outputs.
  Let $\Lambda$ be a set of colors and $\ul:I\ra\Lambda$ and $\umu:O\ra\Lambda$ be colorings of the inputs and outputs.
  Then for $\Nu$ a modular functor on $\Surf$, we define:
  \begin{equation*}
    \Nu(S,\umu;\ul) := \Nu(S,\umu^\dagger,\ul)\bigotimes_{o\in O}\Nu(S_0^2,\mu_o^\dagger,\mu_o)^\vee
  \end{equation*}
  i.e. for every output $o$, we replace its color $\mu_o$ by $\mu_o^\dagger$ and tensor with $\Nu(S_0^2,\mu_o^\dagger,\mu_o)^\vee$.
\end{definition}

The raison d'être of the cobordism notation is to explain the term $\Nu(S_0^2,\mu_o^\dagger,\mu_o)^\vee$ in the gluing axiom.
Remark that this term disappears if we write the gluing axiom with this notation. Indeed, say we want to glue the input $i\in I$
to the output $o\in O$ of $S$ to obtain the surface $S_\pm$. Fix some coloring of the boundary components of $S_\pm$.
Then, omitting this coloring in the notation, the gluing axiom becomes:
\begin{equation}\label{gluingcobordism}
  \Nu(S_{\pm})\simeq \bigoplus_{\mu\in\Lambda}\Nu(S,\mu;\mu).
\end{equation}
So both the involution $(\cdot)^\dagger$ and the extra term do not appear.

Also remark that $\Nu(S_0^2,\mu^\dagger,\mu;)=\Nu(S_0^2,\mu^\dagger,\mu)^\vee$.
Hence if we take cobordism notation, the gluing isomorphism of \Cref{gluingunrooted} is the formula for the gluing of a cylinder
$S_0^2$ with two outputs on two inputs of $S$, as in \Cref{figuregluingcobordism}.

\begin{figure}
  \def\svgwidth{0.6\linewidth}
\begingroup%
  \makeatletter%
  \providecommand\color[2][]{%
    \errmessage{(Inkscape) Color is used for the text in Inkscape, but the package 'color.sty' is not loaded}%
    \renewcommand\color[2][]{}%
  }%
  \providecommand\transparent[1]{%
    \errmessage{(Inkscape) Transparency is used (non-zero) for the text in Inkscape, but the package 'transparent.sty' is not loaded}%
    \renewcommand\transparent[1]{}%
  }%
  \providecommand\rotatebox[2]{#2}%
  \newcommand*\fsize{\dimexpr\f@size pt\relax}%
  \newcommand*\lineheight[1]{\fontsize{\fsize}{#1\fsize}\selectfont}%
  \ifx\svgwidth\undefined%
    \setlength{\unitlength}{253.7342438bp}%
    \ifx\svgscale\undefined%
      \relax%
    \else%
      \setlength{\unitlength}{\unitlength * \real{\svgscale}}%
    \fi%
  \else%
    \setlength{\unitlength}{\svgwidth}%
  \fi%
  \global\let\svgwidth\undefined%
  \global\let\svgscale\undefined%
  \makeatother%
  \begin{picture}(1,0.89277844)%
    \lineheight{1}%
    \setlength\tabcolsep{0pt}%
    \put(0,0){\includegraphics[width=\unitlength,page=1]{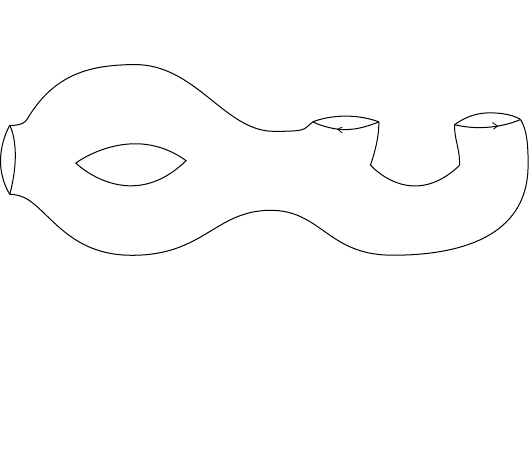}}%
    \put(0.89070155,0.60166395){\color[rgb]{0,0,0}\makebox(0,0)[lt]{\lineheight{1.25}\smash{\begin{tabular}[t]{l}$\partial_-S$\end{tabular}}}}%
    \put(0.61325604,0.60209933){\color[rgb]{0,0,0}\makebox(0,0)[lt]{\lineheight{1.25}\smash{\begin{tabular}[t]{l}$\partial_+S$\end{tabular}}}}%
    \put(0.23713302,0.66821308){\color[rgb]{0,0,0}\makebox(0,0)[lt]{\lineheight{1.25}\smash{\begin{tabular}[t]{l}$S$\end{tabular}}}}%
    \put(0.23713302,0.25497671){\color[rgb]{0,0,0}\makebox(0,0)[lt]{\lineheight{1.25}\smash{\begin{tabular}[t]{l}$S_\pm$\end{tabular}}}}%
    \put(0,0){\includegraphics[width=\unitlength,page=2]{gluingcobordism.pdf}}%
  \end{picture}%
\endgroup%

  \caption{Gluing $S_{\pm}$ of $S$ with $2$ inputs to $S_0^2$ with $2$ outputs along $\partial_{\pm}S$.}
  \label[figure]{figuregluingcobordism}
\end{figure}

We know define ribbon functors by allowing only connected genus $0$ surfaces with exactly $1$ output, and gluings of an input of one surface
to the output of another.

\begin{definition}\label{definitionrootedfunctor}
  For $\Lambda$ a finite set with preferred element $0$,
  let $\Disk$ be the category whose objects are connected genus $0$ surfaces with parametrized boundary components
  colored with $\Lambda$ together with the choice of a boundary component that we call the output. The others will be called inputs.
  The morphisms are those in $\Surf$ that preserve the decomposition in inputs and output.
  We will call such surfaces disks and denote by $D^n$ the disk with $n$ inputs.

  A ribbon functor is the data of functor:
  \begin{equation*}
    \Nu:\Disk\lra\text{finite projective }R-\text{modules}
  \end{equation*}
  together with the following gluing isomorphism:
  \begin{description}
    \item[(G)] For any disks $D_{\mathrm{int}}$ and $D_{\mathrm{ext}}$ such that the input $i\subset\partial D_{\mathrm{ext}}$ and the output
    $o\subset D_{\mathrm{int}}$ are unlabeled, denote by $D$ the disk obtained by gluing of $i$ onto $o$.
    Then for each such data, an isomorphism as below is given (cobordism notations are used; colors of $\partial D$ are implicit):
    \begin{equation}\label{gluingrooted}
        \Nu(D)\simeq \bigoplus_{\mu\in\Lambda}\Nu(D_{\mathrm{ext}},;\mu)\otimes \Nu(D_{\mathrm{int}},\mu;).
    \end{equation}
  \end{description}
  Such isomorphisms are assumed to be functorial and compatible with each other.
  We also require normalization and non-degeneracy, that we combine into one axiom, and a duality axiom:
  \begin{description}
    \item[(N)] $\Nu(D^0,\lambda;)$ is canonically isomorphic to $R$ if $\lambda=0$ and is $0$ otherwise.
    For each $\mu$, $\Nu(D^1,\mu;\mu)$ is also canonically isomorphic to $R$;
    \item[(Dual)] For each $\lambda$ there exists a unique $\mu$ such that $\Nu(0;\lambda,\mu)\neq 0$.
    This $\mu$ will be denoted $\lambda^\dagger$.
  \end{description}
\end{definition}

\begin{definition}\label{definitionbraidedfunctor}
  For $\Lambda$ a finite set with preferred element $0$,
  let $\Disk^b$ be the fundamental groupoid of the space\footnote{The connected components of this space are parametrized
  by the number of disks embedded and the colorings. Each component is parametrized by the possible centers $x_i$ and radii $r_i$.}
  of finite configurations of disjoint $\Lambda$-colored open disks $\{(\mathbb{D}(x_i,r_i),\lambda_i)\mid i\}$ in a $\Lambda$-colored unit disk $(\overline{\mathbb{D}}(0,1),\mu)$.
  We may see an object in $\Disk^b$ as an object of $\Disk$ by considering
  $\overline{\mathbb{D}}(0,1)\setminus \bigsqcup_i \mathbb{D}(x_i,r_i)$.
  By replacing $\Disk$ by $\Disk^b$ in the \Cref{definitionrootedfunctor}, one gets the definition of \emph{braided functor}.
\end{definition}

The essential difference between $\Disk^b$ and $\Disk$ is that Dehn twists around input boundary components
are morphisms in $\Disk$ but not in $\Disk^b$. Hence, for $\Nu$ a braided functor, the group $\Aut{(D^n,\mu;\ul)}$
acting on $\Nu(D^n,\mu;\ul)$ is a braid group while, for $\Nu$ a ribbon functor, it is a (larger) framed braid group.

\begin{remark}
  For a fixed set of colors, cobordism notation induces a functor:
  \begin{equation*}
    \Mod^0\lra \Ribfun
  \end{equation*}
  from the category of genus $0$ modular functors to the category of ribbon functors.
  Moreover, the functor $\Disk^b\ra\Disk$ induces a functor:
  \begin{equation*}
    \Ribfun \lra \Braidfun
  \end{equation*}
  from the category of ribbon functors to the category of braided functors.
\end{remark}


\subsection{Levels for modular functors}

Over integral domains, all modular functors factor through a specific quotient $\Surf(r)$ of $\Surf$.
Note that we will have a slightly augmented definition of level for anomalous modular functors, see \Cref{definitionlevel}.

For modular functors, the following proposition is a direct consequence of the fact that, over a field,
any finite-dimensional representation of the mapping class group
of a surface of genus $g\geq 3$ sends Dehn twists to quasi-unipotent matrices \cite[2.5]{aramayonaRigidityPhenomenaMapping2016}.
For genus $0$ modular functors and ribbon functors $\Nu$
this is the Anderson-Moore-Vafa theorem (\cite[8.18.2]{etingofTensorCategories2015}).
The case of braided functors requires an extra argument, that was communicated to the author by Pavel Etingof, and that we now sketch.
A braided fusion category $\Cat$ admits a pivotalization $\widetilde{\Cat}\ra\Cat$. It is a spherical category
(see \cite[7.21.10]{etingofTensorCategories2015}) whose simple objects map $2$ to $1$ onto those of $\Cat$.
Moreover, the braiding on $\Cat$ lifts to $\widetilde{\Cat}$. As $\widetilde{\Cat}$ is spherical, braided and fusion,
it is ribbon (see \cite[A.4]{henriquesCategorifiedTraceModule2016}). In particular the braiding on $\widetilde{\Cat}$ has finite order
on each pair of objects.
Hence the braiding on $\Cat$ has the same property.
See \Cref{subweaklyribbon,subocneanu} for the definitions of braided and ribbon fusion categories, and their relation to braided and ribbon functors.

\begin{proposition}\label{propositionlnonanomalousevel}
  Let $\Nu$ be a modular functor over an integral domain $R$ with set of colors $\Lambda$.
  Then the twists $t_\mu\in R^\times$ as in \Cref{remarklevel} are roots of unity.
  The same conclusion holds for genus $0$ modular functors and ribbon functors over integral domains.

  Let $\Nu$ a braided functor over an integral domain $R$ with set of colors $\Lambda$.
  Then with the notation $\beta_{\lambda;\mu,\nu}:\Nu(\lambda;\mu,\nu)\ra \Nu(\lambda;\nu,\mu)$ for the braidings,
  the eigenvalues of the $\beta_{\lambda;\mu,\nu}\circ\beta_{\lambda;\mu,\nu}$ are roots of unity.
\end{proposition}

\begin{definition}[Level]\label{defintionnonanomalouslevel}
  Let $\Nu$ be a modular, genus $0$ modular or ribbon functor over a ring $R$ with set of colors $\Lambda$.
  A level for $\Nu$ is an integer $r$ such that $t_\mu^r=1$ for all $\mu\in\Lambda$.
  For $\Nu$ a braided functor over a ring $R$ with set of colors $\Lambda$,
  a level is an integer $r$ such that each
  $(\beta_{\lambda;\mu,\nu}\circ\beta_{\lambda;\mu,\nu})^{r}$ is the identity.
\end{definition}

By \Cref{propositionlnonanomalousevel}, all modular, genus $0$ modular, ribbon or braided functors over integral domains admit levels.

\begin{definition}
  For $\Lambda$ a set of colors and $r\geq 1$, denote by $\Surf(r)$ the quotient of the groupoid $\Surf$ by the $r$-th powers of Dehn twists.
\end{definition}

\begin{proposition}\label{propr}
  Let $\Nu$ be a modular functor and $r$ a level for it. Then $\Nu$ factors through $\Surf(r)$ as a functor.
\end{proposition}

One can provide similar definitions of $\Surf^0(r)$, $\Disk(r)$ and $\Disk^b(r)$, for which analogs of \Cref{propr} hold.


\subsection{Geometric definition of a modular functor}\label{subgeometricdefinition}

We define geometric modular functors, which are equivalent to topological ones.
The main reference for this section is \cite[chap. 6]{bakalovLecturesTensorCategories2000},
but beware that we have made different choices of moduli spaces.

Here is a quick guide to the approach of Bakalov and Kirillov. The idea is to replacing surfaces by Riemann surfaces and boundary components
by marked points with non-zero tangent vectors on them. The groupoid of surfaces $\Surf$ is then replaced
by some moduli spaces $\Mgb{g}{n}$ of such Riemann surfaces with structures.
In this setting, representations of the groupoid of surfaces (i.e. topological modular functors)
are replaced by flat bundles on the $\Mgb{g}{n}$. The gluing axiom and the normalization axiom are reformulated in this context,
using nodal degenerations of Riemann surfaces.

However moduli spaces $\Mgb{g}{n}$ are not compact/proper and the gluing maps are not algebraic (see \cite[6.2.11]{bakalovLecturesTensorCategories2000}).
This makes the application of non-Abelian Hodge theory on the $\Mgb{g}{n}$ harder.
Hence we replace them by proper DM stacks $\Mgrb{g}{n}{r}$, where $r$ is a level for the modular functor.
With this approach, we get algebraic gluing maps, see \Cref{separatinggluing,nonseparatinggluing}.

Let us first mention some non-standard definition we will use.
\begin{definition}
  A pair $(g,n)$ of non-negative integers will be called tangent-stable if
  $(g,n)\neq (0,0),(0,1),(1,0)$.
\end{definition}
Below, when talking about the moduli spaces $\Mgrp{g}{n}{r}$ or $\Mgrb{g}{n}{r}$,
\textbf{we will always implicitly assume that $(g,n)$ is tangent-stable.}

\begin{notation}
  Let $\Mgrp{g}{n}{r}$ be the moduli space of connected stable nodal $r$-twisted curves of genus $g$ with $n$ distinct smooth marked points
  (see \cite[1.3]{chiodoStableTwistedCurves2008}).
\end{notation}

\begin{remark}\label{remarkmgnr}
  Alternatively, $\Mgrp{g}{n}{r}$ can be defined as follows. Let $D_i\subset \oMg{g}{n}$ for $i=1,\dotsc,k$ be the components
  of the boundary divisor. Then $\Mgrp{g}{n}{r}$ is obtained from $\oMg{g}{n}$ by taking $r$-th root stack independently locally
  on each $D_i$:
  \begin{equation*}
    \Mgrp{g}{n}{r}=\oMg{g}{n}\left[\sum_i\frac{D_i}{r}\right].
  \end{equation*}
  See \cite[2.3, 4.5]{chiodoStableTwistedCurves2008} for more details on what is meant by independent root stack on a normal crossing divisor
  with self-crossings in this context.
  With this description, it is clear that the fundamental groupoid of $\Mgrp{g}{n}{r}$ is equivalent
  to the quotient $\PMod{S_{g,n}}/\langle T_\delta^r\mid \delta\rangle$ of the pure mapping class group of the $n$ times punctured
  genus $g$ surface by all $r$-th powers of Dehn twists.
\end{remark}

Let us now introduce a construction that transforms a principal $\C^\times$-bundle into a $\mu_r$-gerbe.

\begin{definition}\label{cstartomur}
  Let $P\ra B$ be a principal $\C^\times$-bundle over a stack $B$ and $r\geq 1$ an integer. Let $L\ra B$ be the corresponding line bundle.
  Denoting by $D\subset L$ the zero section, $P_r$ is defined as $D/r$ the corresponding divisor of the root stack $L[D/r]$.
  It is a $\mu_r$-gerbe over $B$.
\end{definition}

\begin{notation}
  Let $\Mgrb{g}{n}{r}$ be the moduli space of connected stable nodal $r$-twisted curves of genus $g$
  with $n$ distinct smooth marked orbifold points of order $r$ and a section at each such point.
\end{notation}

\begin{remark}\label{conventionzerotwo}
  The moduli space $\Mgb{0}{2}$ of smooth genus $0$ curves with $2$ marked and non-zero tangent vectors
  at the marked points
  is a variety (isomorphic to $\C^\times$) and hence a DM stack, but its $r$-twisted cousin $\Mgrb{0}{2}{r}$
  is not a DM stack. However, it's orbifold fundamental group is $\mu_r$. \textbf{Hence, we will use the convention that
  $\bm{\Mgrb{0}{2}{r}=\Br{r}}$}. We will think of $\Mgrb{0}{2}{r}=\Br{r}$ as the following family of curves with sections: it is the curve $\mathbb{P}^1/\mu_r$
  together with a section at each of the 2 orbifold points $0$ and $\infty$: we make $\xi\in\mu_r$ act
  trivially on the curve, act by $\xi$ on the section at $0$ and act by $\xi^{-1}$
  on the section at $\infty$. We mark $0$ by $1$ and $\infty$ by $2$.
\end{remark}

\begin{remark}
  Let $\justC\ra \Mgrp{g}{n}{r}$ be the universal curve, where marked points are replaced by order $r$ orbifold points
  (this does not change the moduli space). Let $\Sigma_i\subset\justC$ be the $i$-th marked gerbe.
  Then we have:
  \begin{equation*}
    \Mgrb{g}{n}{r} = \Sigma_1\times_{\Mgrp{g}{n}{r}}\dotsb\times_{\Mgrp{g}{n}{r}}\Sigma_n.
  \end{equation*}
  Hence $\Mgrb{g}{n}{r}$ is a $\mu_r^n$-gerbe over $\Mgrp{g}{n}{r}$.

  This description makes clear the following fact. Denote by $\overline{\justM}$ the moduli space of curves in
  $\Mgrp{g}{n}{r}$ together with non-zero tangent vectors at each marked point. Then $\overline{\justM}$
  is a principal $(\C^\times)^n$-bundle over $\Mgrp{g}{n}{r}$ and applying to it the construction \Cref{cstartomur}
  yields $\Mgrb{g}{n}{r}$. Hence the fundamental groupoid of $\Mgrb{g}{n}{r}$ is equivalent
  to the quotient $\PMod{S_g^n}/\langle T_\delta^r\mid \delta\rangle$.
\end{remark}

Let us now describe gluing maps between these moduli spaces.

\begin{proposition}[Separating nodes]\label{separatinggluing}
  Let $(g_1,n_1)$ and $(g_2,n_2)$ be tangent-stable. Then the gluing map:
  \begin{equation*}
      \Mgrb{g_1}{n_1+1}{r}\times\Mgrb{g_2}{n_2+1}{r}\lra \Mgrb{g_1+g_2}{n_1+n_2}{r}
  \end{equation*}
  is the composition of a $\mu_r$-gerbe, possibly a double cover\footnote{The double cover appears for $g_1=g_2$ and $n_1=n_2=0$.},
  and the embedding of a divisor.
  The induced map of orbifold fundamental groups is the morphism:
  \begin{equation*}
      \Modl{r}{S_{g_1}^{n_1+1}}\times\Modl{r}{S_{g_2}^{n_2+1}}\lra \Modl{r}{S_{g}^{n_1+n_2}}
  \end{equation*}
  obtained by gluing the $(n_1+1)$-th boundary component of $S_{g_1}^{n_1+1}$ to the $(n_2+1)$-th boundary component of $S_{g_2}^{n_2+1}$.
\end{proposition}

\begin{proposition}[Nonseparating nodes]\label{nonseparatinggluing}
  Let $g,n\geq 0$ such that $(g,n)$ and $(g-1,n+2)$ are tangent-stable. Then the gluing map:
  \begin{equation*}
      \Mgrb{g-1}{n+2}{r}\lra \Mgrb{g}{n}{r}
  \end{equation*}
  is the composition of a $\mu_r$-gerbe, a double cover and the embedding of a divisor.
  The induced map of orbifold fundamental groups is the morphism:
  \begin{equation*}
      \Modl{r}{S_{g-1}^{n+2}}\lra \Modl{r}{S_{g}^{n}}
  \end{equation*}
  obtained by gluing together the $(n+1)$-th and $(n+2)$-th boundary components of $S_{g-1}^{n+2}$.
\end{proposition}

\begin{definition}[Non-anomalous Modular Functor]\label{definitiongeometricnonanomalousmodularfunctor}
  Let $\Lambda$ be a set of colors, $r\geq 1$ an integer and $R$ a ring.
  Then a non-anomalous geometric modular functor with level $r$ is the data,
  for each $(g,n)$ stable and $\ul\in\Lambda^n$,
  of a local system $\Nu_g(\ul)$ of finite projective $R$-modules over $\Mgrb{g}{n}{r}$, together with some isomorphisms
  described below:
  \begin{description}
      \item[(G-sep)] For each gluing map
      \begin{equation*}
        q:\Mgrb{g_1}{n_1+1}{r}\times\Mgrb{g_2}{n_2+1}{r}\lra \Mgrb{g_1+g_2}{n_1+n_2}{r}
      \end{equation*}
      and each $\ul$, an isomorphism:
      \begin{equation*}
        q^*\Nu_{g_1+g_2}(\lambda_1,\dotsc,\lambda_n)\simeq \bigoplus_\mu\Nu_{g_1}(\lambda_1,\dotsc,\lambda_{n_1},\mu)\otimes
        \Nu_0(\mu,\mu^\dagger)\otimes
        \Nu_{g_2}(\lambda_{n_1+1},\dotsc,\lambda_n,\mu^\dagger);
      \end{equation*}
      \item[(G-nonsep)] For each gluing map
      \begin{equation*}
        p:\Mgrb{g-1}{n+2}{r}\lra \Mgrb{g}{n}{r}
      \end{equation*}
      and each $\ul$, an isomorphism:
      \begin{equation*}
        p^*\Nu_g(\lambda_1,\dotsc,\lambda_n)\simeq \bigoplus_\mu\Nu_{g-1}(\lambda_1,\dotsc,\lambda_n,\mu,\mu^\dagger)
        \otimes \Nu_0(\mu,\mu^\dagger);
      \end{equation*}
      \item[(N)] For each forgetful map $f:\Mgrb{g}{n+1}{r}\ra\Mgrb{g}{n}{r}$, and each $\ul$, a vacuum isomorphism:
      \begin{equation*}
        f^*\Nu_g(\lambda_1,\dotsc,\lambda_n)\simeq \Nu_g(\lambda_1,\dotsc,\lambda_n,0)
      \end{equation*}
      and a canonically isomorphism $\Nu_0(0,0)\simeq \underline{R}$;
      \item[(Perm)] For each $\ul\in\Lambda^n$ and permutation $\sigma\in S_n$, an isomorphism:
      \begin{equation*}
        \Nu_g(\lambda_1,\dotsc,\lambda_n)\simeq \sigma^*\Nu_g(\lambda_{\sigma(1)},\dotsc,\lambda_{\sigma(n)}).
      \end{equation*}
  \end{description}
  The isomorphisms of \textbf{(G-sep)}, \textbf{(G-nonsep)}, \textbf{(N)} and \textbf{(Perm)}
  are to be compatible with each other and repeated applications.
  Moreover, we ask for the gluing to be symmetric in the sense that
  for each gluing isomorphism above, the change of variable $\mu\mapsto\mu^\dagger$
  on the right-hand side has the same effect as permuting the summands and applying the \textbf{(Perm)} isomorphisms
  $\Nu_0(\mu,\mu^\dagger)\simeq\Nu_0(\mu^\dagger,\mu)$ induced by $\sigma\in S_2\setminus\{\id\}$.
  
  The functor is also assumed to verify the non-degeneracy axiom:
  \begin{description}
      \item[(nonD)] For each $\lambda$, $\Nu_0(\lambda,\lambda^\dagger)\neq 0$.
  \end{description}
\end{definition}

Sometimes, we will shorten $\Nu_0(\ul)$ to $\Nu(\ul)$.

Again, as in the topological case, one can give similar geometric definitions of genus $0$ modular functor of level $r$
and ribbon functor of level $r$: for the genus $0$ case we restrict the above definition to $g=0$ and remove the \textbf{(G-sep)} axiom;
for ribbon functors, we replace the gluing axioms by the gluing axiom \textbf{(G-nonsep)} in cobordism notation
and in the permutation axiom allow only permutations of inputs.

As for braided functors, one needs to work with slight variations of the moduli spaces.

\begin{notation}
  For $n\geq 1$, denote by $\Mrpo{n}{r}$ the moduli space of connected stable nodal $r$-twisted curves
  of genus $0$ with $n$ distinct smooth points and $1$ order $r$ stacky point.
  \emph{and a section at the stacky point}. We will use the convention $\Mrpo{1}{r}=*$.
\end{notation}

\begin{remark}
  Let us mark the points in $\Mgrp{0}{n+1}{r}$ from $0$ to $n$.
  The space $\Mrpo{n}{r}$ can de identified with $\Sigma_0\subset \Cgrp{0}{n+1}{r}$,
  the $0$-th stacky marked point $\Sigma_0$ in the universal curve $\Cgrp{g}{n}{r}\ra \Mgrp{g}{n}{r}$.
  Its fundamental group can be identified with
  $\PMod{D_n}/\langle T_\delta^r\mid \delta\rangle$, where $D_n$ is the closed disk with $n$ points removed.
  The pure mapping class group $\PMod{D_n}$ is the pure braid group $\mathrm{PB}_n$.
\end{remark}

Replacing the $\Mgrb{0}{n+1}{r}$ with the $\Mrpo{n}{r}$ in the definition of geometric ribbon functor, one gets
the definition of geometric braided functor. For details, the reader can find a full definition in \cite[2.20]{godfardQuasigeometricityConformalBlocks2025}.

We now define isomorphisms of geometric modular functors.

\begin{definition}\label{definitionmorphismsgeometric}
  Let $\Nu^1$, $\Nu^2$ be geometric modular functors with respective sets of colors $\Lambda_1$ and $\Lambda_2$.
  An isomorphism form $\Nu^1$ to $\Nu^2$ is a pair $(f,\phi)$, where $f:\Lambda_1\ra\Lambda_2$
  is a bijection preserving $0$ and commuting to $(\cdot)^\dagger$, and $\phi$ is a family of isomorphisms:
  \begin{equation*}
    \phi:\Nu^1_g(\ul)\simeq \Nu^2_g(f(\ul))
  \end{equation*}
  commuting to isomorphisms \textbf{(G-sep)}, \textbf{(G-nonsep)}, \textbf{(N)} and \textbf{(Perm)} of
  \Cref{definitiongeometricnonanomalousmodularfunctor}.
\end{definition}

\begin{theorem}
  There is an equivalence between the groupoid of
  topological non-anomalous modular functors with level $r$
  and the groupoid of geometric non-anomalous
  modular functors with level $r$.
\end{theorem}
\begin{proof}
  This is essentially \cite[6.4.2]{bakalovLecturesTensorCategories2000} together with the definition of the level.
\end{proof}
Again, we have similar results for genus $0$ modular functors of level $r$, ribbon functors of level $r$ and braided functors of level $r$.


\subsection{Anomalous modular functors}\label{subsectionanomaly}

For our results to include the examples of interest ($\SU$ modular functors), we need to work in the more general framework
of anomalous modular functors. In a nutshell, anomalous modular functors are functors from a specific central extension of $\Surf$ by $\Z$
(with the topological definition) or local systems on a specific cyclic gerbe over the $\Mgrb{g}{n}{r}$ (in the geometric definition).
The introduction of the anomaly is of little consequence
for the proof of \Cref{maintheoremcomplex}.
In this subsection, for completeness we review the definition of these extensions. 
The main reference for this section is \cite[5.7 and 6.7]{bakalovLecturesTensorCategories2000}.

Let us start with the treatment of the anomaly in the topological setting.

\begin{definition}\label{definitionlagrangian}
  Let $S$ be a compact surface, which can have non-empty boundary.
  Let $\widehat{S}$ be the closed surface obtained from $S$ by capping off each boundary component $B\simeq S^1$ with $D^2$.
  A Lagrangian $L$ on $S$ is a subspace of $H_1(\widehat{S};\R)$ of half dimension on which the intersection form vanishes.
  A split Lagrangian on $S$ is a Lagrangian $L$ that is a direct sum of Lagrangians on the connected components of $S$.

  We denote by $T_S$ the fundamental groupoid of the topological space of split Lagrangians on $S$, with the convention
  that in genus $0$, the unique Lagrangian $\{0\}$ has automorphism group $\Z$.
\end{definition}

Note that $T_S$ is connected and that for each object $L$,
$\Aut{L}\simeq \Z^{\pi_0(S)}$ canonically (see \cite[5.7.2]{bakalovLecturesTensorCategories2000}).

\begin{definition}\label{definitionsurfanomaly}
  Let $\Lambda$ be a set of colors. We define $\tSurf$ as the groupoid whose:
  \begin{description}
    \item[(1)] objects are objects $\Sigma=(S,\varphi,\underline{\lambda})$ of $\Surf$ together with a split Lagrangian $L\in\mathrm{ob}\:T_S$;
    \item[(2)] morphisms from $(\Sigma_1,L_1)$ to $(\Sigma_2,L_2)$ are pairs $(f,u)$ with $f$ a morphism from $\Sigma_1$
    to $\Sigma_2$ in $\Surf$ and $u$ a morphism from $L_1$ to $f^*L_2$ in $T_{S_1}$. Compositions are done in the obvious way.
  \end{description}
This category has a natural monoidal structure induced by the disjoint union $\sqcup$.
\end{definition}

We now describe how gluings are performed in $\tSurf$.
With the notations of the discussion after \Cref{definitioncolouredcategory},
we assume that $S$ comes equipped with a Lagrangian $L$.
Then the gluing $S_{\pm}$ comes equipped with the Lagrangian $L_{\pm}$
generated by the image of $L$ together with the image of $\partial_{\pm}S$.

\begin{definition}[Modular Functor]\label{definitionmodularfunctor}
  Let $\Lambda$ be a set of colors and $R$ a ring.
  Then a modular functor is a monoidal functor:
  $$\Nu:\tSurf\lra \text{finite projective }R-\text{modules}$$
  together with the data and properties \textbf{(G)}, \textbf{(N)}, \textbf{(nonD)}
  and symmetry of gluing described in \Cref{definitionnonanomalousmodularfunctor},
  where the gluing in $\Surf$ is replaced by that in $\tSurf$.
\end{definition}

\begin{remark}\label{remarkcentralcharge}
  For any connected surface $\tilde{\Sigma}$ in $\tSurf$ extending
  an object $\Sigma$ in $\Surf$, we have a canonical central extension:
  \begin{equation*}
    0\lra \Z\lra \Aut{\tilde{\Sigma}} \lra \Aut{\Sigma}\lra 1.
  \end{equation*}
  We shall denote by $\gamma$ the generator $1\in\Z$.
  Then for any modular functor $\Nu$, $\gamma$ acts on $\Nu(\tilde{\Sigma})$
  by multiplication by a scalar $c\in R^\times$ independent of $\tilde{\Sigma}$.
  This scalar $c$ will be called the \textbf{central charge} of $\Nu$.
\end{remark}

\begin{remark}\label{remarkmgnrs}
  As the anomaly central extension is canonically split in genus $0$, we have a natural embedding $\Surf^0\subset\tSurf$.
  Hence we get a forgetful functor:
  \begin{equation*}
    \tMod\lra \Mod^0
  \end{equation*}
  from the category of modular functors to the category of genus $0$ modular functors.
\end{remark}

We can extend the definition of the level to include the central charge.

\begin{proposition}\label{propositionlevel}
  Let $\Nu$ be a modular functor over an integral domain $R$ with set of colors $\Lambda$.
  Then the twists $t_\mu\in R^\times$ (see \Cref{remarklevel}) and the central charge $c\in R^\times$ (see \Cref{remarkcentralcharge}) are roots of unity in $R$.
\end{proposition}

\begin{definition}[Level]\label{definitionlevel}
  Let $\Nu$ be a modular functor over a ring $R$ with set of colors $\Lambda$.
  A level for $\Nu$ is a pair $(r,s)$ of positive integers such that $c^s=1$ and $t_\mu^r=1$ for all $\mu\in\Lambda$.
  In particular, any modular functor over an integral domain admits a level.
\end{definition}

Let us now describe briefly the equivalent treatment of the anomaly in the geometric setting.

\begin{definition}\label{definitionanomalybundle}
  Let $\pi:\justC\ra\justM$ be a family of nodal curves. Then we will denote by $H_\justM$ the holomorphic line bundle
  $\mathrm{det}(\pi_*\omega_{\justC/\justM})^{\otimes(-2)}$
  over $\justM$, where $\omega_{\justC/\justM}$ is the relative dualizing sheaf,
  and by $\tjustM$ the corresponding principal $\C^*$-bundle on $\justM$.
\end{definition}

\begin{definition}\label{definitionmgnrs}
  Let $s\geq 1$ be an integer.
  Let $\widetilde{\overline{\justM}}_{g}^{n}(r)$ be the principal $\C^\times$-bundle associated to $\Mgrb{g}{n}{r}$ via \Cref{definitionanomalybundle}.
  Then by applying the procedure of \Cref{cstartomur}, we get a $\mu_s$-gerbe over $\Mgrb{g}{n}{r}$ that we denote $\Mgrbt{g}{n}{r}{s}$.
\end{definition}

\begin{remark}
  The fundamental groupoid of $\Mgrbt{g}{n}{r}{s}$ is the quotient $\tPMod{S_g^n}/\langle\gamma^s,T_\delta^r\mid [\delta]\in L\rangle$
  where $\tPMod{S_g^n}$ is the anomaly central extension by $\Z$ of the pure mapping class of $S_g^n$ corresponding to the Lagrangian $L$,
  and $\gamma$ is the generator $1\in\Z$ in the extension. Note that for $[\delta]\in L$, $T_\delta\in\tPMod{S_g^n}$
  is the pair $(T_\delta,\id_L:L\ra L=T_\delta^*L)$.
\end{remark}

Let $f:\justM\ra\justM'$ be a gluing map or the contraction of some genus $0$
components.
By studying relative (log)-differential forms, we see that $f^*H_{\justM'}\simeq H_\justM$ canonically.
So we get a natural map $\tilde{f}:\tjustM\ra\tjustM'$.
One can then use this fact to extend the gluing maps to:
\begin{align}
  \Mgrbt{g_1}{n_1+1}{r}{s}\times\Mgrbt{g_2}{n_2+1}{r}{s}&\lra \Mgrbt{g_1+g_2}{n_1+n_2}{r}{s}\label{equationgluinganomalysep} \\
  \Mgrbt{g-1}{n+2}{r}{s}&\lra \Mgrbt{g}{n}{r}{s}\label{equationgluinganomalynonsep}
\end{align}

\begin{definition}[Modular Functor]\label{definitiongeometricmodularfunctor}
  Let $\Lambda$ be a set of colors, $r,s\geq 1$ and $R$ a ring.
  Then the geometric definition of an (anomalous) modular functor is obtained from that of a non-anomalous modular
  functor by replacing the moduli spaces $\Mgrb{g}{n}{r}$ by $\Mgrbt{g}{n}{r}{s}$ and the gluing maps
  by those of \Cref{equationgluinganomalysep,equationgluinganomalynonsep}.
\end{definition}

One defines isomorphisms of topological and geometric modular functors as in \Cref{definitionmorphismstopological,definitionmorphismsgeometric}.
Then again, by \cite[6.7.12]{bakalovLecturesTensorCategories2000}, one has the following.

\begin{theorem}
  There is an equivalence between the groupoid of
  topological (anomalous) modular functors with level $(r,s)$ and the groupoid of geometric (anomalous) modular functors
  with level $(r,s)$.
\end{theorem}


\section{Hodge structures on modular functors}\label{sectionhodge}

Throughout this section, a level $(r,s)$ and a set of colors $\Lambda$ are fixed.


\subsection{Definition of Hodge structure}


\subsubsection{Hodge structures on flat bundles}

Let us define what we mean by Hodge structure in the case of flat bundles.

\begin{definition}[CVHS]\label{definitionCVHS}
  Let $(E,\nabla)$ be a flat $\C$-local system over a complex DM stack $B$. A complex variation of Hodge structure (CVHS) on $(E,\nabla)$
  is a rational number $m$ called weight together with a decomposition:
  \begin{equation}\label{equationHodgedecomposition}
      E = \bigoplus_{\substack{p+q=m\\p,q\in\Q}}E^{p,q}
  \end{equation}
  into smooth subbundles satisfying Griffiths transversality, i.e. for all $p,q$:
  \begin{equation*}
      \nabla E^{p,q}\subset E^{p-1,q+1}\otimes\Omega_B^{1,0}\oplus E^{p,q}\otimes\Omega_B^{1}\oplus E^{p+1,q-1}\otimes\Omega_B^{0,1}.
  \end{equation*}
  We also assume that the structure is polarizable, i.e. that there exists a flat Hermitian form $h$ on $E$
  for which the decomposition is orthogonal such that for all $p+q=m$, $h$ is definite on $E^{p,q}$, and if $p-p'\in\Z$,
  the signs of $h$ on $E^{p,m-p}$ and $E^{p',m-p'}$ differ by $(-1)^{p-p'}$.
  The choice of a polarization is not part of the data of a CVHS.
\end{definition}

\begin{remark}\label{remarkrationalweight}
  Our choice to allow $p$, $q$ and $m$ to be in $\Q$ and not just $\Z$ may seem peculiar.
  Indeed, with our definition, by transversality, for each class $x\in\Q/\Z$, the summand $\bigoplus_{p\in x,q}E^{p,q}$
  is stabilized by the connection. Hence any CVHS with $p,q\in \Q$ can be translated on irreducible factors into one with $p,q\in \Z$.
  However, when we will define the notion of CVHS on a modular functor, we will need to allow $p$ and $q$ to be rational
  to ensure that our proof of existence works. Indeed, the proof relies on the map $r$ in \Cref{actiontoCVHSmodular},
  which only produces a CVHS with bidegrees in $\R$. To pass from $\R$ to $\Q$ we can choose a map $\pi:\R\ra\Q$
  of $\Q$-vector spaces sending $1$ to $1$ and changing every bidegree $(p,q)$ to $(\pi(p),\pi(q))$.
  One could be tempted to find some map $\pi:\R\ra\Z$ that works, but because of the tensor products appearing in the gluing axioms,
  the map would need to be one of $\Z$-modules and send $1$ to $1$, which is not possible.
\end{remark}

\begin{remark}
  For $E=0$, we use the convention that a CVHS on $E$ is empty data. In particular, $E=0$ is of weight $m$ for every $m\in\Q$.
\end{remark}

\begin{proposition}[{\cite[4.1]{simpsonHiggsBundlesLocal1992}}]
  If the local system associated to $(E,\nabla)$ is irreducible, then a complex variation of Hodge structure on $E$ is unique up to translation
  (i.e. $\exists x,y\;\forall p,q,\;E'^{p,q}=E^{p+x,q+y}$).
\end{proposition}

\begin{remark}
  Note that if $(E,\nabla)\ra B$ is unitary, it supports a trivial CVHS.
\end{remark}

A CM number field is a number field $L$ such that there exists an involution $C:L\ra L$ which for every embedding $\iota:L\ra \C$
coincides with complex conjugation. Equivalently, a CM number field is a quadratic extension of a totally real field.
In particular, there is a notion of Hermitian form on a vector space over a CM field.

\begin{definition}[RVHS]\label{definitionRVHS}
  Let $(E,\nabla)$ be a flat $L$-local system over a complex DM stack $B$, for $L$ a CM number field.
  A rational variation of Hodge structures (RVHS) on $(E,\nabla)$ is the data of a rational number $m$ called weight, and for each embedding
  $\iota:L\hookrightarrow\C$, of a CVHS on $E_\iota=E\otimes \C$. We also ask that if $\iota:L\hookrightarrow\C$ and
  $\overline{\iota}:L\hookrightarrow\C$ are complex conjugate embeddings, then for all $p,q$, $\overline{E_\iota^{p,q}}=E_{\overline{\iota}}^{q,p}$.
  We also ask that the structure is polarizable over $L$ in the sense that there exists a Hermitian form $h$ on $E$ defined over $L$ and preserved by $\nabla$,
  such that for each $\iota:L\ra \C$, $h$ is a polarization for the CVHS on $E_\iota$. Such $h$ is called a polarization for the RVHS.

  Such a RVHS will be said to be strongly polarized if $m\in\Z$, all Hodge bidegrees of each $E_\iota$ are in $\Z^2$, and it is given with a polarization $h$
  (defined over $L$) such that $h$ has sign $(-1)^p$ on each $E_\iota^{p,q}$.
\end{definition}

\begin{definition}
  Let $E$ be a local system with a Hodge structure (CVHS or RVHS), then its dual $E^\vee$ has a Hodge structure with
  $(E^\vee)^{p,q}=(E^{-p,-q})^\vee$. If $E'$ is another local system with Hodge structure on the same base,
  then $E\otimes E'$ has a Hodge structure with decomposition:
  \begin{equation*}
    (E\otimes E')^{p,q}=\bigoplus_{\substack{p_1+p_2=p\\q_1+q_2=q}}E^{p_1,q_1}\otimes E'^{p_2,q_2}.
  \end{equation*}
\end{definition}


\subsubsection{Semisimplicity and definition over number fields}

As mentioned in \Cref{remarksemisimplicity}, the main result (\Cref{maintheoremcomplex}) relies on knowing that conformal blocks are semisimple.
The semisimplicity of conformal blocks in general is proved in the separate paper \cite{godfardSemisimplicityConformalBlocks2025}.

\begin{theorem}[{\cite[1.1]{godfardSemisimplicityConformalBlocks2025}}]\label{theoremsemisimplicity}
  Let $\Nu$ be a geometric modular functor.
  Then, for any $g,n\geq 0$, $(g,n)\neq (0,0),(0,1),(1,0)$, and $\ul\in\Lambda$, the conformal block $(\Nu_g(\ul),\nabla)$ is semisimple.
  The same holds for genus $0$ modular, ribbon or braided functors.
\end{theorem}

The following is a known consequence of Ocneanu rigidity.

\begin{theorem}[\ref{functornumberfield}]\label{theoremdefinednumberfield}
  Let $\Nu$ be a topological complex modular, genus $0$ modular, ribbon or braided functor.
  Then there exists a number field $K\subset\C$ such that $\Nu$ is definable over $K$.
\end{theorem}

The theorem above shows that any unitary complex modular functor gives rise to some Galois conjugates, that may then not be unitary
and thus be more interesting from the point of view of Hodge theory.
This is what happens in the case of complex modular functors coming from complex finite dimensional simple Lie algebras $\mathfrak{g}$.
They are known to be unitary (see \cite{belkaleUnitarityKZHitchin2012}), and hence carry trivial uninteresting Hodge structures.
However, in the topological definition, they are defined over cyclotomic fields and their Galois conjugates have non-trivial Hodge structures
(see \Cref{subintronumbers,sectionsu}).

\begin{remark}
  In the case $\mathfrak{g}=\mathfrak{sl}_2$, up to a change of sign in the twists of the ribbon structure, the associated modular functors
  are equivalent to the $\SU$ modular functors. See \Cref{sectionsu,subsectionlie} below for more details.
\end{remark}

\subsubsection{Hodge structures on modular functors}

Let us first introduce weight functions.

\begin{definition}[Weight function]\label{definitionweightfunction}
  Let $\Nu$ be a modular functor with set of colors $\Lambda$.
  A weight function on $\Nu$ is a function $w$ from the set $\{(g,\ul)\mid \Nu_g(\ul)\neq 0\}$ to $\Q$ which is
  invariant under permutation of the coordinates of $\ul$ and
  compatible with gluings and vacuum in the sense that, whenever all the terms are defined, we have the equalities:
  \begin{description}
    \item[(sep)] for $g=g_1+g_2$ and any $\ul$, $\umu$ and $\nu$, $w_g(\ul,\umu)=w_{g_1}(\ul,\nu)+w_{g_2}(\umu,\nu^\dagger)-w_0(\nu,\nu^\dagger)$;
    \item[(non-sep)] for any $g$, $\ul$ and $\nu$, $w_g(\ul)=w_{g-1}(\ul,\nu,\nu^\dagger)-w_0(\nu,\nu^\dagger)$;
    \item[(vacuum)] for any $g$ and $\ul$,  $w_g(\ul,0)=w_g(\ul)$. Moreover, $w_0(0,0)=0$.
  \end{description}
  A weight function $w$ is called exact if there exists a function $f:\Lambda\ra \Q$ with $f(0)=0$ such that for all $g$ and $\lambda$,
  $w_g(\ul)=\sum_i f(\lambda_i)$.

  One can give similar definitions for genus $0$ modular and ribbon functors.
\end{definition}

\begin{remark}
  We can always extend a weight function to cobordism notation by $w_g(\umu;\ul):=w_g(\umu^\dagger,\ul)-\sum_iw_0(\mu_i,\mu_i^\dagger)$. 
\end{remark}

\begin{remark}
  If the involution $\lambda\mapsto\lambda^\dagger$ is the identity, any weight function is exact,
  as one can in this case choose the function $f$
  to be such that $f(\lambda)$ is half the weight of the Hodge structure on $\Nu_0(\lambda,\lambda^\dagger)$.
\end{remark}

We will deduce the following from Ocneanu rigidity. See \Cref{subocneanu} for the proof.

\begin{proposition}\label{exactweightfunctions}
  Let $\Nu$ be a complex modular, genus $0$ modular, ribbon or braided functor.
  Then any weight function on $\Nu$ is exact.
\end{proposition}

By \Cref{remarklocalsystemsconnections}, over $R=\C$, we can view the local systems $\Nu_g(\lambda_1,\dotsc,\lambda_n)$ of a modular functor $\Nu$
as vector bundles with flat connections. Hence we may talk of CVHS on the bundles $\Nu_g(\lambda_1,\dotsc,\lambda_n)$.

\begin{definition}[Complex Hodge structure]\label{definitionCVHSmodular}
  Let $\Nu$ be a complex modular functor. A complex Hodge structure on $\Nu$ is the data, for each $g$, $n$ and $\ul\in\Lambda^n$,
  of a CVHS on $\Nu_g(\lambda_1,\dotsc,\lambda_n)$, such that all gluing, vacuum and permutation isomorphisms
  in \Cref{definitiongeometricnonanomalousmodularfunctor} are isomorphisms of CVHS. We also require that $\Nu_0(0,0)=\Nu_0(0,0)^{0,0}$ is trivial of weight $0$.

  Similarly, one defines the notions of complex Hodge structure on a genus $0$ modular functors and on a ribbon or braided functor, where,
  for ribbon and braided functors,
  we omit the assumption on $\Nu_0(0,0)$ as it is not defined in that case.

  To a Hodge structure we assign the weight function $w$ on $\Nu$ which sends a pair $(g,\ul)$ with $\Nu_g(\ul)\neq 0$
  to the weight of the Hodge structure on $\Nu_g(\ul)$.
\end{definition}

\begin{remark}
  The definition above is independent of the choice of the level $(r,s)$ for $\Nu$, as for $r\mid r'$ and $s\mid s'$,
  there are compatible finite maps $\Mgrbt{g}{n}{r'}{s'}\ra\Mgrbt{g}{n}{r}{s}$.
\end{remark}

\begin{definition}[Shifts]
  Let $\Lambda$ be a finite set. Then the space of shifts $\Shift$ is the set of maps $\delta:\Lambda\ra\Q^2$ with $\delta(0)=(0,0)$.
  Its subspace of weight $0$ shifts $\Shiftz$ is the set of such $\delta$ with image in $\{(p,-p)\}\subset\Q^2$.
\end{definition}

\begin{proposition}[Action of shifts]
  Let $\Nu$ be a complex modular, genus $0$ modular or ribbon functor, with set of colors $\Lambda$.
  Then $\Shift$ acts on the set of complex Hodge structures on $\Nu$,
  and the action of its subspace $\Shiftz$ preserves weight functions.
\end{proposition}

\begin{proof}
  We describe the case of modular functors. The others are similar.
  For a Hodge structure $E=\bigoplus E^{p,q}$ and $x,y\in\Q$ define the translation $E[x,y]$
  by $E[x,y]^{p,q}=E^{p+x,q+y}$. Let $\delta\in\Shift$. Then for any Hodge structure $\Nu^{1}$ on $\Nu$, we can define a Hodge structure
  $\Nu^{2}=\delta\cdot\Nu^{1}$ by:
  \begin{equation*}
    \Nu^2_g(\lambda_1,\dotsc,\lambda_n)=\Nu^1_g(\lambda_1,\dotsc,\lambda_n)[\sum_i\delta(\lambda_i)].
  \end{equation*}
\end{proof}


\subsection{Existence and uniqueness}

\begin{theorem}[Existence and uniqueness]\label{maintheoremcomplex}
  Let $\Nu$ be a complex modular, genus $0$ modular, ribbon or braided functor. Denote by $\Lambda$ its set of colors.
  Then for any weight function $w$, $\Nu$ admits a complex Hodge structure whose weights are given by $w$.
  Moreover, $\Shift$ acts transitively on the set of complex Hodge structures on $\Nu$.
\end{theorem}

\begin{corollary}[RVHS on conformal blocks]\label{rationalqreps}
  Let $\Nu$ be as in \Cref{maintheoremcomplex}.
  Then for any $g,n\geq 0$ tangent-stable and $\ul\in\Lambda^n$,
  the local system $\Nu_g(\ul)$ over $\Mgrbt{g}{n}{r}{s}$ can be defined over a CM number field $L$
  and supports a rational variation of Hodge structures over $L$,
  which can be assumed to have weight $m$ in $\Z$, Hodge bidegrees in $\Z^2$, and be strongly polarized over $L$ (see \Cref{definitionRVHS}).
\end{corollary}

The following corollary is the main motivation for this article, as it explains how to extend genus $0$ knowledge to higher genus.

\begin{corollary}\label{uniqueextension}
  Let $\Nu$ be a complex modular functor.
  Then any complex Hodge structure on the braided functor
  \footnote{which is a genus $0$ structure!}
  associated to $\Nu$
  extends uniquely to the whole modular functor. In particular, if the Hodge numbers are known for the associated braided functor,
  then they are known for the modular functor.
\end{corollary}
For a description of how to actually compute the Hodge numbers in all genera from the genus $0$ Hodge numbers, see \Cref{theoremhodgenumbers} below.
For formulas in the $\SU$ case, see \Cref{theoremhodgenumberssu}.


\subsection{The Frobenius algebra governing the Hodge numbers}\label{subfrobenius}

\begin{definition}[Frobenius algebra]
  A Frobenius algebra on a ring $A$ is a finite free\footnote{i.e. finite free as an $A$-module.}
  algebra $V$ together with a linear map $\epsilon: V\ra A$ such that the bilinear form
  $\eta:a\otimes b\mapsto \epsilon(ab)$ is non-degenerate.
\end{definition}

\begin{remark}
  The notion of a Frobenius algebra is equivalent to that of $(1+1)$-dimensional TFT, see \cite{kockFrobeniusAlgebras2D} for details.
\end{remark}

\begin{notation}
  We will denote by $\Zuv$ the ring of polynomials in rational powers of $u$ and $v$ with coefficients in $\Z$.
\end{notation}

\begin{definition}[Hodge polynomial]
  If $E$ is a Hodge structure with decomposition $E=\bigoplus_{p+q=m}E^{p,q}$, then we define
  the Hodge polynomial of $E$ by:
  \begin{equation*}
    e(E)=\sum_{p+q=m}\dim (E^{p,q})u^pv^q\in\Zuv.
  \end{equation*}
\end{definition}

\begin{definition}\label{definitionquantumproduct}
  Let $\Nu$ be a ribbon functor over $\C$ with a Hodge structure. Denote by $\Lambda$ its set of colors.
  Define $V$ as the free $\Zuv$-module with basis $\Lambda$.
  Let $\star$ be the symmetric bilinear operation on $V$ defined by:
  \begin{equation*}
    [\lambda]\star [\mu] = \sum_{\nu\in\Lambda} e(\Nu(\nu;\lambda,\mu))[\nu]
  \end{equation*}
  and let $\epsilon: V\ra \Zuv$ be the linear map such that $\epsilon([0])=1$ and 
  $\epsilon([\lambda])=0$ for $\lambda\neq 0$. 
\end{definition}

\begin{proposition}
  The operation $\star$ is associative and defines on $V$ the structure of a commutative $\Zuv$-algebra.
  Moreover $(V,\epsilon)$ is a Frobenius algebra.
\end{proposition}

We will sometimes omit $\star$ when writing products in $V$.

\begin{proof}
  Consider the gluing axiom on $\Mgrb{0}{4}{r}$, where the output is the point $4$. Let $C_{(12)3}$ be the nodal curve where point $2$ is on the same
  component as point $1$ and $C_{1(23)}$ the one where $2$ is with $3$. Then by the gluing axiom (in cobordism notation):
  \begin{align*}
    \Nu(\nu;\lambda_1,\lambda_2,\lambda_3)_{\mid C_{1(23)}}&\simeq \bigoplus_\delta \Nu(\nu;\lambda_1,\delta)\otimes \Nu(\delta;\lambda_2,\lambda_3), \\
    \Nu(\nu;\lambda_1,\lambda_2,\lambda_3)_{\mid C_{(12)3}}&\simeq \bigoplus_\delta \Nu(\nu;\delta,\lambda_3)\otimes \Nu(\delta;\lambda_1,\lambda_2). \\
  \end{align*}
  As these are isomorphisms of Hodge structures, they have the same Hodge numbers. This implies associativity of the product.
  Commutativity results from the \textbf{(Perm)} axiom of \Cref{definitiongeometricnonanomalousmodularfunctor}.
  Non-degeneracy of $\eta$ is ensured by the fact that for each $\lambda$, $e(\Nu(0;\lambda,\lambda^\dagger))$ is a monomial and hence invertible.
\end{proof}

\begin{theorem}[Hodge numbers]\label{theoremhodgenumbers}
  Let $\Nu$ be a complex modular functor with Hodge structure. Denote by $V$ the associated Frobenius algebra,
  and $\eta^{-1}\in V^{\otimes 2}$ the inverse of its bilinear form. We will use the notation:
  \begin{equation*}
    \Omega=\star(\eta^{-1})=\sum_{\lambda\in\Lambda}\epsilon([\lambda]\star[\lambda^\dagger])^{-1}[\lambda]\star[\lambda^\dagger]\in V.
  \end{equation*}
  Let $g\geq 0$ and $\lambda_1,\dotsc,\lambda_n\in\Lambda$. Then the Hodge numbers are computed by:
  \begin{equation*}
    e(\Nu_g(\lambda_1,\dotsc,\lambda_n))=\epsilon([\lambda_1]\star\dotsb \star[\lambda_n]\star\Omega^{\star g}).
  \end{equation*}
  Similarly, for $\Nu$ a ribbon or braided functor with Hodge structure, the Hodge polynomial $e(\Nu(\mu;\lambda_1,\dotsc,\lambda_n))$ is equal to
  the coefficient of $[\mu]$ in $[\lambda_1]\star\dotsb \star[\lambda_n]$.
\end{theorem}

The idea that Hodge numbers of braided functors could be organized in an algebra appeared independently in the recent paper \cite{belkaleMotivicFactorisationKZ2023}
of Belkale, Fakhruddin and Mukhopadhyay. See \Cref{subintroBFM,subsectionlie} for more comments on their paper.

\begin{remark}\label{remarkhodgenumbers}
  With $\Nu$ modular as in \Cref{theoremhodgenumbers}, consider curves of genus $g$ with $n$ inputs and $1$ output. Then:
  \begin{equation*}
    \sum_\mu e(\Nu_g(\mu;\lambda_1,\dotsc,\lambda_n))[\mu]=[\lambda_1]\star\dotsb\star[\lambda_n]\star\Omega^{\star g}.
  \end{equation*}
\end{remark}

\begin{proof}
  Let us first prove \Cref{remarkhodgenumbers} by induction. For $g=0$ and $n=2$, this is the definition of the product.
  For $g=1$ and $n=0$, we compute $\sum_{\mu}e(\Nu_1(\mu;))$ on the nodal curve and we get $\Omega$. For any other pair $(g,n)$,
  we can compute on a curve $C$ with $2$ or $3$ components as in \Cref{figurecurve}, $1$ of which, $C_1$, is of genus $0$ with $3$ marked points/nodes,
  one of which is the output. Say $C_i$ has genus $g_i$, the inputs of $C_2$ are marked $\lambda_1,\dotsc,\lambda_{n'}$,
  and the inputs of $C_3$ are marked $\lambda_{n'+1},\dotsc,\lambda_n$. Then, by the gluing axiom, the case $n=2$ and induction hypothesis, we have:
  \begin{align*}
    \sum_\mu e(\Nu_g(\mu;\ul))[\mu] &= (\sum_{\mu_2}e(\Nu_g(\mu_2;\lambda_1,\dotsc,\lambda_{n'}))[\mu_2])
    \star(\sum_{\mu_3}e(\Nu_g(\mu_3;\lambda_{n'+1},\dotsc,\lambda_{n}))[\mu_3]) \\
    &= ([\lambda_1]\star\dotsb\star[\lambda_{n'}]\star\Omega^{\star g_2})\star([\lambda_{n'+1}]\star\dotsb\star[\lambda_{n}]\star\Omega^{\star g_3}) \\
    &= [\lambda_1]\star\dotsb\star[\lambda_n]\star\Omega^{\star g}.
  \end{align*}
  The first statement of the theorem then results directly from \Cref{remarkhodgenumbers} for $\mu=0$ and the vacuum axiom.
\end{proof}

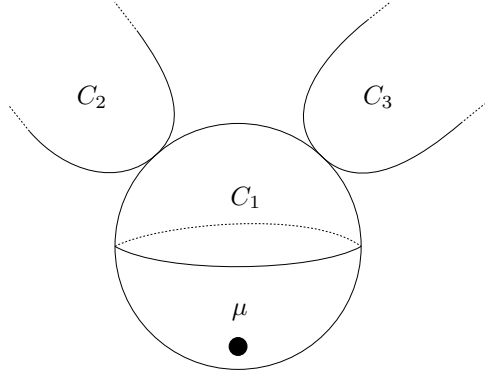
\begin{figure}
  \def\svgwidth{0.5\linewidth}
\begingroup%
  \makeatletter%
  \providecommand\color[2][]{%
    \errmessage{(Inkscape) Color is used for the text in Inkscape, but the package 'color.sty' is not loaded}%
    \renewcommand\color[2][]{}%
  }%
  \providecommand\transparent[1]{%
    \errmessage{(Inkscape) Transparency is used (non-zero) for the text in Inkscape, but the package 'transparent.sty' is not loaded}%
    \renewcommand\transparent[1]{}%
  }%
  \providecommand\rotatebox[2]{#2}%
  \newcommand*\fsize{\dimexpr\f@size pt\relax}%
  \newcommand*\lineheight[1]{\fontsize{\fsize}{#1\fsize}\selectfont}%
  \ifx\svgwidth\undefined%
    \setlength{\unitlength}{366.22730165bp}%
    \ifx\svgscale\undefined%
      \relax%
    \else%
      \setlength{\unitlength}{\unitlength * \real{\svgscale}}%
    \fi%
  \else%
    \setlength{\unitlength}{\svgwidth}%
  \fi%
  \global\let\svgwidth\undefined%
  \global\let\svgscale\undefined%
  \makeatother%
  \begin{picture}(1,0.77395065)%
    \lineheight{1}%
    \setlength\tabcolsep{0pt}%
    \put(0,0){\includegraphics[width=\unitlength,page=1]{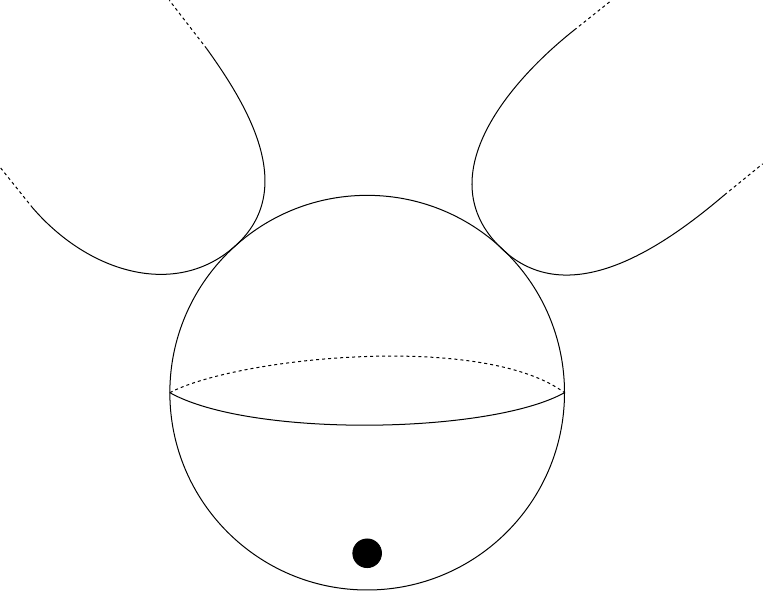}}%
    \put(0.46810145,0.11601982){\color[rgb]{0,0,0}\makebox(0,0)[lt]{\lineheight{1.25}\smash{\begin{tabular}[t]{l}$\mu$\end{tabular}}}}%
    \put(0.46551141,0.34741401){\color[rgb]{0,0,0}\makebox(0,0)[lt]{\lineheight{1.25}\smash{\begin{tabular}[t]{l}$C_1$\end{tabular}}}}%
    \put(0.14213756,0.55934539){\color[rgb]{0,0,0}\makebox(0,0)[lt]{\lineheight{1.25}\smash{\begin{tabular}[t]{l}$C_2$\end{tabular}}}}%
    \put(0.74386903,0.55934539){\color[rgb]{0,0,0}\makebox(0,0)[lt]{\lineheight{1.25}\smash{\begin{tabular}[t]{l}$C_3$\end{tabular}}}}%
  \end{picture}%
\endgroup%

  \caption{The nodal curve $C$ with $2$ or $3$ components in the proof of \Cref{theoremhodgenumbers}.
  If $C$ has only $2$ components, replace $C_3$ by a marked point.}
  \label[figure]{figurecurve}
\end{figure}


\subsection{The Cohomological Field Theory of a Hodge structure}

The main reference for the definition of a CohFT is \cite[0.5]{pandharipandeRelations$overlinemathcalM_gn$3spin2015a}.

\begin{definition}[CohFT with unit]
  Let $V$ be a finite free module over a $\Q$-algebra $A$ with a chosen element $1\in V$ and non-degenerate pairing $\eta:V\otimes V\ra A$.
  Denote the inverse of the pairing by $\eta^{-1}=\sum_i a_i\otimes b_i\in V\otimes V$.
  A Cohomological Field Theory (CohFT) with unit on $V$ is a family of maps:
  \begin{equation*}
    \omega_{g,n}:V^{\otimes n}\ra H^*(\oMg{g}{n};\Q)\otimes A
  \end{equation*}
  for $3g-3+n>0$, each equivariant for the action of $S_n$ and satisfying the gluing and vacuum rules:
  \begin{description}
    \item[(G-sep)] For each gluing map $q:\oMg{g_1}{n_1+1}\times\oMg{g_2}{n_2+1}\ra \oMg{g}{n}$,
    we have:
    \begin{equation*}
      q^*\omega_{g_1+g_2,n_1+n_2}(v_1\otimes\dotsb \otimes v_n)=\sum_i\omega_{g_1,n_1}(v_1\otimes\dotsb \otimes v_{n_1}\otimes a_i)\otimes
      \omega_{g_2,n_2}(v_{n_1+1}\otimes\dotsb \otimes v_{n}\otimes b_i).
    \end{equation*}
    \item[(G-nonsep)] For each gluing map $p:\oMg{g-1}{n+2}\ra \oMg{g}{n}$,
    we have:
    \begin{equation*}
      p^*\omega_{g,n}(v_1\otimes\dotsb \otimes v_n)=\sum_i\omega_{g-1,n+2}(v_1\otimes\dotsb \otimes v_{n}\otimes a_i\otimes b_i).
    \end{equation*}
    \item[(V)] For each forgetful map $f:\oMg{g}{n+1}\ra\oMg{g}{n}$ we have:
    \begin{equation*}
      f^*\omega_{g,n}(v_1\otimes\dotsb \otimes v_n)=\omega_{g,n+1}(v_1\otimes\dotsb \otimes v_n\otimes 1).
    \end{equation*}
  \end{description}
  The topological part of $\omega$ is the family of maps $\omega^0_{g,n}:V^{\otimes n}\ra A$ induced by the restriction to cohomological degree $0$.
\end{definition}

\begin{proposition}
  For a CohFT $\omega$ on $V$, its topological part induces the structure of a Frobenius algebra $(\star,\epsilon)$ on $V$ by the formulas:
  \begin{equation*}
    \eta(v_1\star v_2,v_3)=\omega^0_{0,3}(v_1\otimes v_2\otimes v_3)\text{ and }\epsilon(v)=\eta(v\otimes 1).
  \end{equation*}
\end{proposition}

In our situation, we get directly the following Theorem from the definitions of a CohFT, and of a Hodge structure on a modular functor.
Notice that as $\oMg{g}{n}$ and $\Mgrbt{g}{n}{r}{s}$ have the same coarse space, they have the same rational cohomology.

\begin{theorem}\label{cohft}
  Let $\Nu$ be a modular functor with a Hodge structure.
  Let $\Lambda$ be its set of colors and let $V_\Q$ be the free $\Quv$-module with basis $\Lambda$.
  Then the family of maps:
  \begin{equation*}
    \omega_{g,n}:\begin{cases}
      V_\Q^{\otimes n} &\lra H^*(\oMg{g}{n};\Q)\otimes \Quv \\
      [\lambda_1]\otimes \dotsb \otimes [\lambda_n] &\longmapsto \sum_{p,q}\ch(\Nu_g(\lambda_1,\dotsc,\lambda_n)^{p,q})u^pv^q
    \end{cases}
  \end{equation*}
  form a Cohomological Field Theory, whose associated Frobenius algebra agrees with that of \Cref{definitionquantumproduct} over $\Quv$.
\end{theorem}

Here $\ch(E)$ denotes the Chern character of a vector bundle $E$.


\subsection{Polarizations of Hodge structures on modular functors}

\begin{definition}\label{definitionprolarizedHSonMF}
  Let $\Nu$ be a complex modular or genus $0$ modular functor with a Hodge structure.
  A polarization of its Hodge structure is the data of a flat Hermitian metric $h_g(\ul)$ on each bundle $\Nu_g(\ul)$,
  which polarizes the CVHS on $\Nu_g(\ul)$ (see \Cref{definitionCVHS}),
  and such that all gluing, permutation, vacuum and normalization isomorphisms are isomorphisms of Hermitian bundles
  \footnote{$\Nu_0(0,0)\simeq \underline{\C}$ is thus assumed to have the trivial Hermitian form $(a,b)\mapsto a\overline{b}$.}.

  A similar definition can be provided for complex ribbon and braided functors.
\end{definition}

We can define the notion of a polarization for a functor defined over a CM number field $L$ by asking for forms $h_g(\ul)$
defined over $L$ that are polarizations for every embedding $L\hookrightarrow \C$.

\begin{proposition}\label{DeroinMarche}
  Let $\Nu$ be complex modular functor with a Hodge structure.
  Assume that every index $(p,q)\in\Q^2$ appearing in a non-zero term of the Hodge structure is such that $p$ and $q$ have odd denominators in reduced form.
  Let $h$ be a polarization of the Hodge structure which has sign $(-1)^p$ on each block of index $(p,q)$.

  Then $h$ makes $\Nu$ a Hermitian modular functor and the CohFT over $\Quv$ associated to its Hodge structure by \Cref{cohft}
  can be specialized at $u=-1$ and $v=1$ to obtain the Deroin-Marché CohFT over $\Q$ associated to $h$ (\cite[4.3]{deroinToledoInvariantsTopological2022}).
  In particular, the Frobenius algebra of \Cref{definitionquantumproduct} specializes to the Frobenius algebra computing signatures of $h$
  (\cite[5.1]{deroinToledoInvariantsTopological2022}).
\end{proposition}

\begin{remark}
  The hypothesis that the sign is $(-1)^p$ on blocks of index $(p,q)$ is rather arbitrary. All we need is a coherent formula depending only on $p$ and $q$.
  See \Cref{thoeremgeometricsign} below for an example coming from geometry.
\end{remark}

The specialization of $u$ to $-1$ is possible as no powers of $u$ with even denominators appear in the decomposition.
The proof of \Cref{DeroinMarche} is essentially the fact that the Deroin-Marché construction of the CohFT is the same as ours replacing Hodge decompositions with
signature decompositions.


\section{\texorpdfstring{Application to $\SU$ modular functors and modular functors associated to Lie algebras}
{Application to SU(2) modular functors and modular functors associated to Lie algebras}}\label{sectionsu}


\subsection{\texorpdfstring{The $\SU$ modular functors}{The SU(2) modular functors}}

In this section, we use constructions developed in the first article \cite{godfardConstructionHodgeStructures2024} of this series of two
to compute the Hodge numbers explicitly for $\SU$ modular functors.
We briefly remind the results we will need from that paper.

\begin{proposition}\label{factsSU}
  For each $r\geq 3$ odd integer, the $\SU$ modular functor of level $2r$ is defined over the cyclotomic field $\Q(\zeta_{4r})$,
  has set of colors $\Lambda=\{0,1,2,\dotsc,r-2\}$
  and has a Hermitian structure $h$ (coming from the Skein construction).
  Moreover, whenever $\sum_i\lambda_i$ is odd, for all $g$, $\Nu_g(\ul)=0$.
\end{proposition}

The following theorem gives a geometric construction of the braided functor associated to $\Nusl$.
The notation $\Nuab$ is for the Abelian modular functor of level $2r$, see \cite[2.1.3]{godfardConstructionHodgeStructures2024}.

\begin{theorem}[{\cite[3.3]{godfardConstructionHodgeStructures2024}}]\label{theoremgeometricconstruction}
  Let $r\geq 3$ be an odd integer, $n\geq 2$ and $b,a_1,\dotsc,a_n\in\{0,1,\dotsc,r-2\}$,
  such that $m=\frac{a_1+\dotsb+a_n-b}{2}$ is an integer.
  Consider the $\Q(\zeta_r)$-local system $\Line$ 
  over $\sMrpo{n+m}{r}/S_m$ corresponding to the antisymmetric part of $\Nuab(b;a_1,\dotsc,a_n,-2,\dotsc,-2)$ for the $S_m$ action.
  Let us consider the forgetful map:
  \begin{equation*}
      p:\sMrpo{n+m}{r}/S_m\lra \sMrpo{n}{r}.
  \end{equation*}
  Then there is an isomorphism:
  \begin{equation*}
      \Nusl(b;a_1,\dotsc,a_n) \simeq \mathrm{im}\left(R^mp_!\Line\lra R^mp_*\Line\right).
  \end{equation*}
  In particular, $\Nusl(b;a_1,\dotsc,a_n)$ has a Hodge structure of weight $m$ with only $(p,q)$ indices in $\{(p,q)\mid p,q\in\N, p+q=m\}$.
\end{theorem}

From this fact and computations we proved that gluing maps are compatible with these geometric Hodge structures
(\cite[3.18]{godfardConstructionHodgeStructures2024}).
This provides the following.

\begin{theorem}
  Let $r\geq 3$ be odd. The braided functor associated to $\Nusl$ supports a rational Hodge structure such that each
  $\Nusl(b;a_1,\dotsc,a_n)$ has weight $m=\frac{a_1+\dotsb+a_n-b}{2}$ and all Hodge bidegrees $(p,q)$ are pairs of non-negative integers.
\end{theorem}

By \Cref{uniqueextension}, we get the following.

\begin{corollary}\label{corollaryhodgesu}
  Let $r\geq 3$ be odd, then the modular functor $\Nusl$ supports a rational Hodge structure with non-negative integer 
  Hodge bidegrees and weight function $w_g(a_1,\dotsc,a_n)=\sum_ia_i/2$.
\end{corollary}

Let us now turn to the computations of the Frobenius algebra governing the Hodge numbers.
We use the Hermitian structure $h$ to compute them. To that end, in \cite{godfardConstructionHodgeStructures2024},
we computed the signs of $h$ on the Hodge decomposition.

\begin{theorem}[{\cite[3.9, 3.12]{godfardConstructionHodgeStructures2024}}]\label{thoeremgeometricsign}
  Let $r\geq 3$ be an odd integer, $n\geq 2$ and $b,a_1,\dotsc,a_n\in\{0,1,\dotsc,r-2\}$,
  such that $m=\frac{a_1+\dotsb+a_n-b}{2}$ is an integer.
  Let $s_{\Dc}$ be the intersection form on $\Nusl(b;a_1,\dotsc,a_n)$ coming from the geometric construction
  and let $h_{\Dc}$ be the Hermitian form on $\Nusl(b;a_1,\dotsc,a_n)$ coming from the Skein module construction ($h$ above).
  Then:
  \begin{equation*}
      s_{\Dc}=(\zeta_r^2-\zeta_r^{-2})^mh_{\Dc}
  \end{equation*}
  and $(-1)^{\frac{m(m-1)}{2}}s_{\Dc}$ is a polarization for the Hodge structure on $\Nusl(b;a_1,\dotsc,a_n)$.
  In particular, if $\Nusl(b;a_1,\dotsc,a_n)=\bigoplus_{p+q=m}E^{p,q}$ is the Hodge decomposition,
  then it is orthogonal for $h_{\Dc}$
  and $(-1)^{\frac{m(m-1)}{2}}i^{p-q}(\zeta_r^2-\zeta_r^{-2})^mh_{\Dc}$ is positive definite on $E^{p,q}$.
\end{theorem}

\begin{proposition}\label{signaturessu}
  Let $0<s<r$ be prime to $r$ and odd. Then for $\zeta_r=-e^{i\pi s/r}$ and $1\leq k\leq r-2$,
  $h$ is definite of sign $-(-1)^{\lfloor\frac{ks}{r}\rfloor+\lfloor\frac{(k+1)s}{r}\rfloor}$ on $\Nu_0(k-1;k,1)$.
\end{proposition}

For a reference see, for example, \cite[3]{marcheSignaturesTQFTsTrace2023}.
We now extend the description of signatures given by Marché in \cite[3]{marcheSignaturesTQFTsTrace2023} to Hodge numbers.

\begin{theorem}\label{theoremhodgenumberssu}
  Let $r\geq 3$ be odd and $0<s<r$ be odd and prime to $r$. Define, for $1\leq k\leq r-2$:
  \begin{equation*}
    w_k = \begin{cases} u &\text{if }\lfloor\frac{2s}{r}\rfloor+\lfloor\frac{ks}{r}\rfloor+\lfloor\frac{(k+1)s}{r}\rfloor\text{ is even,}\\
       v&\text{otherwise.}\end{cases}
  \end{equation*}
  Let $M$ be the following square matrix of size $r-1$ with coefficients in $\Z[u,v]$:
  \begin{equation*}
    M = \begin{pmatrix}
      0 & w_1     &           &           &     \\
      1 & 0       & w_2       &           &     \\
        & 1       & 0         & \ddots    &      \\
        &         & \ddots    & \ddots    & w_{r-2}   \\
        &         &           & 1         & 0
    \end{pmatrix}
  \end{equation*}
  and for $0\leq k\leq r-2$, let $M_k$ be its square top left submatrix of size $k$.
  Then the Frobenius algebra of \Cref{definitionquantumproduct} associated to the Hodge structure of \Cref{corollaryhodgesu} on $\Nusl$
  for $\zeta_r=-e^{i\pi s/r}$ is defined over $\Z[u^{\pm 1},v^{\pm 1}]$ by:
  \begin{equation*}
    V = \Z[u^{\pm 1},v^{\pm 1}][X]/\chi_M(X),\text{ and for }\:0\leq k\leq r-2,\;[k]=\chi_{M_k}(X)
  \end{equation*}
  where $\chi_M(X)=\mathrm{det}(XI_{r-1}-M)$ denotes the characteristic polynomial of $M$.
  Moreover specializing at $(u,v)=(-1,1)$ if $\lfloor\frac{2s}{r}\rfloor$ is even or $(u,v)=(1,-1)$ if $\lfloor\frac{2s}{r}\rfloor$ is odd
  gives the Frobenius algebra computing signatures (see \cite[3]{marcheSignaturesTQFTsTrace2023}).
\end{theorem}

Note that $[1]=X$ is a generator of $V$. For details on how to extract the Hodge numbers from $V$, see \Cref{theoremhodgenumbers}.

In their recent paper \cite{belkaleMotivicFactorisationKZ2023}, Belkale, Fakhruddin and Mukhopadhyay independently computed the same rings $V$
by computing coefficients $a_k$ and $b_k$ (defined below) with a different method (see \cite[Th. 1.11]{belkaleMotivicFactorisationKZ2023}).
See \Cref{subsectionlie} below for more comments on their work.

\begin{proof}
  Let $V$ be the Frobenius algebra of \Cref{definitionquantumproduct}, which is defined over $\Z[u^{\pm 1},v^{\pm 1}]$ by \Cref{corollaryhodgesu}.
  By looking at dimensions, one sees that, with the convention $[r-1]=0$, for $1\leq k\leq r-2$, $[k][1]=a_k[k+1]+b_k[k-1]$ for some coefficients
  $a_k$ and $b_k$. As, for $k<r-2$, $\Nu_0(k+1;k,1)$ is geometric of dimension $1$ and weight $m=0$, its Hodge polynomial is $1$. Hence $a_k=1$.
  Now $\Nu_0(k-1;k,1)$ is geometric of dimension $1$ and weight $m=1$, so its Hodge polynomial is either $u$ or $v$.
  If it is $u$, then, by \Cref{thoeremgeometricsign}, $h$ has the sign of $i(\zeta_r^2-\zeta_r^{-2})$ on $\Nu_0(k-1;k,1)$.
  If it is $v$, the sign is opposite.
  However, by \Cref{signaturessu}, the sign is $-(-1)^{\lfloor\frac{ks}{r}\rfloor+\lfloor\frac{(k+1)s}{r}\rfloor}$.
  As $i(\zeta_r^2-\zeta_r^{-2})$ and $-(-1)^{\lfloor\frac{2s}{r}\rfloor}$ have the same sign, we have $b_k=w_k$.

  Notice that $[1]$ generates $V$ over $\Z[u^{\pm 1},v^{\pm 1}]$ and that $[0],[1],\dotsc,[r-2]$ form a basis.
  Now $M$ is the matrix of multiplication by $[1]$ in this basis, hence the result.
\end{proof}


\subsection{Modular functors associated to simple Lie Algebras}\label{subsectionlie}

In this subsection, we explain how under some assumption, one could deduce from the recent work \cite{belkaleMotivicFactorisationKZ2023} of Belkale, Fakhruddin and Mukhopadhyay
the Hodge numbers in all genera of modular functors associated to $\mathfrak{sl}_n$, $n\geq 2$.

Given a complex finite dimensional simple Lie algebra $\mathfrak{g}$ and an integer $\ell\geq 1$, one can define a family of Galois conjugate modular functors.
More precisely, if $m=\frac{(\alpha_l,\alpha_l)}{(\alpha_s,\alpha_s)}\in\{1,2,3\}$ with $\alpha_l$ a long root and $\alpha_s$ a short root,
and $h^\vee$ is the dual Coxeter number of $\mathfrak{g}$, then for each choice of $2m(\ell+h^\vee)$-th root of unity $\zeta$ we have a complex modular functor
$\Nu_{\mathfrak{g},\zeta}$. Its set of colors is $P_\ell=\{\lambda\in P_+\mid (\lambda,\theta)\leq \ell\}$ where $P_+$ is the set of dominant integral weights,
$\theta$ is the highest root, and $(\cdot,\cdot)$ is normalized so that $(\theta,\theta)=2$. See \cite[chp. 7]{bakalovLecturesTensorCategories2000} for details.

\begin{remark}
  For $r\geq 3$ odd and $\zeta$ a primitive $r$-th root of unity, the $\SU$ modular functor $\Nu_\zeta$ of level $2r$ has the same braided tensor category as 
  $\Nu_{\mathfrak{sl}_2,\zeta}$. However, the twists differ by a sign on each odd $\lambda\in\Lambda=P_r=\{0,1,\dotsc,r-2\}$.
  This slight difference is of little importance here as twists have finite order, so that $\Nu_\zeta$ and $\Nu_{\mathfrak{sl}_2,\zeta}$ have the same Hodge decompositions in genus $0$
  and hence the same Hodge numbers in all genera.
\end{remark}

In \cite{belkaleMotivicFactorisationKZ2023}, for any $\mathfrak{g}$ and $\zeta$ as above, Belkale, Fakhruddin and Mukhopadhyay
study the Hodge theory of a geometrically constructed braided functor $\widetilde{\Nu}_{\mathfrak{g},\zeta}$ whose quantum representations
coincide with those of $\Nu_{\mathfrak{g},\zeta}$.
In particular, they provide an algorithm for computing the Hodge numbers when $\mathfrak{g}$ is $\mathfrak{sl}_n$.

\begin{conjecture}\label{conjectureBFM}
  The braided functor associated to $\Nu_{\mathfrak{g},\zeta}$ is isomorphic to $\widetilde{\Nu}_{\mathfrak{g},\zeta}$.
  Equivalently, $\Nu_{\mathfrak{g},\zeta}$ and $\widetilde{\Nu}_{\mathfrak{g},\zeta}$ have isomorphic braided tensor categories.
\end{conjecture}

To prove this Conjecture, essentially, one needs to identify gluing isomorphisms of both braided functors.
Assuming the Conjecture, we can, as in the case of $\SU$ modular functors, apply unicity of the Hodge structure to get the following result.

\begin{proposition}\label{propositionBFM}
  Let $n\geq 2$, $\ell\geq 1$ and $\zeta$ be a $2(\ell+n)$-th root of unity. Then if $\Nu_{\mathfrak{sl}_n,\zeta}$ satisfies \Cref{conjectureBFM},
  the algorithm \cite[1.11]{belkaleMotivicFactorisationKZ2023}
  of Belkale, Fakhruddin and Mukhopadhyay computes all Hodge numbers of $\Nu_{\mathfrak{sl}_n,\zeta}$ (for any genus $g$).
\end{proposition}

\subsection{Gaps in Hodge decompositions}\label{subsectiongaps}

For simplicity this section is written for modular functors.
But, with the exception of \Cref{propositionremovinggaps}, all definitions and statements extend to genus $0$ modular functors, ribbon and braided functors.

Hodge numbers can detect reducibility in the following way.

\begin{definition}
  Let $(E,\nabla)$ be a flat bundle with a CVHS $E=\bigoplus_{p+q=m}E^{p,q}$.
  For $x\in\Q/\Z$, we will use the notation:
  $$E^x=\bigoplus_{\substack{p+q=m\\p\in x}}E^{p,q}.$$
  Each $E^x$ is preserved by the connection.
  We say that the CVHS has a \textit{gap of type $1$} when the decomposition $E=\bigoplus_{x\in\Q/\Z}E^x$ is non trivial.

  Fix some $x\in\Q/\Z$ and $p\in x$ such that $E^{p,m-p}=0$. Then we have a decomposition $E^x=E^{x,>p}\oplus E^{x,<p}$
  preserved by the connection, where:
  $$E^{x,>p}=\bigoplus_{\substack{p'>p\\p'\in x}}E^{p',m-p'}\text{ and } 
  E^{x,<p}=\bigoplus_{\substack{p'<p\\p'\in x}}E^{p',m-p'}.$$
  When there exists a choice of $x$ and $p$ such that the decomposition $E^x=E^{x,>p}\oplus E^{x,<p}$ is non-trivial,
  we say that the CVHS has a \textit{gap of type $2$}.

  A decomposition $E=\bigoplus_{i\in I}E_i$ preserved by the connection \textit{comes from gaps in the CVHS}
  if it is a combination of decompositions of type $1$ and/or $2$ as above.
\end{definition}

\begin{remark}
  A decomposition $E=\bigoplus_{i\in I}E_i$ coming from gaps is always induced by the Hodge decomposition, in the sense that each $E_i$
  is a sum of Hodge summands $E^{p,q}$. However, in general, not all decompositions induced by the Hodge decomposition and preserved by the connection
  come from gaps. Those that come from gaps are exactly those that can be observed by looking only at the Hodge numbers.
\end{remark}

\begin{proposition}\label{propositiongaps}
  Let $\Nu$ be a complex modular functor with a complex Hodge structure. Fix $g,n\geq 0$ tangent stable and $\ul\in\Lambda^n$.
  For any boundary divisor $D\subset\Mgrbt{g}{n}{r}{s}$, denote by:
  \begin{equation*}
    \Nu_g(\ul)_{\mid D}=\bigoplus_{\mu\in\Lambda} \Nu_D(\ul,\mu,\mu^\dagger)
  \end{equation*}
  the corresponding gluing isomorphism. Now any decomposition $\Nu_g(\ul)=\bigoplus_{i\in I}\Nu_i$ coming from gaps is compatible
  with gluing decompositions in the sense that for all $i$ and $D$:
  \begin{equation*}
    \Nu_{i\mid D}=\bigoplus_{\mu\in\Lambda} \left(\Nu_{i\mid D}\cap\Nu_D(\ul,\mu,\mu^\dagger)\right).
  \end{equation*}
  As a corollary, any decomposition $\Nu_g(\ul)=\bigoplus_{i\in I}\Nu_i$ coming from gaps is preserved by the action of curve operators
  (see \Cref{definitioncurveoperator}).
\end{proposition}

\begin{proof}
  By the definition of a complex Hodge structure on a modular functor, the gluing decompositions of $\Nu_g(\ul)_{\mid D}$ above
  is a decomposition of CVHS. In particular, it must be compatible with any decomposition of $\Nu_g(\ul)$
  induced by the Hodge decomposition. As for curve operators, notice that any curve operator $\Op{\gamma}{\lambda}$ acts block-diagonally on the
  gluing decomposition associated to the {\scc} $\gamma$, and hence preserves any decomposition coming from gaps.
\end{proof}

For the general definitions of $\SO$ and $\SU$ modular functors, see \cite{blanchetTopologicalQuantumField1995}.
Koberda and Santharoubane prove the following for $\SU$ modular functors, but the proof also applies to $\SO$ modular functors.

\begin{proposition}[{\cite[3.2]{koberdaIrreducibilityQuantumRepresentations2018}}]\label{propositionramanujan}
  Let $\Nu$ be any $\SO$ or $\SU$ modular functor. Denote by $\Lambda$ its set of colors.
  Then for any $g,n\geq 0$ and $\ul\in\Lambda^n$, $\Nu_g(\ul)$ is irreducible for the action of curve operators.
\end{proposition}

Applying \Cref{propositiongaps}, we get the following.

\begin{corollary}\label{corollarynogaps}
  For any complex Hodge structure on a $\SO$ or $\SU$ modular functor $\Nu$, any $g,n\geq 0$ and $\ul\in\Lambda^n$,
  there are no decompositions of $\Nu_g(\ul)$ coming from gaps in the corresponding CVHS.
\end{corollary}

\begin{remark}
  It is known that at prime and $2$ times prime levels, the corresponding $\SO$/$\SU$ quantum representations are all irreducible
  (for $n=0$ this is \cite{robertsIrreducibilityQuantumRepresentations2001},
  for a proof in general, see for example the discussion in \cite[A.]{godfardRigidityFibonacciRepresentations2023}).
  However, at more general levels, little is known about irreducibility. \Cref{corollarynogaps} tells us that Hodge numbers will be of
  no help in investigating this question.
\end{remark}

Let us finish this subsection by explaining how absence of gaps of type $1$ can imply existence of complex Hodge structures
with \emph{integer} bidegrees.

\begin{proposition}\label{propositionremovinggaps}
  Let $\Nu$ be a modular functor with a complex Hodge structure $\Nu^1$ such that:
  \begin{description}
    \item[(1)] Its weight function has values in $\Z$;
    \item[(2)] For all $g,n\geq 0$ and $\ul\in\Lambda^n$, the CVHS on $\Nu_g(\ul)$ has no gaps of type $1$; 
    \item[(3)] For each $\lambda\in\Lambda$, there exists $g\geq 1$ such that $\Nu_g(\lambda)\neq 0$.
  \end{description}
  Then $\Nu$ admits a complex Hodge structure $\Nu^2$, all of whose bidegrees $(p,q)$ are in $\Z^2$.
\end{proposition}

\begin{proof}
  By the second assumption, we can define a function:
  \begin{equation*}
    w:\{(g,\ul)\mid \Nu_g(\ul)\neq 0\}\lra \Q/\Z
  \end{equation*}
  by asking that for all $g$, $\ul$, $\Nu_g(\ul)=\Nu_g(\ul)^{w_g(\ul)}$. This function satisfies the same axioms as a weight function,
  replacing $\Q$ by $\Q/\Z$ (see \Cref{definitionweightfunction}). In particular, for all $g$ and $\ul$,
  whenever $w_g(\ul)$ is defined, so is $w_{g+1}(\ul)$ and $w_{g+1}(\ul)=w_g(\ul)$.
  From $w$ define $f:\Lambda\ra \Q/\Z$ as follows: for $\lambda\in\Lambda$, choose $g\geq 1$ such that $\Nu_g(\lambda)\neq 0$,
  and set $f(\lambda)=w_g(\lambda)$. It is independent of the choice of $g$ and we have:
  \begin{equation*}
    \bm{(*)}\;\forall g,\ul,\;w_g(\ul)=\sum_if(\lambda_i).
  \end{equation*}
  Note that by normalization, $f(0)=0$. Choose a lift $\widetilde{f}:\Lambda\ra\Q$ of $f$ with $\widetilde{f}(0)=0$.
  Then, seeing $\widetilde{f}$ as an element of $\Shift$, we define $\Nu^2:=(-\widetilde{f})\cdot \Nu^1$.
  Now, by $\bm{(*)}$ above and assumption \textbf{(1)}, all bidegrees of $\Nu^2$ are in $\Z^2$.
\end{proof}

\begin{remark}
  The assumptions of \Cref{propositionremovinggaps} are satisfied by any $\SO$ modular functor.
\end{remark}


\section{The ribbon fusion category of a modular functor}\label{sectioncn}


\subsection{Description of the ribbon fusion category}\label{subweaklyribbon}

We will use the model of $D^n$ described in figure \Cref{figuredisk}, i.e. we align inputs $1$ to $n$ from left to right on a horizontal segment in the disk.

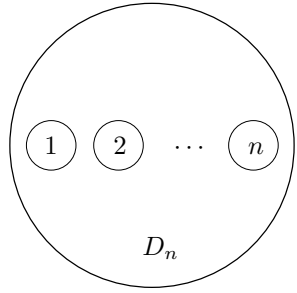
\begin{figure}
  \def\svgwidth{0.3\linewidth}
\begingroup%
  \makeatletter%
  \providecommand\color[2][]{%
    \errmessage{(Inkscape) Color is used for the text in Inkscape, but the package 'color.sty' is not loaded}%
    \renewcommand\color[2][]{}%
  }%
  \providecommand\transparent[1]{%
    \errmessage{(Inkscape) Transparency is used (non-zero) for the text in Inkscape, but the package 'transparent.sty' is not loaded}%
    \renewcommand\transparent[1]{}%
  }%
  \providecommand\rotatebox[2]{#2}%
  \newcommand*\fsize{\dimexpr\f@size pt\relax}%
  \newcommand*\lineheight[1]{\fontsize{\fsize}{#1\fsize}\selectfont}%
  \ifx\svgwidth\undefined%
    \setlength{\unitlength}{320.31498226bp}%
    \ifx\svgscale\undefined%
      \relax%
    \else%
      \setlength{\unitlength}{\unitlength * \real{\svgscale}}%
    \fi%
  \else%
    \setlength{\unitlength}{\svgwidth}%
  \fi%
  \global\let\svgwidth\undefined%
  \global\let\svgscale\undefined%
  \makeatother%
  \begin{picture}(1,0.99999993)%
    \lineheight{1}%
    \setlength\tabcolsep{0pt}%
    \put(0,0){\includegraphics[width=\unitlength,page=1]{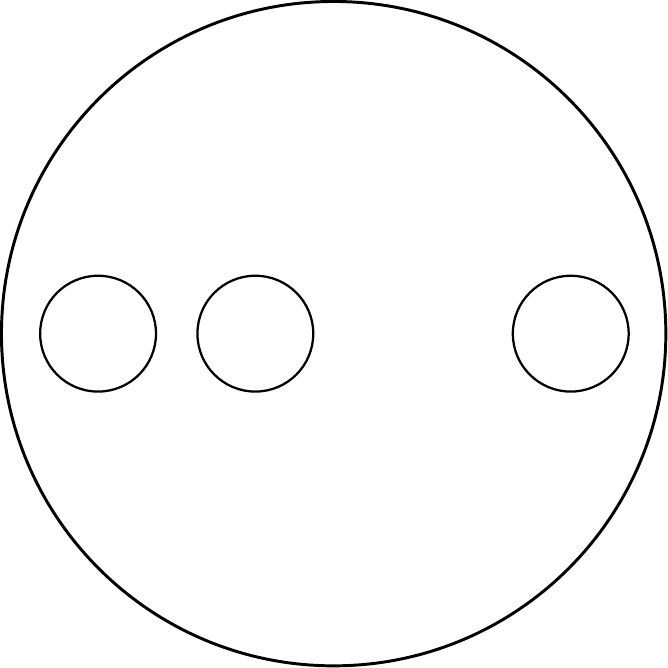}}%
    \put(0.12382489,0.46940974){\color[rgb]{0,0,0}\makebox(0,0)[lt]{\lineheight{1.25}\smash{\begin{tabular}[t]{l}$1$\end{tabular}}}}%
    \put(0.36510583,0.46940974){\color[rgb]{0,0,0}\makebox(0,0)[lt]{\lineheight{1.25}\smash{\begin{tabular}[t]{l}$2$\end{tabular}}}}%
    \put(0.46688688,0.11520547){\color[rgb]{0,0,0}\makebox(0,0)[lt]{\lineheight{1.25}\smash{\begin{tabular}[t]{l}$D_n$\end{tabular}}}}%
    \put(0.57637496,0.46940974){\color[rgb]{0,0,0}\makebox(0,0)[lt]{\lineheight{1.25}\smash{\begin{tabular}[t]{l}$\dotsb$\end{tabular}}}}%
    \put(0.83738051,0.46940974){\color[rgb]{0,0,0}\makebox(0,0)[lt]{\lineheight{1.25}\smash{\begin{tabular}[t]{l}$n$\end{tabular}}}}%
  \end{picture}%
\endgroup%

  \caption{The model for the disk $D_n$.}
  \label[figure]{figuredisk}
\end{figure}

\begin{definition}\label{definitionlegomaps}
  Let $n\geq 1$ and $m\geq 0$. For $1\leq i\leq n$, denote by $\circ_i:D^n\sqcup D^m\ra D^{n+m-1}$ the gluing of the output of $D^m$ to the $i$-th input of $D^n$
  as in \Cref{figurediskgluing}.

  For $n\geq 2$ and $1\leq i\leq n-1$, let $\beta_i:D^n\ra D^n$ be the braiding of the $i$-th and $(i+1)$-th inputs.
  For $n\geq 1$ and $1\leq i\leq n$, let $\theta_i:D^n\ra D^n$ be the Dehn twist around the $i$-th input.
\end{definition}

\begin{figure}
  \def\svgwidth{0.8\linewidth}
  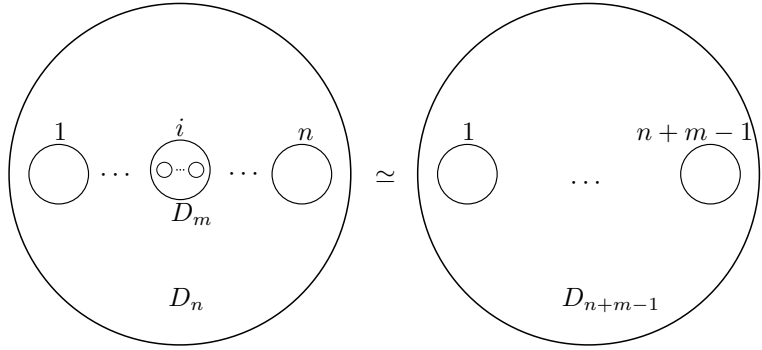
  \caption{The gluing $\circ_i:D^n\cup D^m\ra D^{n+m-1}$.}
  \label[figure]{figurediskgluing}
\end{figure}

\begin{definition}\label{definitionsemisimplecategory}
  For $\Lambda$ a finite set and $R$ a ring, denote by $\VLR$ the category of $\Lambda$-decorated modules over $R$,
  whose objects are families $W=(W_\lambda)_{\lambda\in\Lambda}$ of finite projective $R$-modules, and morphisms from $W$ to $V$
  are families of morphisms $(\varphi_\lambda:W_\lambda\ra V_\lambda)_{\lambda\in\Lambda}$ of $R$-modules.

  We will use the notation $(W_\lambda)_{\lambda\in\Lambda}=\sum_\lambda W_\lambda[\lambda]$.
\end{definition}

We will abbreviate $\VLR$ to $\VL$ when $R$ is obvious.
For definitions of monoidal category, rigid monoidal category, fusion category, braided monoidal category,
ribbon tensor category and modular category, see \cite[2.1, 2.10, 4.1, 8.1, 8.10 and 8.13]{etingofTensorCategories2015}.

Let us now describe how to construct from a ribbon functor $\Nu$ a ribbon structure on $\VL$,
and from a braided functor $\Nu$ braided structure on $\VL$.
Because of how $\VL$ is defined, we only need to define the structure on the generating set of simple objects $\{[\lambda]\mid\lambda\in\Lambda\}$.
We define the monoidal unit by $1=[0]$, the monoidal product by:
\begin{equation*}
  [\lambda]\otimes [\mu]=\sum_{\nu\in\Lambda}\Nu(D^2,\nu;\lambda,\mu)[\nu]
\end{equation*}
and the associator $c$ on simple objects by the following composition of gluing maps:
\begin{align*}
  ([\lambda]\otimes [\mu]) \otimes [\nu]&= \bigoplus_{\delta,\kappa}\Nu(D^2,\delta;\kappa,\nu)\otimes\Nu(D^2,\kappa;\lambda,\mu)[\delta] \\
                                        &\stackrel{\circ_1}{\simeq}\bigoplus_{\delta}\Nu(D^3,\delta;\lambda,\mu,\nu)[\delta] \\
                                        &\stackrel{\circ_2^{-1}}{\simeq}\bigoplus_{\delta,\kappa}\Nu(D^2,\delta;\lambda,\kappa)\otimes\Nu(D^2,\kappa;\mu,\nu)[\delta]\\
                                        &=[\lambda]\otimes ([\mu] \otimes [\nu]).
\end{align*}
The braiding $\beta_{\lambda,\mu}:[\lambda]\otimes [\mu]\ra [\mu]\otimes[\lambda]$ is given on the $[\nu]$ components
by the action of the braid $\beta_1$ on $\Nu(D^2,\nu;\lambda,\mu)$. As for the twist $\theta_\lambda:[\lambda]\ra[\lambda]$ in the case of ribbon functors,
it is given, after the canonical identification $[\lambda]=R[\lambda]\simeq \Nu(D^1,\lambda;\lambda)[\lambda]$, by the action of the Dehn twist $\theta_1$
on $\Nu(D^1,\lambda;\lambda)$. We can also simply describe $\theta_\lambda$ as multiplication by $t_\lambda$.
As this is quite folklore, we leave it to the reader to check that the axioms of a braided monoidal structure (with a twist)
are satisfied (otherwise, an equivalent but slightly different definition of the structures, together with a proof, are provided in
\cite[5.3]{bakalovLecturesTensorCategories2000}).

Let us remark that because of axiom \textbf{(Dual)} in \Cref{definitionrootedfunctor}, the monoidal structure on $\VL$ is weakly rigid in the sense
of \cite[5.3.4]{bakalovLecturesTensorCategories2000}, i.e. for each object $x$, the functor $y\mapsto \hom(1,x\otimes y)$
is (left) representable by an object $x^*$, called its weak right dual, and the functor $x\mapsto x^*$ is an equivalence.
In our case, the weak duals are given on simple objects by $[\lambda]^*=\Nu(D^2,0;\lambda,\lambda^\dagger)^\vee[\lambda^\dagger]$.

Moreover, Etingof and Penneys recently proved the long-standing conjecture that braided weakly fusion categories are braided fusion categories
\cite{etingofRigidityNonnegligibleObjects2024}. Hence the category $(\VL,\otimes,c,\beta)$ is braided fusion,
and in the case where $\Nu$ is a ribbon functor, $(\VL,\otimes,c,\beta,\theta)$ is ribbon.

Hence we have proved the following.

\begin{proposition}\label{propositionribbonstructure}
  Let $\Nu$ be a braided, respectively ribbon, functor over $R$ with set of colors $\Lambda$.
  Then $\Nu$ induces on $\VL$ the structure of a braided fusion category, respectively ribbon fusion category.
  We shall denote this category with structure $\CN$.
\end{proposition}

From the definition of the associator and the gluing axiom we get the following.

\begin{proposition}\label{propositionhomsinCN}
  With the assumptions of \Cref{propositionribbonstructure}, we have, for each $\mu,\lambda_1,\dotsc,\lambda_n\in\Lambda$, canonical isomorphisms:
  \begin{equation*}
    \Nu(D^n,\mu;\lambda_1,\dotsc,\lambda_n)\simeq \hom_{\CN}([\mu],[\lambda_1]\otimes\dotsb\otimes[\lambda_n])
  \end{equation*}
  compatible with braid groups actions and such that gluing on the left-hand side corresponds to partial compositions on the right-hand side.
\end{proposition}

We know define a groupoids of ribbon fusion categories and braided fusion categories. Note that our notion of isomorphism is non-standard.

\begin{definition}\label{definitionweaklyribbon}
  For $R$ a ring, let $\Rib$ be the category whose:
  \begin{description}
    \item[(1)] objects are pairs $(\Lambda,\VL)$, where $\Lambda$ is some finite set, and $\VL$ is endowed with a ribbon structure;
    \item[(2)] morphisms from $(\Lambda_1,\VLone)$ to $(\Lambda_2,\VLtwo)$ are pairs $(f,\phi)$, with $f:\Lambda_1\ra\Lambda_2$ a bijection and 
    $\phi:\otimes_1\simeq f^*\otimes_2$ a natural isomorphism, such that they induce a monoidal isomorphism $\VLone\ra\VLtwo$
    compatible with the braidings and twists. In other words, morphisms form $(\Lambda_1,\VLone)$ to $(\Lambda_2,\VLtwo)$ are ribbon isomorphisms that
    are induced by a bijection $f:\Lambda_1\ra\Lambda_2$ at the level of the $R$-linear categories and are strict on tensor units
    \footnote{Strictness of $f^*[0]=[0]\stackrel{\id}{\simeq} [0]$ is a consequence of the normalization axiom of ribbon functors.}.
  \end{description}
  One similarly defines the category $\Braid$ of braided fusion categories.
\end{definition}


\subsection{Full-faithfulness of forgetful functors}

\begin{theorem}\label{theoremfullfaithfulness}
  Let $R$ be a ring, all of whose quotients have characteristic $0$ and $c\in R^\times$.
  Denote by $\Mod_c$ the full subcategory of $\Mod$ consisting of modular functors with
  central charge $c$. Then all \emph{horizontal}\footnote{The vertical ones are certainly not full.} forgetful functors in the
  commutative diagram:
  \[\begin{tikzcd}
    {\Mod_c} & {\Mod^0} & \Ribfun & \Rib \\
    && \Braidfun & \Braid
    \arrow[from=1-1, to=1-2]
    \arrow[from=1-2, to=1-3]
    \arrow[from=1-3, to=1-4]
    \arrow[from=1-3, to=2-3]
    \arrow[from=1-4, to=2-4]
    \arrow[from=2-3, to=2-4]
  \end{tikzcd}\]
  are fully faithful.
\end{theorem}

\begin{proof}[Proof of faithfulness]
  By \Cref{propositionhomsinCN}, the functors $\Ribfun\ra\Rib$ and $\Braidfun\ra\Braid$ are faithful.
  Moreover, as any surface can be cut into genus $0$ surfaces, we see that $\Mod_c\ra\Mod^0$ is also faithful.
  Let us turn to the functor $\Mod^0\ra\Ribfun$. For any colored genus $0$ surface $\Sigma=(S,\ul)$
  with at least one boundary component $B\subset\partial S$, say numbered $1$, we have for any $\Nu\in\Mod^0$, by definition:
  \begin{align*}
    \Nu(\Sigma)&\simeq\Nu(S,\lambda_1^\dagger;\lambda_2,\dotsc,\lambda_n)\otimes\Nu(S_0^2,\lambda_1,\lambda_1^\dagger)\\
               &\simeq\Nu(S,\lambda_1^\dagger;\lambda_2,\dotsc,\lambda_n)\otimes\Nu(D^2,0;\lambda_1,\lambda_1^\dagger)
  \end{align*}
  canonically. As the right-hand side depends only on the ribbon functor associated to $\Nu$, $\Mod^0\ra\Ribfun$ is faithful.
\end{proof}

The aim of the rest of this section is to prove the fullnesses.

\begin{remark}
  We expect \Cref{theoremfullfaithfulness} to be a consequence of a much more general result of Brochier and Woike \cite[6.8]{brochierClassificationModularFunctors2023}.
  However, their definition of the category of modular functors is more complicated and non-explicit.
  Some work would be necessary to simplify their definition in the case of semisimple modular functors and compare it to ours.
  In this paper, we provide a proof that does not use their results.
\end{remark}


\subsection{Remarks on equivalences of functors with categories and on truncations}

In this subsection, we explain why, over $\C$, the fully-faithful functors in \Cref{theoremfullfaithfulness}
are (almost) equivalences, and also extend the Theorem to the case of truncated functors.
We will need neither of these facts in this paper, but place them here to give the reader a fuller picture and for future reference.

\begin{notation}
  We will denote by $\Modcat\subset\Rib$ the full sub-groupoid of modular fusion category, i.e. those whose $S$-matrix is non-degenerate
  (see \cite[8.13.4]{etingofTensorCategories2015}).
\end{notation}

\begin{theorem}\label{theoremequivalences}
  Here, the base ring is $R=\C$. Consider the commutative diagram of forgetful functors
  \[\begin{tikzcd}
    {\bigsqcup_c\Mod_c} & {\Mod^0} & \Braidfun \\
    & \Ribfun \\
    \Modcat & \Rib & {\Braid.}
    \arrow[from=1-1, to=1-2]
    \arrow["{\sqcup_c\mathrm{f}_c}"', from=1-1, to=3-1]
    \arrow[from=1-2, to=1-3]
    \arrow["{\mathrm{f}^0}", from=1-2, to=2-2]
    \arrow["{\mathrm{f}^b}", from=1-3, to=3-3]
    \arrow["{\mathrm{f}^r}", from=2-2, to=3-2]
    \arrow[from=3-1, to=3-2]
    \arrow[from=3-2, to=3-3]
  \end{tikzcd}\]
  Then $\mathrm{f}^0$, $\mathrm{f}^r$ and $\mathrm{f}^b$ are equivalences of groupoids,
  each $\mathrm{f}_c$ is fully-faithful, and all objects of $\Modcat$ are in the essential image of $\mathrm{f}_c$ for some $c\in\C^\times$
  \footnote{See \cite[5.7.10]{bakalovLecturesTensorCategories2000} for the exact values $c$ may take.}.
\end{theorem}

\begin{proof}
  The facts about the essential images of $\mathrm{f}^r\circ\mathrm{f}^0$ and $\bigsqcup_c\mathrm{f}_c$ are proved in \cite[5.4.1, 6.7.13]{bakalovLecturesTensorCategories2000}.
  For $\bigsqcup_c\mathrm{f}_c$, this is the Reshetikhin-Turaev construction.
  All full-faithfulnesses are part of \Cref{theoremfullfaithfulness}. These imply that $\mathrm{f}^0$ and $\mathrm{f}^r$ are equivalences.
  Note that with our definition a braided functor is equivalent to an $E_2$-algebra in the monoidal $2$-category of $\C$-linear semisimple categories.
  Then, the fact that $\mathrm{f}^b$ is essential surjective follows from the known fact that $E_2$-algebras in the category of categories
  are equivalent to braided monoidal categories, see \cite[1.2.4, 1.2.5]{lurieDerivedAlgebraicGeometry2009}.
\end{proof}

Let us turn to the case of truncated functors. These are obtained by restricting the possible Euler characteristic of surfaces in the definitions.

\begin{definition}
  For $\chi\leq -3$ an integer, denote by $\tSurf^{\geq \chi}\subset\tSurf$ the full subgroupoid
  of surfaces of Euler characteristic $\geq \chi$.
  By replacing $\tSurf$ with $\tSurf^{\geq \chi}$ in the \Cref{definitionmodularfunctor} of modular functor,
  one obtains the definition of $\chi$-truncated modular functor. The groupoid of $\chi$-truncated
  modular functors with central charge $c$ will be denoted $\Mod_c^{\geq \chi}$.

  One similarly defines $\chi$-truncated genus $0$ modular functors, ribbon functors and braided functors.
  The corresponding groupoids will be denoted $\Mod^{0,\geq \chi}$, $\Ribfun^{\geq \chi}$ and $\Braidfun^{\geq \chi}$ respectively.
\end{definition}

\begin{theorem}\label{forgetsoverC}
  Here, the base ring is $R=\C$.
  Fix $c\in\C^\times$. Then for any $\chi\leq -3$, the following forgetful functors are
  well-defined and equivalences:
  \begin{equation*}
    \Mod_c\lra \Mod_c^{\geq \chi}\lra \Modcat.
  \end{equation*}
  The same result holds for genus $0$ modular functors,
  ribbon functors and braided functors, replacing $\Modcat$ by $\Rib$, $\Rib$ and $\Braid$ respectively.
\end{theorem}

\begin{remark}
  Note that $\chi\leq -3$ is required for $\Mod_c^{\geq \chi}\lra \Rib$ to be defined.
\end{remark}

\begin{proof}
  The non-trivial part of the well-definiteness is that categories in the essential image of $\Mod_c^{\geq \chi}\ra \Rib$ are modular.
  When Bakalov-Kirillov prove in \cite[5.5.1]{bakalovLecturesTensorCategories2000} that the ribbon fusion category associated to
  a modular functor is modular, they rely on relations in tori with $1$ and $2$ holes.\footnote{They assume that the associated category is rigid.
  However this assumption is superfluous thanks to the resent result of Etingof and Penneys \cite{etingofRigidityNonnegligibleObjects2024}.}
  These relations will also be satisfied in a $\chi$-truncated modular functor for $\chi\leq -3$, as these surfaces are accessible.
  Hence their proof also implies that the ribbon fusion category associated to a $\chi$-truncated modular functor for $\chi\leq -3$ is modular.

  Thanks to \Cref{theoremequivalences}, it remains only to prove that $\Mod_c^{\geq \chi}\lra \Modcat$, $\Mod^{0,\geq \chi}\lra \Rib$,
  $\Ribfun^{0,\geq \chi}\lra \Rib$, and $\Braidfun^{\geq \chi}\lra \Braid$ are fully-faithful.
  The point is to notice that the proofs given for $\chi=-\infty$ in the next subsections apply directly to the truncated case.
  Let us comment on the modular case, as the other cases are easier.
  In that case, the only non-trivial part of the proof is the use of curve operators, which requires
  $2$ points (see \Cref{subsubAU}):
  that Dehn twists be polynomials in the curve operators and that $g=1$, $n=3$ be accessible.
  The first point is a direct consequence of modularity of the associated ribbon fusion category.
  As for the second point, $\chi\leq-3$ ensures that $g=1$, $n=3$ is accessible.
\end{proof}


\subsection{Fullness for braided and ribbon functors}\label{subsectionbraidedribbon}

We prove fullness of $\Braidfun \ra \Braid$. The case of $\Ribfun \ra \Rib$ is similar, see \Cref{remarkribboncase} below.

Let $\Nu_1$ and $\Nu_2$ be $2$ braided functors with respective braided fusion categories $\CNone$ and $\CNtwo$.
Let $G:\CNone\ra\CNtwo$ be a morphism. We may assume, without loss of generality, that the two functors have the same set of colors $\Lambda$
and that $G$ acts trivially on $\Lambda$.
For $D^n$ the standard disk (see \Cref{subweaklyribbon}) and $\mu,\lambda_1,\dotsc,\lambda_n\in\Lambda$,
\Cref{propositionhomsinCN} provides canonical isomorphisms
\begin{equation*}
  \Nu_i(D^n,\mu;\lambda_1,\dotsc,\lambda_n)\simeq \hom_{\mathrm{C}_{\Nu_i}}([\mu],[\lambda_1]\otimes\dotsb\otimes[\lambda_n])
\end{equation*}
for $i=1,2$, which are compatible with braid group actions and with the standard gluing of standard disks described in
\Cref{definitionlegomaps} and \Cref{figurediskgluing}.
Hence $G$ provides isomorphisms
$$G_{n,\mu,\ul}:\Nu_1(D^n,\mu;\lambda_1,\dotsc,\lambda_n)\simeq \Nu_2(D^n,\mu;\lambda_1,\dotsc,\lambda_n)$$
compatible with standard gluings and braid group actions.

To show fullness, we need to extend them to any configuration
of disks. Fix $D$ a $\Lambda$-labeled configuration of $n$ disks for some $n\geq 0$.
Then there exists an isomorphism $\phi:D\simeq (D^n,\mu;\lambda_1,\dotsc,\lambda_n)$ for some choice of $\mu,\lambda_1,\dotsc,\lambda_n$.
We define an isomorphism $G_{D,\phi}:\Nu_1(D)\simeq \Nu_2(D)$ by imposing the commutation of
\[\begin{tikzcd}
	{\Nu_1(D)} & {\Nu_2(D)} \\
	{\Nu_1(D^n,\mu;\lambda_1,\dotsc,\lambda_n)} & {\Nu_2(D^n,\mu;\lambda_1,\dotsc,\lambda_n).}
	\arrow["{G_{D,\phi}}", from=1-1, to=1-2]
	\arrow["{\Nu_1(\phi)}", from=1-1, to=2-1]
	\arrow["{G_{n,\mu,\ul}}", from=2-1, to=2-2]
	\arrow["{\Nu_2(\phi^{-1})}"', from=2-2, to=1-2]
\end{tikzcd}\]

In fact $G_{D,\phi}$ does not depend on $\phi$ or the choice of the ordering of the $\lambda_i$.
Indeed, if $\phi':D\simeq (D^n,\mu;\lambda_1',\dotsc,\lambda_n')$ is another isomorphism, then there exists a unique
isomorphism $\alpha:(D^n,\mu;\lambda_1,\dotsc,\lambda_n)\simeq(D^n,\mu;\lambda_1',\dotsc,\lambda_n')$ compatible with $\phi$ and $\phi'$.
As $\alpha$ is a braid, $G_{n,\mu,\ul}$ and $G_{n,\mu,\ul'}$ are compatible with it, and thus we have commutation of
\[\begin{tikzcd}
	&& {\Nu_1(D)} \\
	& {\Nu_1(D^n,\mu;\lambda_1,\dotsc,\lambda_n)} \\
	& {\Nu_2(D)} & {\Nu_1(D^n,\mu;\lambda_1',\dotsc,\lambda_n')} \\
	{\Nu_2(D^n,\mu;\lambda_1,\dotsc,\lambda_n)} \\
	& {\Nu_2(D^n,\mu;\lambda_1',\dotsc,\lambda_n')}
	\arrow["{\Nu_1(\phi)}"', from=1-3, to=2-2]
	\arrow["{\Nu_1(\phi')}", from=1-3, to=3-3]
	\arrow["{\Nu_1(\alpha)}", from=2-2, to=3-3]
	\arrow["{G_{n,\mu,\ul}}"', from=2-2, to=4-1]
	\arrow["{G_{n,\mu,\ul'}}", from=3-3, to=5-2]
	\arrow["{\Nu_2(\phi^{-1})}"', from=4-1, to=3-2]
	\arrow["{\Nu_2(\alpha)}", from=4-1, to=5-2]
	\arrow["{\Nu_2(\phi^{\prime-1})}"', from=5-2, to=3-2]
\end{tikzcd}\]
and thus $G_{D,\phi'}=G_{D,\phi}$. Hence we may abbreviate $G_{D,\phi}$ to $G_{D}$, and the same argument proves that
the family $G_D$ induces a natural isomorphism of functors $\Nu_1\Rightarrow\Nu_2:\: \Disk^b\ra\C\text{-vector spaces}$.
Moreover, compatibility with gluing axioms then follows from the fact that any gluing map is isomorphic to a standard one.
This concludes the proof of fullness of $\Braidfun\ra\Braid$.

\begin{remark}\label{remarkribboncase}
  The proof works as is in the case of ribbon functors, replacing braid groups by framed braid groups and $\Disk^b$ by $\Disk$.
\end{remark}


\subsection{Decomposition complex, following Bakalov-Kirillov}\label{subsectionlego}

We give a concise introduction to the decomposition complex of a surface, following \cite[5.2]{bakalovLecturesTensorCategories2000}.
See the reference for details and proofs.

We will use a model of $S_0^n$ which we describe by $S_0^n=D^n\cup_\mathrm{output}D^0$ where $D^n$ is the model as in \Cref{figuredisk}.
We will denote simply by $i$ the $i$-th boundary component from left to right in the Figure.
Clearly, standard braiding maps and twists $\beta_i,\theta_i:D^n\ra D^n$ of \Cref{definitionlegomaps} extend to $\beta_i,\theta_i:S_0^n\ra S_0^n$.
For $n\geq 2$, we define $z:S_0^n\ra S_0^n$ as the homeomorphism that cyclically permutes boundary components around the horizontal axis
sending $1$ to $n$.
We can also define a standard gluing $\circ:S_0^n\sqcup S_0^m\ra S_0^{n+m-2}$ of $n$ in $S_0^n$ to $0$ in $S_0^m$.

\begin{definition}[Decomposition]\label{defintiondecomposition}
  Let $S$ be a compact surface with parametrized boundary. A decomposition of $S$ is a pair $(\Gamma,(\phi_{S_i})_i)$
  where:
  \begin{description}
    \item[(1)] $\Gamma$ is a family of disjoint simple closed curves 
    on $S$ such that the surface $S'=\bigsqcup_iS_i$ obtained by cutting along $\Gamma$ has all of its connected components $S_i$ of genus $0$,
    each with non empty boundary;
    \item[(2)] for each $S_i$, $\phi_{S_i}:S_i\ra S_0^{n_i}$ is a homeomorphism compatible with boundary parametrizations.
  \end{description}
\end{definition}

\begin{definition}[Decomposition moves]\label{defintiondecompositionmoves}
  Let $S$ be a compact surface with parametrized boundary. A $Z$-move on a decomposition $(\Gamma,(\phi_{S_i})_i)$ is done
  by choosing a component $S_i$ and post-composing $\phi_{S_i}:S_i\ra S_0^{n_i}$ by $z:S_0^{n_i}\ra S_0^{n_i}$.
  A $B$-move is done similarly by post-composing some $\phi_{S_i}:S_i\ra S^{n_i}$ with $n_i=3$ by the braid $\beta_1:S_0^3\ra S_0^3$.
  
  To perform an $F$-move, choose a component $\gamma\subset\Gamma$ such that, after cutting, it meets $2$ distinct components of $S'$,
  $S_i$ and $S_j$, and such that $\phi_{S_i}$ sends $\gamma$ to $n_i$ in $S_0^{n_i}$ and $\phi_{S_j}$ sends $\gamma$ to $0$ in $S_0^{n_j}$.
  Then the associated $F$-move consists in replacing $\Gamma$ by $\Gamma\setminus\gamma$,
  $S_i$ and $S_j$ by $S_i\cup_\gamma S_j$ with $\phi: S_i\cup_\gamma S_j\ra S_0^{n_i}\cup_\gamma S_0^{n_j}\stackrel{\circ}{\ra}S_0^{n_i+n_j-2}$.

  To perform an $S$-move, choose a component $S_i$ of $S'$ with $n_i=2$ and such that $2$ and $3$ correspond to the same component $\gamma\in\Gamma$.
  Then replace $\gamma$ by $\gamma'$ as in \Cref{figureSmove}, and change $S_i$ and $\phi_i$ accordingly.
\end{definition}

\begin{figure}
  \def\svgwidth{0.4\linewidth}
\begingroup%
  \makeatletter%
  \providecommand\color[2][]{%
    \errmessage{(Inkscape) Color is used for the text in Inkscape, but the package 'color.sty' is not loaded}%
    \renewcommand\color[2][]{}%
  }%
  \providecommand\transparent[1]{%
    \errmessage{(Inkscape) Transparency is used (non-zero) for the text in Inkscape, but the package 'transparent.sty' is not loaded}%
    \renewcommand\transparent[1]{}%
  }%
  \providecommand\rotatebox[2]{#2}%
  \newcommand*\fsize{\dimexpr\f@size pt\relax}%
  \newcommand*\lineheight[1]{\fontsize{\fsize}{#1\fsize}\selectfont}%
  \ifx\svgwidth\undefined%
    \setlength{\unitlength}{120.74451477bp}%
    \ifx\svgscale\undefined%
      \relax%
    \else%
      \setlength{\unitlength}{\unitlength * \real{\svgscale}}%
    \fi%
  \else%
    \setlength{\unitlength}{\svgwidth}%
  \fi%
  \global\let\svgwidth\undefined%
  \global\let\svgscale\undefined%
  \makeatother%
  \begin{picture}(1,0.8820764)%
    \lineheight{1}%
    \setlength\tabcolsep{0pt}%
    \put(0,0){\includegraphics[width=\unitlength,page=1]{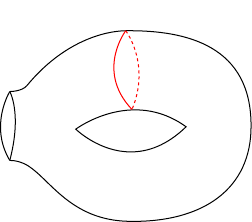}}%
    \put(0.470072,0.81834436){\color[rgb]{1,0,0}\makebox(0,0)[lt]{\lineheight{1.25}\smash{\begin{tabular}[t]{l}$\gamma$\end{tabular}}}}%
    \put(0.45328261,0.16637089){\color[rgb]{0.15686275,0,1}\makebox(0,0)[lt]{\lineheight{1.25}\smash{\begin{tabular}[t]{l}$\gamma'$\end{tabular}}}}%
    \put(0,0){\includegraphics[width=\unitlength,page=2]{Smove.pdf}}%
  \end{picture}%
\endgroup%

  \caption{The curves $\gamma$ and $\gamma'$ in $S_1^1$ describing an $S$-move.}
  \label[figure]{figureSmove}
\end{figure}

\begin{theorem}[{\cite[5.2.9]{bakalovLecturesTensorCategories2000}}]\label{theoremmovesconnected}
  Let $S$ be a compact surface with parametrized boundary.
  Then any two decompositions of $S$ are related by a sequence of $Z$, $B$, $F$ and $S$-moves or inverses of these moves.
\end{theorem}

Let $\Nu$ be a modular functor with associated ribbon fusion category $\CN$. Then for colors $\lambda_1,\dotsc,\lambda_n\in\Lambda$,
we have natural\footnote{with respect to $\Nu$.} isomorphisms:
\begin{equation}\label{canonicaliso}
  \Nu(S_0^n,\lambda_1,\dotsc,\lambda_n)\stackrel{\text{\textbf{(V)}}}{\longleftrightarrow}
  \Nu(D^n,0;\lambda_1,\dotsc,\lambda_n)\stackrel{\ref{propositionhomsinCN}}{\longleftrightarrow}
  \hom_{\CN}([0],[\lambda_1]\otimes\dotsb\otimes[\lambda_n]).
\end{equation}

\begin{definition}\label{definitionGD}
  Let $\Nu_1$ and $\Nu_2$ be $2$ modular functors in $\Mod_c$ with respective ribbon fusion categories $\CNone$ and $\CNtwo$.
  Let $G:\CNone\ra\CNtwo$ be a morphism. For a colored surface $\Sigma=(S,\ul)$, and a decomposition $D=(\Gamma,(\phi_{S_i})_i)$ of $S$,
  we define an isomorphism of $R$-modules $G_D:\Nu_1(\Sigma)\ra\Nu_2(\Sigma)$ as follows.
  For $k=1,2$, the gluing axiom along $\Gamma$ expresses $\Nu_1(\Sigma)$ as a direct sum of tensor powers of terms
  of the form $\Nu_k(S_i,\umu)$ or $\Nu_k(S_0^2,\nu,\nu^\dagger)$. Now we have isomorphisms:
  \begin{equation*}
    \Nu_k(S_i,\umu)\stackrel{\phi_{S_i}}{\simeq}\Nu_k(S_0^{n_i},\umu)\stackrel{(\ref{canonicaliso})}{\simeq}
    \hom_{\mathrm{C}_{\Nu_k}}([0],[\mu_1]\otimes\dotsb\otimes[\mu_{n_i}])
  \end{equation*}
  and $\Nu_k(S_0^2,\nu,\nu^\dagger)\simeq \hom_{\mathrm{C}_{\Nu_k}}([0],[\nu]\otimes[\nu^\dagger])$.
  Hence $\Nu_k(\Sigma)$ is expressed via $D$ in terms of vector spaces depending only on $\mathrm{C}_{\Nu_k}$.
  This allows us to apply $G$ and get an isomorphism $G_D:\Nu_1(\Sigma)\ra\Nu_2(\Sigma)$.
\end{definition}

To prove fullness of $\Mod_c\ra\Rib$, we will need that $G_D$ is independent of the decomposition $D$ (see \Cref{prooffullness}).
For this, by \Cref{theoremmovesconnected}, it is sufficient to prove that $G_D$ is invariant under $Z$, $B$, $F$ and $S$-moves.
So we need to show that if $D$ and $D'$ differ by a move,
then the associated decompositions of $\Nu(\Sigma)$ as in \Cref{definitionGD} are related by isomorphisms depending only on $\CN$
and natural in $\Nu$.
Below, we explain the invariance under $Z$, $B$, $F$-moves. Invariance under $S$-moves will be the subject of the next subsection \ref{subsubAU}.

Invariance under $B$-move is a consequence of the definition of the braiding on $\CN$.
For invariance under $F$-moves, we need to show that genus $0$ gluing maps of a modular functor $\Nu$ are determined by $\CN$.
This is a consequence of the fact that the following composition:
\[\begin{tikzcd}
	{\Nu(S_0^n,\lambda_1,\dotsc,\lambda_{n-1},\mu)\otimes\Nu(S_0^2,\mu,\mu^\dagger)\otimes\Nu(S_0^m,\mu^\dagger,\nu_1,\dotsc,\nu_{m-1})} \\
	\begin{array}{c}     \hom_{\CN}([0],[\lambda_1]\otimes\dotsb\otimes[\lambda_{n-1}]\otimes[\mu])\otimes\hom_{\CN}([0],[\mu]\otimes[\mu^\dagger])\\\otimes\hom_{\CN}([0],[\mu^\dagger]\otimes[\nu_1]\otimes\dotsb\otimes[\nu_{m-1}]) \end{array} \\
	{    \hom_{\CN}([0],[\lambda_1]\otimes\dotsb\otimes[\lambda_{n-1}]\otimes[\mu])\otimes\hom_{\CN}([\mu],\otimes[\nu_1]\otimes\dotsb\otimes[\nu_{m-1}])} \\
	{    \hom_{\CN}([0],[\lambda_1]\otimes\dotsb\otimes[\lambda_{n-1}]\otimes[\nu_1]\otimes\dotsb\otimes[\nu_{m-1}])} \\
	{\Nu(S_0^n,\lambda_1,\dotsc,\lambda_{n-1},\nu_1,\dotsc,\nu_{m-1})}
	\arrow[tail reversed, from=1-1, to=2-1]
	\arrow["{\text{duality}}"', tail reversed, from=2-1, to=3-1]
	\arrow["{\text{partial composition}}"', tail reversed, from=3-1, to=4-1]
	\arrow[tail reversed, from=5-1, to=4-1]
\end{tikzcd}\]
is the gluing map.

As for $Z$-moves, notice that $z:S_0^n\ra S_0^n$ is induced by $z'=\beta'\circ\theta_1:D^n\ra D^n$ where $\beta'$ is the braiding of $1$ around all the other $i$
for $2\leq i\leq n$. Now, under the isomorphism of \Cref{canonicaliso}, $\beta'$ and $\theta_1$ can be described by the braiding and ribbon structure on $\CN$.
Hence $Z$-moves can be described in $\CN$.


\subsection{Higher genus and curve operators, following Andersen-Ueno}\label{subsubAU}

We define curve operators following \cite[4]{andersenModularFunctorsAre2012}.
Note that although our definition of modular functor is different from that of Andersen-Ueno, any modular functor $\Nu$ with our definition
satisfies their definition after choosing for each $\lambda\in\Lambda$ an isomorphism $\Nu(S_0^2,\lambda,\lambda^\dagger)\simeq R$.
Hence we can use their results.

\begin{definition}\label{definitioncurveoperator}
  Let $\Nu$ be a modular functor, $\Sigma$ a colored surface, $\gamma\subset\Sigma$ an oriented simple closed curve and $\mu\in\Lambda$ a color.
  Choose $2$ distinct points $x$ and $y$ on $\gamma$. Denote by $\Sigma'$ the colored surface obtained by removing disks around $x$ and $y$ and
  coloring the new boundary component corresponding to $x$ by $\mu$ and that corresponding to $y$ by $\mu^\dagger$.
  Let $\alpha$ and $\beta$ be the simple closed curves on $\Sigma'$ as on \Cref{figurecurveoperator}. Denote by $\Sigma'_\alpha$ and $\Sigma'_\beta$
  the cut of $\Sigma'$ along $\alpha$, respectively $\beta$.

  Then we define the curve operator $\Op{\gamma}{\mu}:\Nu(\Sigma)\ra\Nu(\Sigma)$ by tensoring the 
  following composition of maps with the identity of $\Nu(S_0^2,\mu,\mu^\dagger)^\vee$:
  \[\begin{tikzcd}
    {\bigoplus_\nu\Nu(\Sigma_\alpha,\nu,\nu^\dagger)\otimes\Nu(S_0^2,\nu,\nu^\dagger)^\vee} & {\Nu(\Sigma')} & {\bigoplus_\nu\Nu(\Sigma_\beta,\nu,\nu^\dagger)\otimes\Nu(S_0^2,\nu,\nu^\dagger)^\vee} \\
    {\Nu(\Sigma_\alpha,0,0)} && {\Nu(\Sigma_\beta,0,0)} \\
    {\Nu(\Sigma)\otimes\Nu(S_0^2,\mu,\mu^\dagger)} && {\Nu(\Sigma)\otimes\Nu(S_0^2,\mu,\mu^\dagger).}
    \arrow[tail reversed, from=2-1, to=3-1]
    \arrow[tail reversed, from=2-3, to=3-3]
    \arrow["{\nu=0\text{ summand}}"', hook, from=2-1, to=1-1]
    \arrow["{\nu=0\text{ summand}}", two heads, from=1-3, to=2-3]
    \arrow[tail reversed, from=1-1, to=1-2]
    \arrow[tail reversed, from=1-2, to=1-3]
  \end{tikzcd}\]
\end{definition}

\begin{figure}
  \def\svgwidth{0.5\linewidth}
\begingroup%
  \makeatletter%
  \providecommand\color[2][]{%
    \errmessage{(Inkscape) Color is used for the text in Inkscape, but the package 'color.sty' is not loaded}%
    \renewcommand\color[2][]{}%
  }%
  \providecommand\transparent[1]{%
    \errmessage{(Inkscape) Transparency is used (non-zero) for the text in Inkscape, but the package 'transparent.sty' is not loaded}%
    \renewcommand\transparent[1]{}%
  }%
  \providecommand\rotatebox[2]{#2}%
  \newcommand*\fsize{\dimexpr\f@size pt\relax}%
  \newcommand*\lineheight[1]{\fontsize{\fsize}{#1\fsize}\selectfont}%
  \ifx\svgwidth\undefined%
    \setlength{\unitlength}{547.16905489bp}%
    \ifx\svgscale\undefined%
      \relax%
    \else%
      \setlength{\unitlength}{\unitlength * \real{\svgscale}}%
    \fi%
  \else%
    \setlength{\unitlength}{\svgwidth}%
  \fi%
  \global\let\svgwidth\undefined%
  \global\let\svgscale\undefined%
  \makeatother%
  \begin{picture}(1,0.77456465)%
    \lineheight{1}%
    \setlength\tabcolsep{0pt}%
    \put(0,0){\includegraphics[width=\unitlength,page=1]{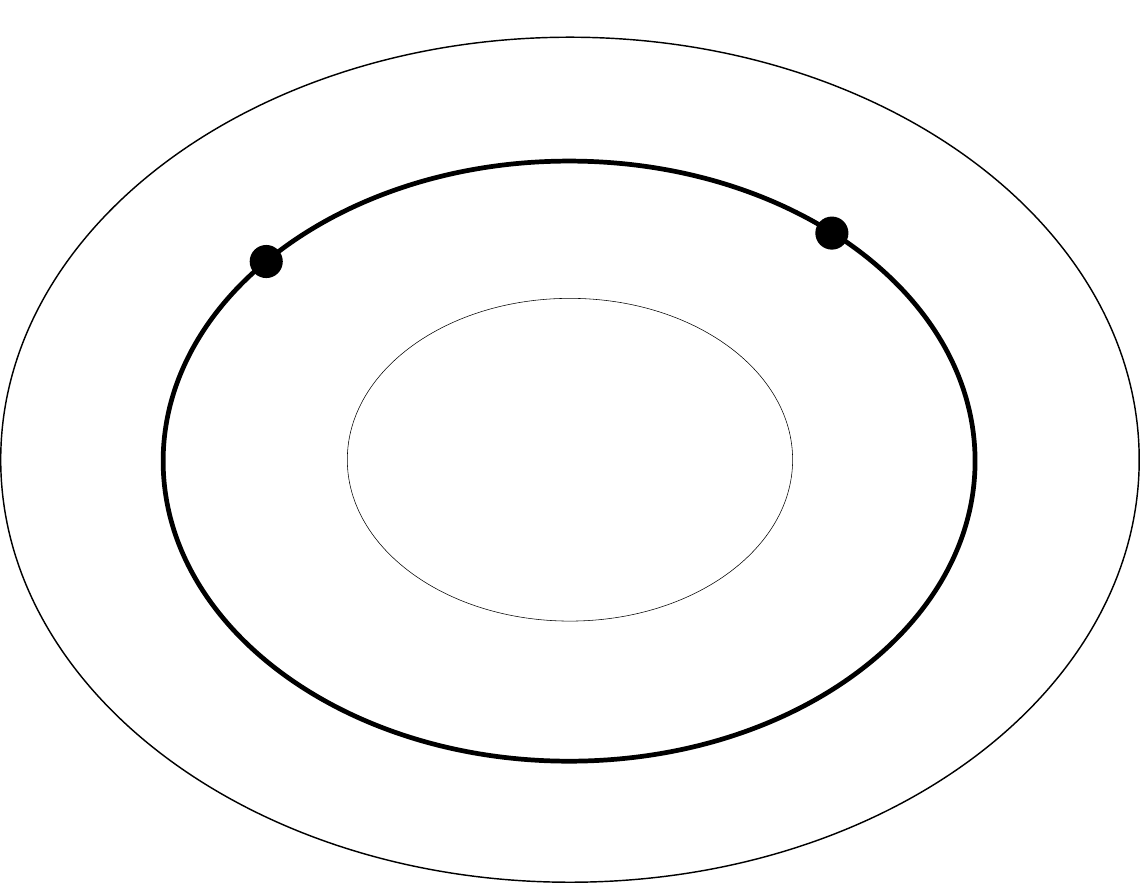}}%
    \put(0.2915926,0.61039327){\color[rgb]{0,0,0}\makebox(0,0)[lt]{\lineheight{1.25}\smash{\begin{tabular}[t]{l}$x$\\\end{tabular}}}}%
    \put(0.48038245,0.58463075){\color[rgb]{0,0,0}\makebox(0,0)[lt]{\lineheight{1.25}\smash{\begin{tabular}[t]{l}$\gamma$\\\end{tabular}}}}%
    \put(0.46895293,0.70306986){\color[rgb]{0.15294118,0.00392157,1}\makebox(0,0)[lt]{\lineheight{1.25}\smash{\begin{tabular}[t]{l}$\alpha$\\\end{tabular}}}}%
    \put(0.42100822,0.17872539){\color[rgb]{1,0.00392157,0.00392157}\makebox(0,0)[lt]{\lineheight{1.25}\smash{\begin{tabular}[t]{l}$\beta$\\\end{tabular}}}}%
    \put(0.61305344,0.63087639){\color[rgb]{0,0,0}\makebox(0,0)[lt]{\lineheight{1.25}\smash{\begin{tabular}[t]{l}$y$\\\end{tabular}}}}%
    \put(0,0){\includegraphics[width=\unitlength,page=2]{curveoperator.pdf}}%
  \end{picture}%
\endgroup%

  \caption{An annulus around $\gamma$ in $\Sigma$, with boundary components of $\Sigma'$ in red around $x$ and $y$, and curves $\alpha$ and $\beta$.}
  \label[figure]{figurecurveoperator}
\end{figure}

In words, we express $\Nu(\Sigma)$ as a summand in $\Nu(\Sigma')\otimes\Nu(S_0^2,\mu,\mu^\dagger)^\vee$ in $2$ different ways using $\alpha$ and $\beta$.
Then the curve operator $\Op{\gamma}{\mu}$ is the projection of the summand $\Nu(\Sigma)$ corresponding to $\alpha$
onto the one corresponding to $\beta$.

\begin{theorem}[{\cite[5.1]{andersenModularFunctorsAre2012}}]\label{theoremAU}
  Fix a modular functor $\Nu$ and a colored surface $\Sigma$.
  Then for each simple closed curve $\gamma\subset\Sigma$, the action of the Dehn twist $T_\gamma$ on $\Nu(\Sigma)$
  is a linear combination of the actions of the curve operators $\Op{\gamma}{\mu}$ for $\mu\in\Lambda$.
\end{theorem}

\begin{proof}
  Andersen and Ueno prove this for $R=\C$. However their proof works without changes for $R$ any field of characteristic $0$.
  
  It consists in showing that a matrix $C$ depending on the data of the modular functor is invertible \cite[5.2]{andersenModularFunctorsAre2012}.
  Assume by contradiction that the statement is false for some $\Nu$ defined over some ring $R$ all of whose quotients have characteristic $0$.
  Then $\det(C)$ is included in a maximal ideal $\mathfrak{m}$ of $R$.
  Hence the reduction of $\Nu$ over $R/\mathfrak{m}$ has non-invertible $C$-matrix. However $R/\mathfrak{m}$ is a field, so this contradicts
  Andersen and Ueno's result.
\end{proof}

\begin{remark}
  Another proof for $R$ a field of characteristic $0$
  is provided by the fact that the ribbon category associated to a modular functor is modular (see discussion in the proof of \Cref{forgetsoverC}).
  Indeed, the matrix $C$ in Andersen and Ueno's paper is invertible when the $S$-matrix is.
\end{remark}

We can now conclude on the fullness of $\Mod_c\ra\Rib$.

\begin{proof}[Proof of fullness in \Cref{theoremfullfaithfulness}]\label{prooffullness}
  Let $\Nu_1$ and $\Nu_2$ be two modular functors with respective ribbon fusion categories $\CNone$ and $\CNtwo$.
  Choose a morphisms $G:\CNone\ra\CNtwo$. Then for each decompositions $D$ of a surface $\Sigma$,
  \Cref{definitionGD} provides an isomorphism $G_D:\Nu_1(\Sigma)\ra\Nu_2(\Sigma)$.
  If we show that $G_D$ is independent of the decompositions $D$, then setting, for each colored surface $\Sigma$, $G_\Sigma:=G_D$,
  the family $(G_\Sigma)_\Sigma$ provides an isomorphism of modular functor $\widetilde{G}:\Nu_1\ra\Nu_2$ lifting $G:\CNone\ra\CNtwo$.

  Hence, by \Cref{theoremmovesconnected}, if we check that $G_D$ is invariant under $Z$, $B$, $F$ and $S$-moves, we are done.
  The discussion at the end of \Cref{subsectionlego} explains invariance under $Z$, $B$ and $F$-moves.
  Let us turn to $S$-moves. We only need to show that $G_{D_1}=G_{D_2}$
  where $D$ and $D'$ are the $2$ decompositions of the torus $\Sigma_1^1$ with one boundary component colored $\lambda$
  corresponding respectively to $\gamma$ and $\gamma'$ in \Cref{figureSmove}.
  But $D'=f(D)$ for an element $f$ of the mapping class group of $S_1^1$, and thus $G_{D'}=fG_{D}f^{-1}$.
  Hence we need only show that $G_{D}$ commutes to the action of the mapping class group of $\Sigma_1^1$.
  This mapping class group is generated by the $2$ Dehn twists $T_\gamma$ and $T_{\gamma'}$.
  Clearly, $G_{D}$ commutes to $T_\gamma$.
  As for $T_\gamma$, by \Cref{theoremAU}, it is sufficient to show that $G_{D}$ commutes to the curve operators $\Op{\gamma'}{\mu}$ for $\mu\in\Lambda$.
  On \Cref{figureoperatorasZBFmoves} we describe the horizontal gluing maps of the diagram in \Cref{definitioncurveoperator} in terms of $Z$, $B$ and $F$ moves.
  Hence $G$ commutes to these gluing maps.
  As the decompositions $D_\alpha$ and $D_\beta$ on $S_1^3$ in the figure both come from $D$ on $S_1^1$,
  $G$ also commutes to the vertical inclusions and projections maps of the diagram in \Cref{definitioncurveoperator}.
  So $\Op{\gamma'}{\mu}$ commutes to $G_{D}$ for all $\mu$, and thus $G_D$ is invariant under $S$-moves. We are done.
\end{proof}

\begin{figure}
  \def\svgwidth{\linewidth}
  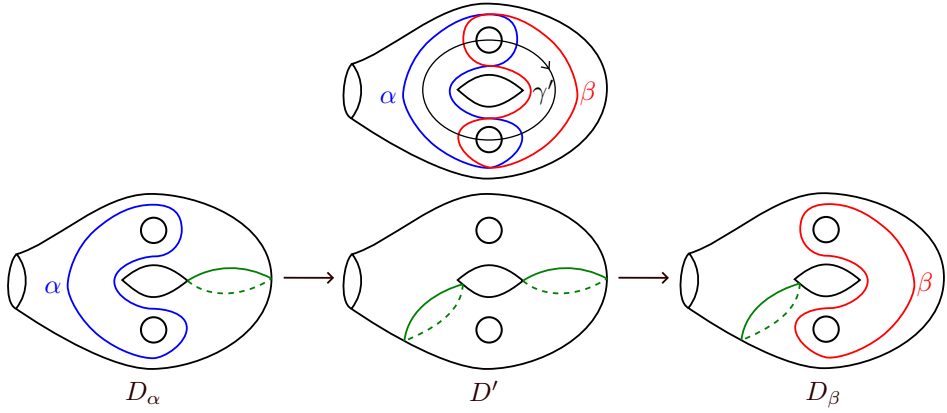
  \caption{The decompositions $D_\alpha$ and $D_\beta$ of $S_1^3$ are linked by genus $0$ moves ($Z$, $B$ and $F$).}
  \label[figure]{figureoperatorasZBFmoves}
\end{figure}


\section{Non-Abelian Hodge correspondence}\label{sectionnonabelian}

Our reference for the non-Abelian Hodge correspondence is Simpson: \cite[1]{simpsonHiggsBundlesLocal1992} over compact Kähler manifolds,
and \cite[9]{simpsonLocalSystemsProper2011} over smooth proper DM stacks. For Hodge bundles and complex variations of Hodge structures,
see \cite[4]{simpsonHiggsBundlesLocal1992}.

All varieties, manifolds and stacks considered will be smooth.
Note that by DM stack we mean algebraic Deligne-Mumford stack over $\C$. These have realizations as smooth analytic/differentiable Deligne-Mumford stacks,
and hence we may talk about holomorphic/$\Cinf$-bundles over them. See \cite[2]{eyssidieuxInstantonsFramedSheaves2018} for more on this.


\subsection{Higgs bundles and local systems}

The purpose of this subsection is to state the non-Abelian Hodge correspondence for modular functors.

\begin{notation}
  Let $X$ be a DM stack. We will use the notation $\uE$, $\uF$, $\uG$,... for $\C$-local systems on $X$,
  the notation $E$, $F$, $G$,... for $\Cinf$ complex vector bundles on $X$
  and their sheaves of $\Cinf$  sections,
  and the notation $\cE$, $\cF$, $\cG$,... for holomorphic vector bundles and their sheaves of holomorphic sections.
  Connections will be denoted $\nabla$ when they are $\Cinf$ and $\rmd$ when they are holomorphic.
\end{notation}

Let us first give a reminder about flat connections on DM stacks.

\begin{proposition}\label{propositionlocalsystemsconnections}
  Let $X$ be a smooth Deligne-Mumford stack. Then we have equivalences between the $\C$-linear symmetric tensor categories
  \footnote{i.e. with a tensor product and duals.} of:
  \begin{description}
    \item[(1)] local systems $\uE$ of finite dimensional $\C$-vector spaces on $X$;
    \item[(2)] $\Cinf$ vector bundles $E$ on $X$ with flat $\Cinf$ connections $\nabla$;
    \item[(3)] holomorphic vector bundles $\cE$ on $X$ with flat holomorphic connections $\rmd$.
  \end{description}
  The equivalence between \textbf{(2)} and \textbf{(1)} is induced by the functor sending $(E,\nabla)$ to its sheaf of flat sections $\uE_\nabla$.
  The equivalence between \textbf{(3)} and \textbf{(2)}
  is induced by the functor sending $(\cE,\rmd)$ to $(E,\overline{\partial}_E+\rmd)$,
  where $E$ is the $\Cinf$ bundle associated to $\cE$ and $\overline{\partial}_E:E\ra E\otimes\Omega_X^{0,1}$ is the Dolbeault operator associated
  to the complex structure on $\cE$.
\end{proposition}

\begin{remark}\label{remarklocalsystemsconnections}
  These equivalences are compatible with pullbacks along maps of DM stacks in the following sense.
  If $\DMC$ denotes the $2$-category of Deligne-Mumford stacks,
  the equivalences of \Cref{propositionlocalsystemsconnections} are actually equivalences of (weak) functors:
  \begin{equation*}
    (\DMC)^\mathrm{op}\lra \{\C-\text{linear symmetric tensor categories}\}.
  \end{equation*}

  From this we see that in the geometric definition of a modular functor (\ref{definitiongeometricmodularfunctor}),
  when $R=\C$, we can replace local systems by $\Cinf$ vector bundles with flat connections, or even holomorphic vector bundles with flat connections.
\end{remark}

\begin{definition}[Higgs bundle]\label{definitionhiggsbundle}
  Let $X$ be a proper smooth Deligne-Mumford stack.
  A Higgs bundle on $X$ is a holomorphic bundle $\cE$ together with an $\mathcal{O}_X$-linear map of sheaves $\theta:\cE\ra \cE\otimes\Omega_X^{1,0}$
  such that $\theta\wedge\theta=0$ in $\End{\cE}\otimes\Omega_X^{2,0}$. 
  Such a map $\theta$ is called a Higgs field.
\end{definition}

\begin{remark}
  Equivalently, a Higgs bundle is a $\Cinf$ bundle $E$ together with a flat $\dol$-connection $\dol_E:E\ra E\otimes\Omega_X^{0,1}$
  (i.e. a complex structure) and a $\Cinf_X$-linear map of sheaves $\theta:E\ra E\otimes\Omega_X^{1,0}$ such that $(\dol_E+\theta)^2=0$.
\end{remark}

For a holomorphic bundle $\cE$ one a curve, its slope is $\mu(\cE)=\frac{\mathrm{deg}\:\cE}{\mathrm{rk}\:\cE}$.

\begin{definition}
  A Higgs bundle $(\cE,\theta)$ on a smooth projective curve $Y$ is stable if every strict subsheaf $0\subsetneq \cF\subsetneq \cE$ preserved by $\theta$
  has strictly smaller slope, i.e. $\mu(\cF)<\mu(\cE)$.
  A Higgs bundle $(\cE,\theta)$ on a smooth projective curve $Y$ is polystable if it is a direct sum of stable Higgs bundles with the same slope.
\end{definition}

\begin{definition}[{\cite[9.6]{simpsonLocalSystemsProper2011}}]\label{definitionharmonictype}
  A Higgs bundle $(\cE,\theta)$ on a proper smooth DM stack $X$ is of harmonic type if:
  \begin{itemize}
    \item for every map $f:Y\ra X$ with $Y$ a smooth projective curve,
    the Higgs bundle $(f^*\cE,f^*\theta)$ on $Y$ is polystable;
    \item the rational Chern classes of $\cE$ vanish.
  \end{itemize}
\end{definition}

\begin{theorem}[{\cite[9.7]{simpsonLocalSystemsProper2011}}]\label{theoremsimpsoncorrespondance}
  Let $X$ be a proper smooth DM stack. The Simpson correspondence is an isomorphism between the $\C$-linear symmetric tensor categories of:
  \begin{description}
    \item[(1)] $\Cinf$ vector bundles $E$ on $X$ with semisimple flat $\Cinf$ connections $\nabla$;
    \item[(2)] Higgs bundles $(E,\dol_E,\theta)$ of harmonic type over $X$.
  \end{description}
  This isomorphism of categories is the identity on the $\Cinf$ bundle $E$ and on morphisms.
  Moreover, the correspondence is functorial with respect to maps of DM stacks, in the sense that it induces a strict isomorphism of (weak) functors:
  \begin{equation*}
    (\DMCp)^\mathrm{op}\lra \{\C-\text{linear symmetric tensor categories}\}
  \end{equation*}
  where $\DMCp$ is the $2$-category of proper smooth DM stacks.
\end{theorem}

\begin{remark}
  That the correspondence is the identity on morphisms means the following.
  A morphism $f:(E,\nabla)\ra(F,\nabla')$ is just a morphism $f:E\ra F$ which commutes to the connections.
  Similarly, a morphisms $g:(E,\overline{\partial},\theta)\ra (F,\overline{\partial}',\theta')$
  is just a morphism $g:E\ra F$ which commutes to the complex structures and Higgs fields.
  The claim on morphisms is that the correspondence then sends $f$ to $f$ and $g$ to $g$. This is a consequence of \Cref{lemmaharmonicforms} applied to $E^\vee\otimes F$.

  At the level of functors with source $(\DMCp)^\mathrm{op}$, this means that the isomorphism between the functor of semisimple flat connections
  and the functor of Higgs bundles of harmonic type is the identity over the functor of $\Cinf$ vector bundles.
\end{remark}

Let us briefly take the time to recap how in \Cref{theoremsimpsoncorrespondance},
one associates Higgs bundles of harmonic type to semisimple flat bundles and vice versa.
Let $(E,\nabla)$ be a flat bundle on $X$ proper smooth.
There is a unique decomposition $\nabla=d'+d''$ into $(1,0)$ and $(0,1)$ parts.
Choose a $\Cinf$ Hermitian metric $K$ on the underlying bundle $E$. Then there exists a unique $(0,1)$ connection $\delta''$
such that $d'+\delta''$ preserves $K$. Similarly, there exists a unique $(1,0)$ connection $\delta'$
such that $\delta'+d''$ preserves $K$. Then set:
\begin{equation*}
  \dol_E:=\frac{d''+\delta''}{2}\text{ and }\theta:=\frac{d'-\delta'}{2}.
\end{equation*}
The condition for $(E,\overline{\partial}_E,\theta)$ to be a Higgs bundle is $(\overline{\partial}_E+\theta)^2=0$.
We call the metric $K$ pluri-harmonic if this condition is satisfied. One way in the correspondence is then implied by the following result.
\begin{theorem}[{\cite[Th. 1]{simpsonHiggsBundlesLocal1992}}]
  The flat bundle $(E,\nabla)$ admits a pluri-harmonic metric if and only if $\nabla$ is semisimple.
  If moreover $\nabla$ is simple, the pluri-harmonic metric is unique up to scalar.
\end{theorem}
If $(E,\nabla)$ is simple, by uniqueness of $K$ pluri-harmonic and the fact that the above procedure does not depend on the scalar normalization of $K$,
we have a unique pair $(\overline{\partial}_E,\theta)$. This is also true in the semisimple case, and is a consequence of the following result,
that Simpson proves using Kähler identities.

\begin{lemma}[{\cite[Lem. 1.2]{simpsonHiggsBundlesLocal1992}}]\label{lemmaharmonicforms}
  Let $(E,\nabla)$ be a semisimple flat bundle, $K$ a pluri-harmonic metric and $(E,\dol_E,\theta)$ the associated Higgs bundle.
  Then for any smooth global section $e$ of $E$, $\nabla e=0$ if and only if $(\dol_E+\theta)e=0$.
\end{lemma}

Indeed, if $K_1$ and $K_2$ are $2$ pluri-harmonic metrics on $(E,\nabla)$ with respective associated Higgs bundles $(E,\dol_E^1,\theta^1)$ and $(E,\dol_E^2,\theta^2)$
then $K_1^\vee\otimes K_2$ is pluri-harmonic on $E^\vee\otimes E=End(E)$ and has associated Higgs bundle
$(End(E),\dol_{End(E)},\theta_{End(E)})$. As the correspondence is compatible with duals and tensor products, for any $\phi$ global
section of $End(E)$ and $e$ global section of $E$, we have:
$$(\dol_{End(E)}+\theta_{End(E)})(\varphi)(e)+\varphi((\dol_E^1+\theta^1)e)=(\dol_E^2+\theta^2)\varphi(e).$$
Applying this for $\varphi=\id_E$, by the \Cref{lemmaharmonicforms} the first term vanishes and we get $(\dol_E^1+\theta^1)e=(\dol_E^2+\theta^2)e$.
Hence $\dol_E^1=\dol_E^2$ and $\theta^1=\theta^2$.

As for the other way in the correspondence, notice that given a Higgs bundle $(E,\dol_E,\theta)$
and a $\Cinf$ Hermitian metric $K$ on $E$,
we can reverse the above procedure to obtain a connection $\nabla=d'+d''$.
Assume $E$ has vanishing rational Chern classes. We say that $K$ is Hermitian-Yang-Mills when $\nabla^2=0$.
The other way in the correspondence is then implied by the following.
\begin{theorem}[{\cite[Th. 1]{simpsonHiggsBundlesLocal1992}}]
  Assume the rational Chern classes of $E$ vanish. Then the Higgs bundle $(E,\dol_E,\theta)$
  admits a Hermitian-Yang-Mills metric if and only if it is of harmonic type.
  If it is stable, the Hermitian-Yang-Mills metric is unique up to scalar.
\end{theorem}

Notice that this correspondence is really a correspondence between $\nabla$ and $(\overline{\partial}_E,\theta)$.
It is compatible with morphisms (\Cref{lemmaharmonicforms}).
Moreover the pluri-harmonic and Hermitian-Yang-Mills conditions are compatible with direct sums, tensor products, duals and pullbacks.
This last point implies the functoriality of the correspondence with respect to $X$.

One crucial fact we will use in the proof of \Cref{maintheoremcomplex} is that the correspondence is continuous.
We do not need the full statement of continuity but only the following special case.

\begin{proposition}[{\cite[2.9, proof of 4.5]{simpsonHiggsBundlesLocal1992}}]\label{propositioncontinuity}
  Let $(\cE,\theta)$ be a Higgs bundle of harmonic type on a smooth proper DM stack $X$, and $x\in X$ a base point.
  For any $t\in\C^*$, $(\cE,t\theta)$ corresponds to some semisimple connection $\nabla_t$ on $E$.
  Let $\rho_t$ be the monodromy representation of $\nabla_t$ at $x$. Then the map $\C^*\ra\hom(\pi_1X,\GL{E_x})$, $t\mapsto \rho_t$ is continuous.
\end{proposition}

\begin{definition}[Modular functor of Higgs bundles]
  Let $\Lambda$ be a set of colors and $r,s\geq 1$ integers.
  Then the definition of a modular functor of Higgs bundles with level $(r,s)$
  is obtained from that of a (non-anomalous) complex modular
  functor (\ref{definitiongeometricnonanomalousmodularfunctor}, \ref{definitiongeometricmodularfunctor})
  by replacing the $\C$-local systems $\Nu_g(\ul)$ by Higgs bundles $(\Nu_g(\ul),\theta_g(\ul))$.
  
  Similarly for genus $0$ modular functors, ribbon and braided functors.
\end{definition}

From the correspondence, its functoriality and \Cref{theoremsemisimplicity} we immediately get the following.

\begin{corollary}[Non-Abelian Hodge correspondence for modular functors]\label{correspondencemodular}
  There is an equivalence between the categories of:
  \begin{description}
    \item[(1)] complex modular functors;
    \item[(2)] modular functors of Higgs bundles $(\Nu,\theta)$ all of whose bundles $(\Nu_g(\ul),\theta_g(\ul))$ are of harmonic type.
  \end{description}
  Similarly for genus $0$ modular functors, ribbon functors and braided functors.
\end{corollary}


\subsection{Systems of Hodge bundles}

This section is essentially paraphrasing Simpson \cite[4]{simpsonHiggsBundlesLocal1992}. See there for the proofs.
By CVHS with real bidegrees we shall mean that in the Hodge decomposition, we allow the weight $m$ and the bidegrees $(p,q)$ to be in $\R$,
and not just $\Q$.

\begin{definition}[system of Hodge bundles]
  Let $X$ be a smooth proper DM stack. A system of Hodge bundles on $X$ is a Higgs bundle $(\cE,\theta)$ together with a real number $m$ called weight
  and a holomorphic bundle decomposition\footnote{Here $p,q\in\R$.} $\cE=\bigoplus_{p+q=m}\cE^{p,q}$
  such that $\theta$ sends $\cE^{p,q}$ into $\cE^{p-1,q+1}\otimes\Omega_X^{1,0}$,
  and such that there exists a polarization $h$ for this decomposition, i.e. $h$ is a Hermitian form on the underlying $\Cinf$ bundle $E$ of $\cE$
  for which the decomposition is orthogonal, and $h$ is definite on $E^{p,q}$ with sign opposite to that on $E^{p-1,q+1}$.
\end{definition} 

\begin{proposition}[{\cite[4]{simpsonHiggsBundlesLocal1992}}]\label{propositionharmonicpolarization}
  Let $(\cE,\theta)$ be a Higgs bundle over a smooth DM stack $X$ with a decomposition $\cE=\bigoplus_{p+q=m}\cE^{p,q}$
  such that $\theta$ sends $\cE^{p,q}$ into $\cE^{p-1,q+1}\otimes\Omega_X^{1,0}$. Then this decomposition admits a polarization
  if and only if $(\cE,\theta)$ is of harmonic type.
\end{proposition}

Remember that the connection underlying a CVHS is always semisimple.

The Simpson correspondence induces a bijection between systems of Hodge bundles and complex variations of Hodge structures (with real bidegrees).
Let us describe this.

Given a CVHS on $(E,\nabla)$ with real bidegrees, Griffiths transversality tells us that $\nabla$ sends $E^{p,q}$ into:
\begin{equation*}
  E^{p+1,q-1}\otimes\Omega_X^{0,1}\oplus E^{p,q}\otimes\Omega_X^{1,0}\oplus E^{p,q}\otimes\Omega_X^{0,1}\oplus E^{p-1,q+1}\otimes\Omega_X^{1,0}.
\end{equation*}
Hence we can split $\nabla$ into $\overline{\theta}+\partial_E+\dol_E+\theta$ according to the decomposition above.
Then $(E,\dol_E,\theta)$ is the associated system of Hodge bundles, and any polarization for $(E,\nabla)$ is also one for $(E,\dol_E,\theta)$.
To see this, notice that for any polarization $h$, the metric $K=\sum_{p+q=m}\mathrm{sign}(h_{\mid E^{p,q}}) h_{\mid E^{p,q}}$ is pluri-harmonic.

Let $(\cE,\theta)$ be a system of Hodge bundles on $(\cE,\theta)$. Denote by $E$ the underlying $\Cinf$ bundle and by $\dol_E$ the Dolbeault operator.
Choose a polarization $h$. Now let $\partial_E$ be the unique $(1,0)$-connection such that $\partial_E+\dol_E$ preserves $h$,
and let $\overline{\theta}$ be the adjoint of $\theta$ for $h$.
Then $(E,\nabla=\overline{\theta}+\partial_E+\dol_E+\theta)$ is the associated CVHS (with real bidegrees) and $h$ polarizes it.
\textit{Moreover $\nabla$ is independent of the choice of the polarization $h$.} Again, to see this, notice
that the metric $K=\sum_{p+q=m}\mathrm{sign}(h_{\mid E^{p,q}}) h_{\mid E^{p,q}}$ is Hermitian-Yang-Mills.

The bijection described above is clearly functorial with respect to maps of DM stacks and compatible with tensor products and duals.

\begin{theorem}\label{theoremcvhshodgebundle}
  Let $X$ be a proper smooth DM stack. The non-Abelian Hodge correspondence described in \Cref{theoremsimpsoncorrespondance}
  puts bundles with flat connections over $X$ admitting a CVHS structure in bijection with Higgs bundles over $X$ admitting a structure of system of Hodge bundles.
  Moreover, it induces, as described above, an isomorphism between the $\C$-linear symmetric tensor categories of:
  \begin{description}
    \item[(1)] complex variations of Hodge structures over $X$ with real bidegrees;
    \item[(2)] systems of Hodges bundles over $X$.
  \end{description}
  This isomorphism of categories is functorial in the sense that it extends to an isomorphism of functors:
  \begin{equation*}
    (\DMCp)^\mathrm{op}\lra \{\C-\text{linear symmetric tensor categories}\}.
  \end{equation*}
\end{theorem}

The fact of interest to us is the following.

\begin{lemma}[{\cite[4.1]{simpsonHiggsBundlesLocal1992}}]
  A Higgs bundle of harmonic type $(\cE,\theta)$ admits a structure of system of Hodge bundles if and only if there exists $t\in\C^*$ not a root of unity
  such that $(\cE,\theta)$ is isomorphic to $(\cE,t\theta)$.
\end{lemma}

We shall use the following refinement of this Lemma. For $t\in \R_+^*\setminus\{1\}$, we denote by $\log_tx:=\ln(x)/\ln(t)$ the logarithm of $x$ in base $t$.

\begin{lemma}\label{actiontoCVHS}
  Let $(\cE,\theta)$ be a Higgs bundle of harmonic type on a smooth proper DM stack $X$, $m$ a real number and $t\in \R_+^*\setminus\{1\}$.
  Then there are maps of sets natural in $X$:
  \[\begin{tikzcd}
    {\{\text{isomorphisms }f:(\cE,\theta)\ra(\cE,t\theta)\}} \\
    {\{\text{structures of system of Hodge bundles on }(\cE,\theta)\text{ with weight }m\}}
    \arrow["r"', curve={height=6pt}, two heads, from=1-1, to=2-1]
    \arrow["i"', curve={height=6pt}, hook, from=2-1, to=1-1]
  \end{tikzcd}\]
  such that $r\circ i=\id$. The map $r$ sends $f$ to the decomposition $\cE=\bigoplus\cE^{p,q}$ where
  $\cE^{m-q,q}=\bigoplus_{\log_t|x|=q}\ker(f-x\cdot\id)^{\dim\cE}$
  is on each fiber the characteristic spaces of $f$ for the eigenvalues $x$ such that $\log_t|x|=q$.
  The map $i$ sends a decomposition $\cE=\bigoplus\cE^{p,q}$ to the maps $f:(\cE,\theta)\ra(\cE,t\theta)$
  such that $f_{\mid \cE^{p,q}}$ is multiplication by $t^q$.
\end{lemma}

\begin{proof}
  We need only check that $r$ and $i$ are well defined. That $i$ is well defined is clear since $f$
  is an isomorphism of Higgs bundles if and only if $f\theta=t\theta f$. Let us turn to $r$.
  Given $f$, its eigenvalues on the fibers vary holomorphically.
  As $X$ is proper, they are constant. So the spaces $\cE^{m-q,q}:=\bigoplus_{\log_t|x|=q}\ker(f-x\cdot\id)^{\dim\cE}$
  define a direct sum decomposition of $\cE$. The equality $(f-x\cdot\id)^{\dim\cE}\theta=\theta(tf-x\cdot\id)^{\dim\cE}$ shows that
  $\theta$ sends $\cE^{m-q,q}$ into $\cE^{m-q-1,q+1}$. As $(\cE,\theta)$ is of harmonic type, by \Cref{propositionharmonicpolarization}, 
  the decomposition is a system of Hodge bundles.
\end{proof}

\begin{corollary}\label{actiontoCVHSmodular}
  Let $\Nu$ be a complex modular functor. Let $\Lambda$ be its set of colors
  and $(\Nu,\theta)$ the associated modular functor of Higgs bundles (\ref{correspondencemodular}).
  Choose a weight function $w$ on $\Nu$ and some $t\in \R_+^*\setminus\{1\}$.
  Then the maps $i$ and $r$ of \Cref{actiontoCVHS} induce maps:
  \[\begin{tikzcd}
    {\{\text{isomorphisms }f:(\Nu,\theta)\ra(\Nu,t\theta)\}} \\
    {\{\text{Hodge structures on }\Nu\text{ with weight function }w\text{ and real bidegrees}\}}
    \arrow["r"', curve={height=6pt}, two heads, from=1-1, to=2-1]
    \arrow["i"', curve={height=6pt}, hook, from=2-1, to=1-1]
  \end{tikzcd}\]
  such that $r\circ i=\id$.
\end{corollary}

\begin{proof}
  We apply \Cref{actiontoCVHS} to each $(\Nu_g(\ul),\theta_g(\ul))$. By functoriality of $r$ and $i$ in the Lemma,
  we get similar maps $r$ and $i$ between the set of isomorphisms $f$ as above,
  and the set of structures of systems of Hodge bundles on $(\Nu,\theta)$ with weight function $w$ and real weights.
  Now by \Cref{theoremcvhshodgebundle}, systems of Hodge bundles on $(\Nu,\theta)$ are in bijection with CVHS on $\Nu$.
  This bijection preserves the weight function $w$.
\end{proof}


\section{Ocneanu rigidity}\label{sectionocneanu}


\subsection{Deformations of associators and endofunctors of ribbon fusion categories}\label{subocneanu}

\begin{definition}
  A fusion category over a ring $R$ is the data of a finite set $\Lambda$
  together with a rigid $R$-linear monoidal structure on $\VLR$.
  By rigid, we mean that each object in $\VLR$ has a right dual and a left dual (see \cite[2.10]{etingofTensorCategories2015}).
  Defining morphisms as in \Cref{definitionweaklyribbon}, we get a category $\Fus$ of fusion categories.
\end{definition}

Note that our definition of morphisms of fusion and ribbon fusion categories is more restrictive than the usual one.

\begin{notation}
  We will denote by $\Ceps:=\C[X]/X^2$ the ring of dual numbers ($X=\epsilon$).
  For any structure $\bullet$ over $\C$, we will use the subscript $\bullet_\Ceps$ to denote its trivial deformation,
  i.e. the structure over $\Ceps$ obtained by tensoring along $\C\ra\Ceps$.
\end{notation}

\begin{definition}
  Let $\Cat=(\VL,\otimes,c)$ be a fusion category over $\C$, where $c$ denotes its associator.
  A deformation of $c$ is an associator $\widetilde{c}$ on $(\VL^{\Ceps},\otimes_{\Ceps})$ that restricts to $c$ when taking the quotient by $\epsilon$.
  We assume that $\widetilde{c}$ is compatible with the trivial deformations of the left and right unitors in $\Cat$.

  Such a deformation is said to be trivial if there exists an (infinitesimal) isomorphism $(\id_\Lambda,(\id+\epsilon\phi):\otimes_\Ceps\simeq\otimes_\Ceps)$
  from $(\VL^{\Ceps},\otimes_{\Ceps},c_\Ceps)$ to $(\VL^{\Ceps},\otimes_{\Ceps},\widetilde{c})$.

  If $\Cat=(\VL,\otimes,c,\beta,\theta)$ is a ribbon fusion category over $\C$, a deformation of $c$ is an associator $\widetilde{c}$
  such that $(\VL^{\Ceps},\otimes_{\Ceps},\widetilde{c},\beta_\Ceps,\theta_\Ceps)$ is a braided category with a twist.
  Such a deformation is trivial if there exists a pair $(\id_\Lambda,(\id+\epsilon\phi))$ as above that is an isomorphism compatible with the braiding and twist.
  Similarly for braided fusion categories over $\C$.
\end{definition}

Note that any deformation $\widetilde{c}$ of $c$ is always of the form $c(\id+\epsilon d)$
for $d:(-\otimes-)\otimes-\ra (-\otimes-)\otimes-$ a natural transformation.

\begin{theorem}[{Ocneanu rigidity for associators (\cite[9.1.4]{etingofTensorCategories2015})}]\label{theoremocneanu}
  Let $\Cat=(\VL,\otimes,c)$ be a fusion category over $\C$. Then any deformation $\widetilde{c}$ of $c$ is trivial.
\end{theorem}

We now briefly explain how this result for fusion categories can be extended to ribbon fusion categories.
A quick proof of this is explained in \cite[Rmk. 2.33]{etingofFusionCategories2005}.
Here we give another proof, and take this as an opportunity to better understand the meaning of \Cref{theoremocneanu}.
Fix a ribbon fusion category $\Cat=(\VL,\otimes,c,\beta,\theta)$. Let $C^3(\Cat)$ be the vector space of
all natural transformations $d:(-\otimes-)\otimes-\ra (-\otimes-)\otimes-$,
let $Z^3(\Cat)$ be the subspace of the $d$ such that $c(\id+d)$ is an associator for $(\VL^{\Ceps},\otimes_{\Ceps})$,
and let $Z^3(\Cat,\beta)$ be the subspace of the $d$ such that $c(\id+d)$ is an associator compatible with the braiding $\beta_\Ceps$.

Now define $C^2(\Cat)$ to be the vector space of natural transformations $\phi:\otimes\ra\otimes$ such that $\phi(-\otimes [0])$
and $\phi([0]\otimes -)$ are the identity natural transformations.
Define $C^2(\Cat,\beta)$ as the subspace of $C^2(\Cat)$ whose elements are the $\phi$ such that $\beta\circ\phi=\phi\circ\beta$.

That the associator $c(\id+d)$ with $d\in C^3(\Cat)$ is trivial means that there exists $\phi\in C^2(\Cat)$ such that the following square commutes:
\[\begin{tikzcd}
	{(-\otimes-)\otimes-} & {-\otimes(-\otimes-)} \\
	{(-\otimes-)\otimes-} & {-\otimes(-\otimes-).}
	\arrow["{c(\id+\epsilon d)}", from=1-1, to=1-2]
	\arrow["{(\id+\epsilon\phi)\circ((\id+\epsilon\phi)\otimes\id)}"', from=1-1, to=2-1]
	\arrow["{(\id+\epsilon\phi)\circ(\id\otimes(\id+\epsilon\phi))}", from=1-2, to=2-2]
	\arrow["c", from=2-1, to=2-2]
\end{tikzcd}\]

As the inverse of $(\id+\epsilon\phi)$ is $(\id-\epsilon\phi)$, using the square above, we can linearly associate to each $\phi$ a $d$,
and we get linear maps:
\begin{equation*}
  \partial:C^2(\Cat)\lra Z^3(\Cat)\text{ and }\partial_\beta:C^2(\Cat,\beta)\lra Z^3(\Cat,\beta).
\end{equation*}
Ocneanu rigidity (\ref{theoremocneanu}) is equivalent to the statement that $\partial$ is surjective.
The map $\partial$ is a differential in the proper Davydov-Yetter complex of $\id_\Cat$. This complex is quasi-isomorphic
to the usual Davydov-Yetter complex (see \cite[(47) and B.3]{faitgAdjunctionTheoremDavydovYetter2024}).
Ocneanu rigidity is usually stated as the acyclicity of this latter complex (see for example \cite[7.22, 9.1.3]{etingofFusionCategories2005}).

We would like to show that this is also the case for $\partial_\beta$. For this, notice that compatibility of an associator $c$
to a braiding $\beta$ is the statement that the following diagram commutes:
\[\begin{tikzcd}
	{(-_1\otimes-_2)\otimes-_3} & {-_1\otimes(-_2\otimes-_3)} \\
	{(-_3\otimes-_2)\otimes-_1} & {-_3\otimes(-_2\otimes-_1).}
	\arrow["c", from=1-1, to=1-2]
	\arrow["{\beta\circ(\beta\otimes\id)}"', from=1-1, to=2-1]
	\arrow["{\beta\circ(\id\otimes\beta)}", from=1-2, to=2-2]
	\arrow["c", from=2-1, to=2-2]
\end{tikzcd}\]
Let $d$ be in $Z^3(\Cat,\beta)$. Then, by Ocneanu rigidity, there exists $\phi$ in $C^2(\Cat)$ with $\partial\phi=d$.
As $\partial\phi$ is in $Z^3(\Cat,\beta)$, using the $2$ diagrams above, we see that $\partial\phi=\partial(\beta\phi\beta^{-1})$,
and in particular $d=\partial\phi=\partial\frac{1}{2}(\phi+\beta\phi\beta^{-1})$. But $\frac{1}{2}(\phi+\beta\phi\beta^{-1})$ is in $C^2(\Cat,\beta)$.
Hence $\partial_\beta$ is surjective. Note that any element of $C^2(\Cat,\beta)$ is compatible with the twist $\theta$ in $\Cat$. Hence we have the following.

\begin{corollary}[{Ocneanu rigidity for associators in ribbon and braided fusion categories (\cite[2.33]{etingofTensorCategories2015})}]\label{theoremocneanuribbon}
  Let $\Cat=(\VL,\otimes,c,\beta,\theta)$ be a ribbon fusion category over $\C$. Then any deformation $\widetilde{c}$ of $c$ is trivial.
  Similarly for braided fusion categories over $\C$.
\end{corollary}

We now reproduce the standard algebro-geometric argument in \cite[9.1.4]{etingofTensorCategories2015}
to deduce triviality of continuous deformations of associators in ribbon fusion categories.
Let us fix a $\C$-linear category $\Cat=(\VL,\otimes,\beta,\theta)$ where $\otimes:\VL\times\VL\ra\VL$ is bilinear, $\beta:-_1\otimes-_2\ra-_2\otimes-_1$
and $\theta:\id\ra\id$ are linear. Then the data of an associator $c$ for $\otimes$ which makes $\Cat$ into a ribbon category is a family
of linear maps satisfying algebraic conditions. Hence such associators are parametrized
by an affine $\C$-scheme of finite type, that we denote $\Ass(\Cat)$.

Now $2$ such associators $c_1$ and $c_2$ may be gauge equivalent in the sense that there exists an automorphism
$f:\otimes\ra\otimes$ preserving $\beta$ such that $c_2=(f^{-1}\otimes\id)f^{-1} c_1 f (f\otimes\id)$.
Such automorphisms $f$ are parametrized by an affine complex algebraic
group $\Autb(\otimes)$ which acts algebraically on $\Ass(\Cat)$.

The Zariski tangent space of $\Autb(\otimes)$ at $\id$ is $C^2(\beta)=C^2(\Cat,\beta)$ above,
and the Zariski tangent space of $\Ass(\Cat)$ at $c$ is $Z^3(c,\beta)=Z^3(\Cat,\beta)$ above.

The differential of the action $\Autb(\otimes)\ra\Ass(\Cat),f\mapsto f\cdot c$ at $f=\id$ is then the map
$\partial_\beta:C^2(c,\beta)\lra Z^3(c,\beta)$ above.
By \Cref{theoremocneanuribbon} this map is surjective, so by \Cref{lemmaopen}, the $\Autb(\otimes)$-orbit of any $c$ is Zariski open.

\begin{lemma}\label{lemmaopen}
  Let $G$ be a connected complex algebraic group acting on a scheme $X$ of finite type over $\C$. If for $x\in X$,
  the differential of the map $G\ra X$, $g\mapsto g\cdot x$ is surjective, then set theoretic orbit $G\cdot x$
  is Zariski open in $X$.
\end{lemma}
\begin{proof}
  We may replace $G$ by its neutral connected component. Then as $G$ is smooth, it is irreducible, and hence $G\cdot x$ lies in an irreducible
  component of $X$. So we may also assume that $X$ is integral. By Chevalley's theorem, $G\cdot x$ is constructible. Hence there exists $y\in G\cdot x$
  such that $G\cdot x$ is locally closed near $y$. Because we are in characteristic $0$, we may further assume that $G\cdot x$ is smooth at $y$.
  As $G$ acts transitively, $G\cdot x$ is locally closed and smooth at every point. Now by smoothness,
  the local dimension of $G\cdot x$ equals that of its tangent spaces, which is maximal by the assumption. This concludes.
\end{proof}

From the discussion above, we get the following.

\begin{corollary}\label{continuityassociators}
  Let $\Cat=(\VL,\otimes,c,\beta,\theta)$ be a ribbon fusion category over $\C$.
  Then for any continuous family of associators $(c_t)_{t\in (-1,1)}$ with $c_0=c$,
  there exists a continuous family $(f_t)_{t\in(-1,1)}$ of elements in $\Autb(\otimes)$
  such that $f_0=\id$ and for all $t$, $c_t=f_t\cdot c$. The same result holds in the braided fusion case.
\end{corollary}
\begin{proof}
  As all the $\Autb(\otimes)$-orbits are open, they are also all closed. Hence the family $(c_t)_{t\in (-1,1)}$ lies in the orbit of $c$.
  Let $c\in\Ass(\Cat)(\C)$, $O_c$ its orbit and $G_c\subset \Autb(\otimes)$ the stabilizer.
  Then the associated reduced subscheme $O_c^\mathrm{red}$ is isomorphic to the quotient $\Autb(\otimes)/G_c$.
  Analytically, $\Autb(\otimes)/G_c$ is the quotient of a Lie group by a subgroup. So the action $\Autb(\otimes)\ra O_c^\mathrm{red}=\Autb(\otimes)/G_c$
  is a Serre fibration and we can lift the map $c_t:(-1,1)\ra O_c^\mathrm{red}$ to $\Autb(\otimes)$.
\end{proof}

Let us now turn to infinitesimal automorphisms of ribbon fusion categories.

\begin{definition}
  Let $\Cat=(\VL,\otimes,c,\beta,\theta)$ be a ribbon fusion category.
  An infinitesimal automorphism of $\Cat$ is a functor of ribbon categories $(\id_\Lambda,\id+\epsilon k):\Cat_\Ceps\ra\Cat_\Ceps$.
  Similarly for fusion categories and braided fusion categories.

  Such an isomorphism $(\id_\Lambda,\id+\epsilon k)$ is said to be trivial if there exists $\delta:\Lambda\ra\C$ with $\delta(0)=0$,
  such that for each $\lambda,\mu,\nu\in\Lambda$, $k$ is on $\hom([\lambda],[\mu]\otimes[\nu])$ multiplication
  by $\delta(\mu)+\delta(\nu)-\delta(\lambda)$.
\end{definition}

\begin{theorem}[{Ocneanu rigidity for endofunctors (\cite[9.1.5]{etingofTensorCategories2015})}]\label{theoremocneanufunctor}
  Any infinitesimal automorphism of a fusion category over $\C$ is trivial.
\end{theorem}

As any trivial infinitesimal fusion automorphism is also a braided/ribbon automorphism, we deduce the following.

\begin{corollary}\label{theoremocneanuribbonfunctor}
  Any infinitesimal automorphism of a braided (or ribbon) fusion category over $\C$ is trivial.
\end{corollary}

As an application, we prove that all weight functions on a modular functor are exact.

\begin{proof}[Proof of \Cref{exactweightfunctions}]
  Let $w$ be a weight function on $\Nu$. Then we can define an infinitesimal automorphism $(\id_\Lambda,\id+\epsilon k):\CNeps\ra\CNeps$
  by asking that for all $\lambda,\mu,\nu\in\Lambda$, $k$ is on $\hom([\lambda],[\mu]\otimes [\nu])$
  multiplication by $w_0(\lambda;\mu,\nu)$. Now,
  Ocneanu rigidity (\ref{theoremocneanuribbonfunctor}) implies that $(\id_\Lambda,\id+\epsilon k)$ is trivial
  and thus that there exists $f:\Lambda\ra \Q$ with $f(0)=0$ such that for all $\lambda,\mu,\nu\in\Lambda$,
  $w_0(\lambda;\mu,\nu)=f(\mu)+f(\nu)-f(\lambda)$. By the gluing axiom, for any $g$ and $\ul$, $w_g(\ul)=\frac{1}{2}\sum_if(\lambda_i)$.
\end{proof}

Now, by the same algebro-geometric argument as for deformations of associators (\ref{continuityassociators} and above),
we deduce from \Cref{theoremocneanuribbonfunctor}
that any continuous family of deformations of the identity braided/ribbon automorphism can be trivialized.

\begin{proposition}\label{continuityautomorphisms}
  Let $\Cat=(\VL,\otimes,c,\beta,\theta)$ be a ribbon fusion category, and $(f_t)_{t\in(-\epsilon,\epsilon)}$
  a continuous family of automorphisms of $\Cat$ with $f_0=\id$.
  Then there exists a continuous family of maps $(\delta_t:\Lambda\ra \C^*)_{t\in(-\epsilon,\epsilon)}$
  with $\delta(0,\cdot)\equiv 1$ such that for any $t$ and $\lambda,\mu,\nu\in\Lambda$,
  $f_t$ is on $\hom([\lambda],[\mu]\otimes[\nu])$ multiplication by $\delta_t(\mu)\delta_t(\nu)\delta_t(\lambda)^{-1}$.
  The same holds for braided fusion categories.
\end{proposition}


\subsection{Modular functors are defined over number fields}

In this subsection, we deduce from the well known fact that all complex ribbon fusion categories can be defined over number fields
that any complex modular functor $\Nu$ can be defined over a number field.

The following theorem is a corollary of Ocneanu rigidity for multi-fusion categories.
For a proof and a definition of multi-fusion category, see the reference.

\begin{theorem}[{\cite[9.1.8]{etingofTensorCategories2015}}]\label{theoremocneanunumberfield}
  Any multi-fusion category or functor between multi-fusion categories can be defined over a number field.
\end{theorem}

Let us sketch the proof for fusion categories. If $\Cat=(\VL,\otimes,c)$ is a fusion category, notice that the affine algebraic manifold
$\Ass(\Cat)$ which parametrizes all associators $c'$ for $(\VL,\otimes)$ is defined over $\Q$ (or even $\Z$).
As in the discussion above \Cref{continuityassociators}, by Ocneanu rigidity (\ref{theoremocneanu}), $\Ass(\Cat)$ is over $\C$
a finite union of open orbits for the action of an algebraic group $\Aut{\otimes}$ preserving isomorphism classes.
The orbit of $c$ admits a $K$-point $c'$ for some number field $K$, and then $\Cat'=(\VL,\otimes,c')$ is isomorphic to $\Cat$ and defined over $K$.
The proof of \Cref{theoremocneanunumberfield} is similarly deduced from Ocneanu rigidity results for multi-fusion categories.

\begin{corollary}\label{ribbonnumberfield}
  Any braided or ribbon fusion category over $\C$ is defined on a number field.
\end{corollary}

\begin{proof}[Sketch of proof]
  Let $\Cat=(\VL,\otimes,c,\beta)$ be a braided fusion category over $\C$. By the theorem, the underlying fusion category is defined over a number field $K$.
  So $c$ is defined over $K$. Now, the data of the braiding $\beta$ is equivalent to the data of a certain functor of multi-fusion categories
  $\Phi:\Cat\times\Cat\ra\Cat$ called a multiplication (see Yetter's article \cite[2.18,2.19]{yetterBraidedDeformationsMonoidal1997} for a definition and details).
  By the theorem, $\Phi$ is defined over a finite extension $K'$ of $K$. Hence so is $\beta$.

  If $\Cat$ is ribbon for a twist $\theta$, the data of such a twist
  is the same as that of the numbers $t_\lambda\in\C^*$.
  By the Anderson-Moore-Vafa theorem, they are always roots of unity, and thus defined on a finite extension $K''$ of $K'$
  (see, for example \cite[8.18.2]{etingofTensorCategories2015}). In conclusion, $\Cat$ is defined over the number field $K''$.
\end{proof}

\begin{corollary}\label{functornumberfield}
  Any modular, genus $0$ modular, ribbon or braided functor $\Nu$ over $\C$ is defined over a number field.
\end{corollary}

\begin{proof}
  By \Cref{ribbonnumberfield}, $\CN$ can be defined over a number field $K$.
  Let us translate this fact in the language of descent data (for vector spaces).
  Consider the $2$ inclusions $i_1,i_2:\C\ra\C\otimes_K\C$. For a vector space $W$ over $\C$, the data of a $K$-lattice
  is equivalent to the one of an isomorphism of $\C\otimes_K\C$-modules $f:i_1^*W\simeq i_2^*W$ satisfying the cocycle conditions over $\C\otimes_K\C\otimes_K\C$
  (the $K$-lattice is then simply the subset $\{w\in W\mid i_1^*(w)=f(i_2^*(w))\}$).
  
  In our context, this means that there exists an isomorphism $\varphi:i_1^*\CN\simeq i_2^*\CN$ of ribbon categories over $\C\otimes_K\C$
  satisfying the cocycle conditions. If $\Nu$ is modular, its central charge $c\in\C^*$ is a root of unity (\ref{propositionlevel}).
  Hence we may assume that $c$ is in $K$.
  By \Cref{theoremfullfaithfulness}, $\varphi$ extends to an isomorphism $\widetilde{\varphi}:i_1^*\Nu\simeq i_2^*\Nu$
  of modular, genus $0$ modular, ribbon or braided functors over $\C\otimes_K\C$
  satisfying the cocycle conditions. By descent, this means that $\Nu$ has a $K$-lattice.
\end{proof}


\section{Proof of existence and uniqueness of Hodge structures on modular functors}\label{sectionproof}
\subsection{Proof of \texorpdfstring{\Cref{maintheoremcomplex}}{existence and uniqueness of CVHS}}
Let $\Nu$ be a complex modular, genus $0$ modular, ribbon or braided functor.
Then by the non-Abelian Hodge correspondence for complex modular functors (\ref{correspondencemodular}),
we have a functor of Higgs bundles that we denote $(\Nu,\theta)$.

For $t\in\R_+^*$, $(\Nu,t\theta)$ is also a functor of Higgs bundles and by the correspondence we have a family
of functors $(\Nu_t)_{t\in \R_+^*}$. By \Cref{propositioncontinuity}, this family, seen as a family of gluing and vacuum
isomorphisms, is continuous (for the product topology).

Notice that $\Mgrb{0}{2}{r}\simeq \Br{r}$ (see our convention \ref{conventionzerotwo}) and $\Mgrb{0}{3}{r}\simeq \mathrm{B}\mu_r^3$,
so that $\theta=0$ on these spaces. Hence for any $t$, since $(\Nu,\theta)$ and $(\Nu,t\theta)$ coincide on these spaces
so do $\Nu$ and $\Nu_t$. In particular, the associated ribbon or braided fusion categories $\CN$ and $\CNt$ have the same braiding (and twist),
and may differ only in their associator, that we denote $c$ and $c_t$ respectively.

By \Cref{continuityassociators}, $(c_t)_{t\in(1-\epsilon,1+\epsilon)}$ is a trivial deformation of $c=c_1$ and there exists a continuous family of isomorphisms
$f_t:\CN\ra\CNt$ with $f_1=\id$.
By fullness in \Cref{theoremfullfaithfulness}, these isomorphisms lift to isomorphisms $\Nu\ra\Nu_t$.
Then, by the non-Abelian Hodge correspondence for complex modular/ribbon/braided functors (\ref{correspondencemodular}),
these isomorphisms induce isomorphisms $f_t:(\Nu,\theta)\ra (\Nu,t\theta)$.

Applying \Cref{actiontoCVHSmodular}, we have existence of a CVHS with real bidegrees on $\Nu$ and the weight function $w$ of our choice.
To obtain rational bidegrees, choose a map $\pi:\R\ra\Q$ of $\Q$-vector spaces that sends $1$ to $1$,
and replace every bidegree $(p,q)\in\R^2$ by $(\pi(p),\pi(q))\in\Q^2$.

Let us now turn to uniqueness up to shifts. Let $\Nu^1$ and $\Nu^2$ be $2$ Hodge structures on $\Nu$.
Using the map $i$ of \Cref{actiontoCVHSmodular}, for each $t\in \R_+^*$ they induce isomorphisms $f^1_t$ and $f^2_t$
from $(\Nu,\theta)$ to $(\Nu,t\theta)$ (for $t=1$, $f^1_1=f^2_1=\id$). Note that the map $i$ of \Cref{actiontoCVHSmodular} is continuous in $t$,
and hence so are the families $(f^1_t)$ and $(f^2_t)$.
In particular, the maps $(f^2_t)^{-1}\circ (f^1_t)$ form a continuous family of automorphisms of $(\Nu,\theta)$.

By the correspondence (\ref{correspondencemodular}) they induce automorphisms
$k_t$ of $\Nu$ and thus also of the ribbon or braided fusion category $\CN$.
By \Cref{theoremocneanuribbonfunctor}, there exists a family $\delta_t:\Lambda\ra \C^*$ with $\delta_t(0)=0$
such that $k_t$ is on $\hom([\lambda],[\mu]\otimes[\nu])$ multiplication by $\delta_t(\mu)\delta_t(\nu)\delta_t(\lambda)^{-1}$.
Thus $k_t$ is on $\Nu_g(\ul)$ multiplication by $\prod_i\delta_t(\lambda_i)$.

Now fix $t\in\R_+^*\setminus\{1\}$. As $r\circ i=\id$ in \Cref{actiontoCVHSmodular}, we see that $\Nu^1=\log_t|\delta_t|\cdot \Nu^2$.
Now, notice that since $\Nu^1$ and $\Nu^2$ have rational bidegrees, we also have $\Nu^1=(\pi\circ\log_t|\delta_t|)\cdot \Nu^2$.
As $\pi\circ\log_t|\delta_t|$ is in $\Shift$, this concludes.
\mbox{}\hfill $\square$

\subsection{Proof of \texorpdfstring{\Cref{rationalqreps}}{existence of RVHS}}

The proof essentially follows from that of \cite[Theorem 5]{simpsonHiggsBundlesLocal1992}. See there for details.
Let $\Nu$ be a complex modular, genus $0$ modular, ribbon or braided functor. Denote by $\Lambda$ its set of colors.
Fix $g,n\geq 0$ tangent stable and $\ul\in\Lambda^n$. Denote by $\rho_g(\ul)$ the monodromy representation of the local system $\Nu_g(\ul)$,
and by $\Gamma$ the fundamental group of $\Mgrbt{g}{n}{r}{s}$, $\Mgrb{0}{n}{r}$ or $\Mrpo{n}{r}$, depending on the type of functor considered.

As $\Nu$ can be defined over a number field $K$, so can $\rho_g(\ul)$. Replacing $K$ by some finite extension, we may assume that
all irreducible factors of $\rho_g(\ul)$ are defined over $K$ (note that $\rho_g(\ul)$ is semisimple by assumption).
Denote by $L\subset K$ the subfield generated by the traces $\tr{\rho(\gamma)}$ for $\gamma\in\Gamma$ and $\rho$ an irreducible factor of $\rho_g(\ul)$.
Let us first show that $L$ is a totally real or CM field.

Choose $\sigma:L\ra\C$ an embedding. As $\sigma(\Nu_g(\ul))$ supports a CVHS, so does each of its irreducible factors.
Hence such a factor $\sigma(\rho)$ preserves a Hermitian form, and in particular $\overline{\sigma}(\rho)$ is isomorphic to $\sigma(\rho^*)$.
So that for all $\rho$ and $\gamma$, $\sigma^{-1}\overline{\sigma}\tr{\rho(\gamma)}=\tr{\rho^*{\gamma}}$ is independent of $\sigma$.
Denote by $C=\sigma^{-1}\overline{\sigma}$ this involution on $L$. If $C=\id$, $L$ is totally real, otherwise it is CM.

We now apply the following lemma, which Simpson attributes to M. Larsen.

\begin{lemma}[{\cite[4.8]{simpsonHiggsBundlesLocal1992}}]
  Let $\rho:\Gamma\ra\GLn{n}{K}$ be an irreducible representation for $K$ a number field such that its trace field $L\subset K$ is CM (or totally real).
  Define $F$ to be $L\cap \R$.
  Then there exists a totally real extension $F'$ of $F$ such that $\rho$ can be defined over $F'\cdot L$
  (or \textbf{any} purely imaginary extension thereof if $L$ is totally real).
\end{lemma}

We can choose for each $\rho$ irreducible factor of $\rho_g(\ul)$ an extension $F'_\rho$ as in the Lemma. Denote by $F'$ there amalgamation.
Then all irreducible factors of $\rho_g(\ul)$ can be defined over $F'\cdot L$ (or any purely imaginary extension thereof if $L$ is totally real).
In conclusion, $\rho_g(\ul)$ and all its irreducible factors are defined over a CM field $L'$.

Let $C$ be the conjugation in $L'$. Choose an irreducible factor $\rho$ of $\rho_g(\ul)$. Then $C\rho\simeq\rho^*$,
so that $\rho$ preserves a Hermitian form $h_\rho$ defined over $L'$. For each pair of complex embeddings $\{\sigma,C\sigma\}$ of $L'$,
choose a CVHS on the local system $\uE_{\sigma(\rho)}$ associated to $\sigma\rho$ and choose the conjugate CVHS on $\uE_{C\sigma(\rho)}$.
We may choose these CVHS to have a fixed weight $m\in 2\Z$ independent of $\rho$ and $\sigma$, and have all Hodge bidegrees in $\Z^2$.
Then these CVHS together with $h_\rho$ form a RVHS on the local system $\uE_\rho$ associated to $\rho$.
By irreducibility of $\rho$, for each $\sigma$, $h_\rho$ has either sign $(-1)^p$ on each $\uE_{\sigma(\rho)}^{p,q}$
or sign $-(-1)^p$ on each $\uE_{\sigma(\rho)}^{p,q}$. As $F'\setminus\{0\}\ra (\R^*)^{\{\tau:F'\hookrightarrow \R\}}$ attains every connected component of the target,
we may multiply $h_\rho$ by some element in $F'\setminus\{0\}$ to ensure that for each $\sigma$ and $p+q=m$, $h_\rho$ has sign $(-1)^p$ on $\uE_{\sigma(\rho)}^{p,q}$.
After this modification, $h_\rho$ provides a strong polarization on $\uE_{\rho}$.
We may sum these RVHS over $\rho$ to obtain a strongly polarized RVHS on $\Nu_g(\ul)$ over $L'$.
\mbox{}\hfill $\square$

\bibliographystyle{plain}
\bibliography{biblio}

\end{document}